\documentclass[leqno,12pt]{article}

\usepackage[usenames, dvipsnames]{xcolor}
\usepackage{pgf,tikz} % Pour les figures avec tikz\usetikzlibrary{shapes,decorations} %pour rajouter des figures Ã  tikz : diamond, etc...
\usetikzlibrary{shapes,arrows,chains}
\usepackage{setspace} 
\usepackage{mathpazo}
\usepackage{amssymb,hyperref,amsthm}
\usepackage{euscript}
\usepackage{flafter}
\usepackage{pstricks}
\usepackage[latin1]{inputenc}
\usepackage{amsmath}
\usepackage{amsfonts}
\usepackage{pstricks-add}
\usepackage{yfonts}
\usepackage{egothic}
\usepackage{dsfont}
\usepackage{bm}
\usepackage{variations}
\usepackage{mathtools}

\hypersetup{
    linktoc=page,
   linkcolor=red,          % color of internal links
  citecolor=blue,        % color of links to bibliography
    filecolor=blue,      % color of file links
   urlcolor=cyan,
    colorlinks=true           % color of external links
}

\usepackage[refpage]{nomencl}
\usepackage{makeidx}
\makeindex

\usepackage{mathrsfs} %%% Nice caligraphic fonts
\evensidemargin\oddsidemargin
\newcommand{\eqnsection}{
\renewcommand{\theequation}{\thesection.\arabic{equation}}
   \makeatletter
   \csname  @addtoreset\endcsname{equation}{section}
   \makeatother}
\eqnsection
\def\tM{\mathtt{M}}
\def\M{{\mathcal M}}
\def\R{{\mathbb R}}
\def\e{{\mathbb E}}
\def\p{{\mathbb P}}
\def\P{{\bf P}}
\def\E{{\bf E}}
\def\Q{{\bf Q}}
\def\N{{\mathbb N}}
\def\T{{\mathbb T}}
\def\L{{\mathcal L}}
\def\B{{\mathcal B}}
\def\C{{\mathcal C}}
\def\en{\mathcal{E}}
 \def\pa{\overleftarrow}
\def\nf{\mathcal{F}}

\def\mS{{\underline{S}}}
\def\MS{{\overline{S}}}

\def\MV{{\overline{V}}}
\newtheorem{thm}{Theorem}[section]
\newtheorem{prop}[thm]{Proposition}
\newtheorem{lem}[thm]{Lemma}

\newtheorem{rem}[thm]{Remark}

\newtheorem{condition}[thm]{Assumption}
\newcommand{\ind}[1]{\mathbf{1}_{\left\{ #1 \right\}}}
\renewcommand{\epsilon}{\varepsilon}

\def\cb{{\bf c}}

\newcommand{\egloi}{\stackrel{\textrm{(d)}}{=}}

\newcommand{\ty}{{\beta}} % Type dans la marche alÃ©atoires
\newcommand{\tauh}{{\widehat{\tau}}} % Type dans la marche alÃ©atoires
\newcommand{\fl}[1]{{\lfloor #1 \rfloor}} % partie entiere infÃ©rieure
\usepackage{stmaryrd} % Pour utiliser les doubles crochets des intervalles d'entiers \llbracket \rrbracket
\title{\bf Maximal local time of randomly biased random walks on a Galton-Watson tree%Tail behaviours of branching random walks, of associated multitype Galton-Watson tree and of branching process with immigration in random environment
}

\author{Xinxin Chen\footnote{Institut Camille Jordan-C. N. R. S. UMR 5208-Universit\'e Claude Bernard Lyon 1(France).
\newline\vspace{0.1cm}\hspace{0.2cm} $\dag$Institut Denis Poisson-C. N. R. S. UMR 7013-Université d'Orléans (France).
\newline\vspace{0.1cm} MSC 2000 60J80; 60G50; 60K37. Supported by ANR MALIN} $\ $, Lo\"ic de Raph\'elis$^\dag$ }

\begin{document}

\maketitle
%\listoftodos
\begin{abstract}
We consider a recurrent random walk on a rooted tree in random environment given by a branching random walk. Up to the first return to the root, its edge local times form a Multi-type Galton-Watson tree with countably infinitely many types. When the walk is the diffusive or sub-diffusive, by studying the maximal type of this Galton-Watson tree, we establish the asymptotic behaviour of the largest local times of this walk  during $n$ excursions, under the annealed law.

%The environment is given by a branching random walk for which the tail distribution of $W_\infty$, the limit of its additive martingale has been given in \cite{Liu}. In this work, we study the joint tail distribution of $W_\infty$ and the minimum of the branching random walk, and then use it to obtain the tail behaviour of the largest local times of the associated random walk in random environment during one excursion. We also establish the asymptotic behaviour of the largest local times during $n$ excursions. %Our arguments also show the tail distribution of the maximum of a branching process in random environment with immigration.

\noindent\textit{Keywords: Branching random walk, Randomly biased random walk, Multi-type Galton-Watson tree}%, Branching process in random environment with immigration}
\end{abstract}

\section{Introduction: Models and results}
\label{Intro}
\subsection{Branching random walk and randomly biased random walk}
Let us first introduce a branching random walk on the real line, whose reproduction law is given by the law of $\C$, a point process on $\R$. The construction is as follows.

We start with one vertex at time $0$, which is called the root, denoted by $\rho$ and positioned at $V(\rho)=0$. At time $1$, the root gives birth to some children whose positions constitute a point process distributed as $\C$. These children form the first generation. Recursively, for any $n\geq0$, at time $n+1$, every vertex $u$ of the $n$-th generation produces its children independently of the other vertices so that the displacements of its children with respect to its position are distributed as $\C$. All children of the vertices of the $n$-th generation form the $(n+1)$-th generation. 

We hence get the genealogical tree $\T$. For any vertex $u\in\T$, let $V(u)$ denote its position and $|u|$ denote its generation with $|\rho|=0$. For two vertices $u,v\in\T$, write $u\leq v$ if $u$ is an ancestor of $v$ and write $u<v$ if $u\leq v$ but $u\neq v$. %In particular, we denote by $\overleftarrow{u}$ the parent of $u$. 
Denote by $\P$ the law of the branching random walk $\en=(\T, (V(u))_{u\in\T})$, which serves as the environment.

Given $\en=(\T, (V(u))_{u\in\T})$, the randomly biased random walk $(X_n)_{n\geq0}$ is a nearest-neighbour random walk on $\T$ started from $X_0=\rho$, and with transition probabilities as follows: ${\forall x\in\T\setminus\{\rho\}}$,
\begin{equation}\label{RW}
\p^\en(X_{n+1}=v\vert X_n=u)=
\begin{cases}
\frac{e^{-V(u)}}{e^{-V(u)}+\sum_{z: \overleftarrow{z}=u}e^{-V(z)}}\textrm{ if } \pa{u}=v\\
\\
\frac{e^{-V(v)}}{e^{-V(u)}+\sum_{z: \overleftarrow{z}=u}e^{-V(z)}}\textrm{ if } \pa{v}=u
\end{cases}
\end{equation}
where $\pa{u}$ represents the parent of $u\in\T\setminus\{\rho\}$. To define the transition probabilities for $x=\rho$ in a proper way, we add artificially a parent $\overleftarrow{\rho}$ to the root $\rho$ and suppose that \eqref{RW} holds for $u=\rho$ and that $\p^\en(X_{n+1}=\rho\vert X_n=\overleftarrow{\rho})=1$.  The \textit{quenched} law of the random walk $(X_n)_{n\geq0}$ on $\T\cup\{\pa{\rho}\}$ is denoted by $\p^\en$. Its\textit{ annealed} law is denoted by $\p(\cdot):=\int \p^\en(\cdot)\P(d\en)$.

Note that the law of the environment $\en$ is characterised by the law of $\C$. Let us introduce the Laplace transform of $\C$ defined by
\[
\psi(t):=\E\left[\sum_{|u|=1}e^{-tV(u)}\right]=\E\left[\int e^{-tx}\C(dx)\right], \forall t\in\R.
\]
In this paper, we assume
\begin{condition}\label{cond1}
$\psi(0)>1$, $\E[\sum_{|u|=1}|V(u)|e^{-V(u)}]<\infty$ 
and
\begin{equation}\label{hyp1}
\psi(1):=\E\left[\sum_{|u|=1}e^{-V(u)}\right]=1, \ \psi'(1):=-\E\left[\sum_{|u|=1}V(u)e^{-V(u)}\right]<0.
\end{equation}
\end{condition}
Note that $\psi(0)>1$ means that the Galton-Watson tree $\T$ is supercritical. %Let $\P^*$ be the measure conditioned to survive.

Let us introduce the quantity 
\begin{equation*}
\kappa:=\inf\{t>1: \psi(1)\geq 1\}
\end{equation*}
with convention that $\inf\emptyset=\infty$.  We also require the following assumptions.
\begin{condition}\label{cond2}
Either there exists some $\kappa\in(1,\infty)$ with $\psi(\kappa)=1$; or $\kappa=\infty$ and $\psi(t)<1$ for all $t>1$.
\end{condition} 
\begin{condition}\label{cond3}
The support of $\C$ is non-lattice.
\end{condition}
\begin{condition}\label{cond4}
If $\kappa\in(1,\infty)$, there exists some $\delta>0$ such that 
\begin{equation}\label{hyp1+}
\psi(t)<\infty, \forall t\in(1-\delta,\kappa+\delta),\textrm{ and } \E\left[\left(\sum_{|u|=1}e^{-V(u)}\right)^{\kappa+\delta}\right]<\infty.
\end{equation}
If $\kappa=\infty$, there exists some $\delta>0$ such that
\begin{equation}
 \E\left[\left(\sum_{|u|=1}e^{-V(u)}\right)^{2+\delta}\right]<\infty.
\end{equation}
\end{condition}

In~\cite{LP}, a criterion for recurrence/transience is established. So, from \eqref{hyp1}, it is known that $\min_{t\in[0;1]}\psi(t)=\psi(1)=1$ and that the random walk is recurrent. More precisely, Faraud showed with some extra-conditions that if $\psi(1)=1$ and $\psi'(1)<0$ then the random walk $(X_n)_{n\geq 0}$ is null recurrent. 
%It is known from \cite{LP} that under the assumption \ref{cond1}, the random walk $(X_n)$ is null recurrent. 
Moreover, under Assumptions \ref{cond1} and \ref{cond2}, when $\kappa\in(1,2]$,  it has been proved in \cite{HS} and \cite{dR16} that the random walk is sub-diffusive; and when $\kappa>2$, the walk is diffusive and satisfies an invariance principle, see \cite{AdR15} and \cite{Faraud}. More precisely, A\"id\'ekon and de Raph\'elis also proved in \cite{AdR15} that for $\kappa>2$, the tree visited by the walk up to time $n$, after being rescaled by a factor $n^{1/2}$, converges in law to the Brownian forest. Then in \cite{dR16} similar result is obtained for $\kappa\in(1,2]$, but in a stable regime. %with a scaling factor $n^{1-1/\kappa}$ and a convergence towards the Levy forest. 
Next, Chen, de Raph\'elis and Hu studied the localisation of the most visited sites in \cite{CdRH16}. 

In this paper, we are interested in the most visited edges and want to know how many times the walk crosses them up to some fixed or random time.

Let us introduce the edge local times $\overline{L}_n(u)$, $n\geq1$, $u\in\T$ defined by
\begin{equation}\label{localtime}
\overline{L}_n(u):=\sum_{k=1}^n\ind{X_{k-1}=\pa{u}, X_k=u}.
\end{equation}

Define a sequence of stopping times $(\tau_n)_{n\geq0}$ by
\[
\tau_n:=\inf\{k>\tau_{n-1}: X_{k-1}=\pa{\rho}, X_k=\rho\}, \forall n\geq1
\]
with $\tau_0:=0$. Note that $\tau_n-1$ is the $n$-th return to $\pa{\rho}$ of the random walk and that $\tau_1<\infty$, $\p$-a.s., as the walk is recurrent. Usually, we call the walk up to $\tau_1$ the first excursion and the $n$-th excursion means the walk from $\tau_{n-1}$ to $\tau_n$.

Then observe that $(\overline{L}_{\tau_1}(u))_{u\in\T}$ is a multi-type Galton-Watson tree with root of type 1, under $\p$, according to Lemma 3.1 of \cite{AdR15}. Its detailed reproduction distributions will be given in \eqref{offspringMtype}. We are initially interested in the tail distribution of the maximal local time during the first excursion:
\[
\p\left(\max_{u\in\T}\overline{L}_{\tau_1}(u)\geq x\right),
\]
as $x\rightarrow\infty$. The order of this tail has been considered in Theorem 1.5 of \cite{CdRH16}. We obtain the precise tail and use it to study the asymptotic of $\max_{u\in\T}\overline{L}_{\tau_n}(u)$ under $\p$ and eventually under $\p^\en$.%One will see later that the law of $\max_{u\in\T}\overline{L}_{\tau_1}(u)$ satisfies a distribution equation studied in \cite{CM18}. We will prove that it has a Cauchy-type tail when $1<\kappa<2$.
\subsection{Main results}
Let us state the main results of this paper. For the branching random walk, let us define
\[
W_n:=\sum_{|u|=n}e^{-V(u)}, \forall n\geq0.
\]
Obviously, by \eqref{hyp1}, $(W_n)_{n\geq0}$ is a $\P$-martingale with respect to the filtration of sigma-fields $\{\nf^V_n:=\sigma((u, V(u)); |u|\leq n)\}_{n\geq0}$. It is usually called the additive martinale. It is immediate that $W_n$ converges $\P$-a.s. to a nonnegative limit $W_\infty$. Under assumptions \ref{cond1}, \ref{cond4}, it converges also in $L^1$ (see for instance \cite{Lyons}). 

Denote by $f(x)\sim g(x)$ as $x\rightarrow x_0$ if $\lim_{x\rightarrow x_0}f(x)/g(x)=1$. Then it is known in \cite{Liu} that if $\kappa<\infty$, there exists a constant $C_0\in(0,\infty)$ such that 
\begin{equation}
\P(W_\infty\geq x)\sim C_0 x^{-\kappa}, \textrm{ as } x\rightarrow\infty.
\end{equation}
%{\color{red} what if $\kappa=\infty$?}
Moreover, according to Theorem 2.1 of \cite{Liu}, for any $p\in(1,\kappa)$, $\E[W_\infty^p]<\infty$ if and only if $\E[W_1^p]<\infty$.

Let $\tM:=\inf_{u\in\T}V(u)$ be the minimum of the branching random walk and let
\[
\mathcal{M}_e:=\sup_{u\in\T}e^{-V(u)}=\exp(-\tM).
\]
Then the assumption \ref{cond2} implies that $\tM\in\R$, $\P$-a.s. We have the following theorem on the joint tail of $(W_\infty, \mathcal{M}_e)$.
\begin{thm}\label{thmtailWM} Under the assumptions \ref{cond1}, \ref{cond2}, \ref{cond3} and \ref{cond4},
if $\kappa\in(1,\infty)$, there exists an decreasing continuous function $\gamma: [0,\infty)\rightarrow (0,\infty)$ such that $\gamma(0)>0$, $\lim_{a\rightarrow\infty}\gamma(a)=0$ and that for any $a\geq0$,
\begin{equation}
\P\left(W_\infty \geq ax, \mathcal{M}_e\geq x\right)\sim \gamma(a)x^{-\kappa},\textrm{ as }x\rightarrow\infty,
\end{equation}
where $\gamma$ will be given later in \eqref{gammaa}. In particular, for $a=0$, as $x\rightarrow\infty$,
\[
\P\left(\mathcal{M}_e\geq x\right)\sim c_{\tM}x^{-\kappa}.
\]
with $c_{\tM}=\gamma(0)$.
\end{thm}
This result brings out the following theorem on the randomly biased random walk $(X_n)_{n\geq0}$.

\begin{thm}\label{tailmaxL}
Under the assumptions \ref{cond1}, \ref{cond2}, \ref{cond3} and \ref{cond4}, there exists $c^\star_\kappa\in(0,\infty)$ such that
\begin{enumerate}
\item if  $\kappa\in(1,2)$,
\begin{equation}
\p\left(\max_{u\in\T}\overline{L}_{\tau_1}(u)\geq x\right)\sim c_\kappa^\star x^{-1}, \textrm{ as } x\rightarrow\infty.
\end{equation}
\item if $\kappa=2$,
\begin{equation}
\p\left(\max_{u\in\T}\overline{L}_{\tau_1}(u)\geq x\right)\sim \frac{c^\star_\kappa}{x\sqrt{\log x}} , \textrm{ as } x\rightarrow\infty.
\end{equation}
\item if $\kappa\in(2,\infty)$,
\begin{equation}
\p\left(\max_{u\in\T}\overline{L}_{\tau_1}(u)\geq x\right)\sim c^\star_\kappa x^{-\kappa/2}, \textrm{ as } x\rightarrow\infty.
\end{equation}
\item if $\kappa=\infty$, for any $p>1$,
\begin{equation}
\e\left[\left(\max_{u\in\T}\overline{L}_{\tau_1}(u)\right)^p\right]<\infty.
%\p\left(\max_{u\in\T}\overline{L}_{\tau_1}(u)\geq x\right)\sim c^\star_\kappa \sqrt{\p\left(M_1\geq x\right)}, \textrm{ as } x\rightarrow\infty,
\end{equation}
%where $M_1$ is defined in \eqref{}.
\end{enumerate}
\end{thm}
In addition, as a corollary of Theorem \ref{tailmaxL}, we have the following result on the maximal edge local time up to time $\tau_n$.
\begin{thm}\label{maxLn}Under Assumptions \ref{cond1}, \ref{cond2}, \ref{cond3} and \ref{cond4},
\begin{enumerate}
\item if $\kappa\in(1,2)$, under the annealed probability $\p$,
\begin{equation}
\frac{\max_{u\in \T}\overline{L}_{\tau_n}(u)}{n}\xrightarrow[n\rightarrow\infty]{(d)} X_*,
\end{equation}
where $X_*$ is a positive random variable of distribution function $ \E\left[e^{-c^\star_\kappa \frac{W_\infty}{t}}; \M_e\leq t\right]$ (which stochastically dominates ${\cal M}_e$).
\item if $\kappa\geq2$, under the annealed probability $\p$,
\begin{equation}
\frac{\max_{u\in \T}\overline{L}_{\tau_n}(u)}{n}\xrightarrow[n\rightarrow\infty]{\textrm{ in }\p} {\cal M}_e.
\end{equation}
\end{enumerate}
\end{thm}

\begin{rem}
In fact, it is known from Theorem~1.1 in \cite{CdRH16} that if $\kappa>2$, even without Assumption~\ref{cond3}, 
\begin{equation}
\frac{\max_{u\in \T}\overline{L}_{\tau_n}(u)}{n}\xrightarrow[n\rightarrow\infty]{\p-a.s.} {\cal M}_e;
\end{equation}
and that if $\kappa\in(1,2]$, $\p$-a.s.,
\begin{equation}
\liminf_{n\rightarrow\infty}\frac{\max_{u\in \T}\overline{L}_{\tau_n}(u)}{n}={\cal M}_e. 
\end{equation}
Moreover, Proposition~5.1 of~\cite{CdRH16} says that when $\kappa\in(1,2)$, $\p(\cdot|\#\T=\infty)$-a.s.,
\begin{equation}
\limsup_{n\rightarrow\infty}\frac{\max_{u\in \T}\overline{L}_{\tau_n}(u)}{n}=\infty.
\end{equation}
\end{rem}

\begin{rem}
Notice that if one defines the vertex local times as $L_n(u):=\sum_{k=1}^n\ind{X_k=u}$, then $L_{\tau_n}(u)=\overline{L}_{\tau_n}(u)+\sum_{\pa{v}=u}\overline{L}_{\tau_n}(v)$. Thus, the behaviour of vertex local times is closely related to that of edge local times, and we expect our result to hold for vertex local times. However, vertex local times are less convenient to manipulate, and our method would not apply without several technical adjustments. 
\end{rem}

A natural question is then to study $\max_{u\in \T}\overline{L}_{n}(u)$ under $\p$ and under $\p^\en$. In fact, the asymptotic behaviour of $\tau_n$ under the quenched probability $\p^\mathcal{E}$ has been considered by Hu \cite{Hu}. The quenched and joint asymptotic of $(\tau_n, \max_{u\in \T}\overline{L}_{\tau_n}(u))$ will be treated in an upcoming paper.

The organisation of this paper is as follows. Sections~\ref{tailBRW} and \ref{lemBRW} deal with the branching random walk and are self-contained for the proof of Theorem \ref{thmtailWM}. In Section~\ref{idea}, we prove Theorems \ref{tailmaxL} and \ref{maxLn} by use of Proposition \ref{tailML}.
%In Section~\ref{tailBRW}, we consider the branching random walk which serves as the environment and prove Theorem \ref{thmtailWM}. 
In Section~\ref{tailRWRE}, we prove Proposition \ref{tailML} from Theorem \ref{thmtailWM} by use of two changes of measures. Section~\ref{lemRWRE} contains the proofs of the lemmas in Section \ref{tailRWRE}.

Throughout the paper, $(c_i)_{i\geq0}$ and $(C_i)_{i\geq0}$ denote positive constants. We say that $a_n\sim b_n$ as $n\rightarrow\infty$ if $\lim_{n\rightarrow\infty}\frac{a_n}{b_n}=1$. We set $\sum_\emptyset=0$ and $\prod_{\emptyset}=1$. For $x,y\in\R$, we let $x\wedge y:=\min(x,y)$ and $x\vee y:=\max(x,y)$. 

\section{Maximal edge local time: proofs of Theorems \ref{tailmaxL} and \ref{maxLn}}
\label{idea}

In this section, we consider $\{\overline{L}_{\tau_1}(u); u\in\T\}$. In fact, for $x\in\T$ with children $\{y:\pa{y}=x\}$, and for any $s_y\in[0,1]$, 
\begin{equation*}
\e^\en\left[\prod_{y: \pa{y}=x}(s_y)^{\overline{L}_{\tau_1}(y)}\Big\vert \overline{L}_{\tau_1}(x)=k \right]=\frac{e^{-kV(x)}}{\left(e^{-V(x)}+\sum_{y:\pa{y}=x}(1-s_y) e^{-V(y)}\right)^k} \quad \forall k\geq1.
\end{equation*}
In other words, $(\overline{L}_{\tau_1}(y))_{\pa{y}=x}$ has negative multinomial distribution with parameters $\overline{L}_{\tau_1}(x)$ and $(\frac{e^{-V(y)}}{e^{-V(u)}+\sum_{\pa{z}=x}e^{-V(z)}})_{\pa{y}=x}$. In particular, if $x$ is of type $\overline{L}_{\tau_1}(x)=0$, then all its descendants are of type $0$. According to Lemma 3.1 of \cite{AdR15}, for any $k\geq1$, under the annealed probability $\p$, $\{\overline{L}_{\tau_k}(u); u\in\T\}$ is a multi-type Galton Watson tree with types taking values in $\N$, whose root $\rho$ is of type $k$. 

We denote by $\xi=\{\xi_i; i\geq1\}$ its offspring distribution. Here $\xi_i$ stands for the offspring law of a vertex of type $i$. For any $k\geq1$, denote by $\p_k$ the law of a multi-type GW tree with offspring $\xi$ and initial type $k$. From now on, we use $(\beta(u), u\in\T)$ to represent this multi-type GW tree. So, $\p(\max_{u\in\T}\overline{L}_{\tau_n}(u)\in\cdot)=\p_n(\max_{u\in\T}\beta(u)\in\cdot)$. 
%
%With a little abuse of notations, let
%\[
%\mathcal{L}_1:=\{u\in\T: \beta(u)=1, \min_{\rho<v<u}\beta(v)\geq2\},\textrm{ and } \B_1:=\{u\in\T: \min_{\rho<v<u}\beta(v)\geq2\}\cup\{\rho\}.
%\]  
%Also let $L_1:=\#\L_1$ and 
%\[
%M_1:=\max_{u\leq \L_1}\beta(u)
%\]

%Then $\p(\max_{u\in\T}\overline{L}_{\tau_1}(u)\geq x)=\p_1(\max_{u\in\T}\beta(u)\geq x)$. 
Now define
\begin{equation}\label{lineL}
\L_1:=\{x\in\T: \beta(x)=1, \min_{\rho<y<x}\beta(y)\geq2\}.
\end{equation}
and $\B_1:=\{x\in\T: \min_{\rho<y< x}\beta(y)\geq2\}\cup\{\rho\}$. For convenience, we sometimes write $u\leq\mathcal{L}_1$ for $u\in\B_1$, and $u<\mathcal{L}_1$ for $u\in\B_1\setminus\L_1$. %Moreover, let $L_1$ denote the cardinality of $\mathcal{L}_1$. 
Observe that
\begin{equation}\label{Maxelt}
\max_{u\in\T}\beta(u)=\max\{\max_{u\in \B_1}\beta(u), \max_{u\in\mathcal{L}_1} \max_{v: u\leq v} \beta(v)\}.
\end{equation}
As $(\beta(u))_{u\in\T}$ is a multi-type Galton-Watson tree with root of type 1 under $\p_1$, by Markov property at the stopping line $\mathcal{L}_1$, $\{\max_{v: u\leq v} \beta(v)\}_{u\in\mathcal{L}_1}$ are i.i.d. and distributed as $\max_{u\in\T}\beta(u)$ under $\p_1$, and independent of $(\beta(u), u\in\B_1)$. Let
\begin{equation}\label{cardLM}
L_1:=Card(\mathcal{L}_1),\quad M_1:=\max_{u\in \B_1}\beta(u),\quad M^\star:=\max_{u\in\T}\beta(u).
\end{equation}
So we rewrite the equation \eqref{Maxelt} as follows: under $\p_1$,
\begin{equation}\label{keyeq}
M^\star\egloi\max\{M_1, \max_{1\leq i\leq L_1}M^{\star}_i\}
\end{equation}
where $M^{\star}_i, i\geq1$ are i.i.d. copies of $M^\star$, independent of $(M_1, L_1)$. Thus the tail of $\max_{u\in\T}\beta(u)$ under $\p_1$ depends mainly on the joint tail of $(M_1, L_1)$. 

\begin{prop}\label{tailML}
Under the assumptions \ref{cond1}, \ref{cond2}, \ref{cond3} and \ref{cond4}, and assuming $\kappa\in(1,\infty)$, there exists a  constant $C_\infty\in(0,\infty)$ such that for any $a\geq0$, 
\begin{equation}
\p_1(L_1\geq ax, M_1\geq x)\sim C_\infty\gamma(a)x^{-\kappa},\textrm{ as } x\rightarrow\infty,
\end{equation}
where $C_\infty$ is defined in equation~\eqref{eq:defCinf}. In particular, $\p(M_1\geq x)\sim C_\infty \gamma(0) x^{-\kappa}$. Moreover, there exists $c_L\in(0,\infty)$ such that
\[
\p_1(L_1\geq x)\sim c_Lx^{-\kappa},\textrm{ as } x\rightarrow\infty,
\]
where $c_L=C_0C_\infty$.
\end{prop}
The proof of Proposition \ref{tailML} is postponed to Section~\ref{tailRWRE}. Here we use the result to obtain the tail of $M^\star=\max_{u\in\T}\beta(u)$ under $\p_1$.

\begin{proof}[Proof of Theorem \ref{tailmaxL}]
If $\kappa\in(1,\infty)$, given equation~\eqref{keyeq}, Proposition~\ref{tailML} allows us to apply Corollary~1.4 of~\cite{ChenMa} which yields Theorem~\ref{tailmaxL}. 

If $\kappa=\infty$, given equation~\eqref{keyeq}, as $x\rightarrow\infty$,
\[
\p\left(\max_{u\in\T}\overline{L}_{\tau_1}(u)\geq x\right)\sim c^\star_\kappa \sqrt{\p_1(M_1\geq x)},
\]
because $L_1$ has finite variance according to Lemma \ref{Mom}. Moreover, $M_1$ has moments of all orders by Lemma \ref{Mom}, which concludes that $\max_{u\in\T}\overline{L}_{\tau_1}(u)$ also has moments of all orders.
\end{proof}

\begin{prop}\label{cvgLM}
Under the assumptions \ref{cond1}, \ref{cond2}, \ref{cond3} and \ref{cond4}, for any $\alpha\in(0,\kappa-1)$, we have as $n\rightarrow\infty$,
\begin{equation}
\e_n\left[|\frac{L_1}{n}-W_\infty|^{1+\alpha}+|\frac{M_1}{n}-\mathcal{M}_e|^{1+\alpha}\right]\rightarrow 0.
\end{equation}
\end{prop}
The proof will be postponed in Section~\ref{tailRWRE}. Now we are ready to prove Theorem \ref{maxLn}.

\begin{proof}[Proof of Theorem \ref{maxLn}] When $\kappa\in(1,2)$, Theorem \ref{tailmaxL} tells us that $\p_1(\max_{u\in\T}\beta(u)> x)\sim \frac{c_\kappa^\star}{x}$ as $x\rightarrow\infty$. Observe that for any $t>0$,
\[
\p_n\left(\max_{u\in\T}\beta(u)\leq tn\right)=\p_n\left(\max_{u\leq \L_1}\beta(u)\leq tn,  \max_{u\in\L_1}\max_{v: v\geq u}\beta(v)\leq tn\right)
\]
By Markov property at $\L_1$, it follows that
\begin{align*}
\p_n\left(\max_{u\in\T}\beta(u)\leq tn\right)=&\e_n\left[\prod_{u\in\L_1}\p(\max_{v: v\geq u}\beta(v)\leq tn\vert \beta(u)=1); \max_{u\leq \L_1}\beta(u)\leq tn\right]\\
=&\e_n\left[\left(1-\p_1(\max_{v\in\T}\beta(v)> tn)\right)^{L_1}; M_1\leq tn\right]\\
=& \e_n\left[ e^{- c^\star_\kappa\frac{L_1}{tn}(1+o_n(1))}; M_1\leq tn\right]
\end{align*}
Proposition \ref{cvgLM} implies that under $\p_n$, $(\frac{L_1}{n}, \frac{M_1}{n})$ converges in probability to $(W_\infty, \M_e)$. We hence deduce that
\[
\lim_{n\rightarrow\infty}\p_n\left(\max_{u\in\T}\beta(u)\leq tn\right)= \E\left[e^{-c^\star_\kappa \frac{W_\infty}{t}}; \M_e\leq t\right].
\]
One can easily check that $F(t):= \E\left[e^{-c^\star_\kappa \frac{W_\infty}{t}}; \M_e\leq t\right]$ with $F(0):=0$ is a distribution function. Therefore, under $\p$, $\frac{\max_{u\in\T}\overline{L}_{\tau_n}(u)}{n}$ converges in law to some random variable $X_*$ with distribution function $F$.

When $\kappa\geq2$, as  $\p_1(\max_{u\in\T}\beta(u)> x)\sim \frac{o_x(1)}{x}$, similar arguments yield that for any $t>0$,
\begin{align*}
\p_n\left(\max_{u\in\T}\beta(u)\leq tn\right)=&\e_n\left[\left(1-\p_1(\max_{v\in\T}\beta(v)> tn)\right)^{L_1}, M_1\leq tn\right]\\
=& \e_n\left[ e^{-L_1\frac{o_n(1)}{n}}; M_1\leq tn\right]
\end{align*}
which converges to $\P(\M_e\leq t)$. As a result, under $\p$, $\frac{\max_{u\in\T}\overline{L}_{\tau_n}(u)}{n}$ converges in law to $\M_e$. Moreover, one could show that for any $\epsilon>0$,
\begin{align*}
\p_n\left(|\frac{\max_{u\in\T}\beta(u)}{n}-\M_e|\ge \epsilon\right)\leq & o_n(1)+\p_n\left(|\frac{M_1}{n}-\M_e|\ge \delta/2\right)=o_n(1).
\end{align*}
This suffices to conclude the convergence in probability of $\frac{\max_{u\in\T}\overline{L}_{\tau_n}(u)}{n}$ under $\p$.
\end{proof}

\section{Tail behaviours of the branching random walk}
\label{tailBRW}

This section is devoted to proving Theorem \ref{thmtailWM}.

Let us first consider $W_\infty$, the almost sure limit of the additive martingale $W_n=\sum_{|z|=n}e^{-V(z)}$. According to~\cite{Big77}, under assumptions~\ref{cond1} and~\ref{cond4}, $W_\infty>0$ if and only if $\T$ is infinite. Immediately, one sees that $\P$-a.s.,
\[
W_\infty=\sum_{|z|=1}e^{-V(z)}W_\infty^{(z)},
\]
where $W_\infty^{(z)}, |z|=1$ are martingale limits associated with the subtree rooted at $z$, respectively, which are therefore i.i.d. copies of $W_\infty$ and are independent of $(V(z), |z|=1)$. We will generalise this decomposition.

For any $u\in\T$ such that $|u|=n$, let $(\rho, u_1,\cdots, u_n)$ be its ancestral line. For any $z\in \T$, let $\Omega(z)$ be the set of all brothers of $z$, i.e.,
\begin{equation}\label{defbro}
\Omega(z):=\{v\in \T: \overleftarrow{v}=\overleftarrow{z}, v\neq z\}.
\end{equation}
Then, observe that
\[
W_\infty=\sum_{k=1}^n \sum_{z\in\Omega(u_k)}e^{-V(z)}W_\infty^{(z)}+e^{-V(u)}W_\infty^{(u)}, \P\textrm{-a.s.}
\]
To deal with $(W_\infty,\mathcal{M}_e)$, recall that $\tM=\inf_{u\in\T}V(u)$ and $\mathcal{M}_e=e^{-\tM}$. Let us take $u^*\in\T$ such that $V(u^*)=\tM$, if there exist several choices, one chooses $u^*$ at random among the youngest ones. So, $\P$-a.s.
\begin{equation*}
W_\infty= e^{-\tM}\sum_{k=1}^{|u^*|} \sum_{z\in\Omega(u^*_k)}e^{\tM-V(z)}W_\infty^{(z)}+e^{-\tM}W_\infty^{(u^*)}.
\end{equation*}
One sees hence that
\begin{equation}\label{eMW}
\mathcal{W}^{\tM}:=e^{\tM}W_\infty=\sum_{k=1}^{|u^*|} \sum_{z\in\Omega(u^*_k)}e^{\tM-V(z)}W_\infty^{(z)}+W_\infty^{(u^*)}.
\end{equation}

Observe that the joint law of $(W_\infty,\mathcal{M}_e)$ is totally given by the joint law of $(\mathcal{W}^{\tM}, \tM)$. Let us state the following theorem, which is largely inspired by \cite{Madaule17} in which the boundary case is treated. 

\begin{thm}\label{BRWcvg}
Suppose that the assumptions \ref{cond1},\ref{cond2}, \ref{cond3} and \ref{cond4} are all fulfilled. Assume that $\kappa\in(1,\infty)$. Then there exists a constant $c_\tM\in(0,\infty)$ such that as $x\rightarrow\infty$, 
\begin{equation}\label{tailM}
\P(\tM\leq -x)\sim c_\tM e^{-\kappa x}.
\end{equation}
Further, conditionally on $\{\tM\leq-x\}$, the following convergence in law holds as $x\rightarrow\infty$:
\begin{equation}
(\mathcal{W}^{\tM}, \tM+x)\Longrightarrow (\mathcal{W}^{\tM}_\infty, -U)
\end{equation}
where $U$ is an exponential random variable with parameter $\kappa$, independent of $\mathcal{W}^\tM_\infty$.
\end{thm}

\begin{lem}\label{BRWtailM}
Under the assumptions \ref{cond1},\ref{cond2}, \ref{cond3} and \ref{cond4}, if $\kappa\in(1,\infty)$, we have
\begin{equation}
\limsup_{\epsilon\downarrow0}\limsup_{x\rightarrow\infty}x^{\kappa}\P(W_\infty\leq \epsilon x, \M\geq x)=0.
\end{equation}
\end{lem}
The proof of Lemma~\ref{BRWtailM} is given in Subsection~\ref{subsubsec:prooflemmas}. We prove that Theorem \ref{thmtailWM} is a direct consequence of Theorem \ref{BRWcvg} (whose proof is postponed to the next section) in the following.
\begin{proof}[Proof of Theorem \ref{thmtailWM}]
For any $a>0$ and $x>0$, observe that
\begin{align*}
\P(W_\infty\geq ax, \mathcal{M}_e\geq x)
=&\P(W_\infty\geq ax, \tM\leq -\log x)\\
=&\P(e^{-\tM-\log x}\mathcal{W}^\tM\geq a, \tM+\log x\leq 0).
\end{align*}
Let $x\rightarrow\infty$, as a consequence of Theorem \ref{BRWcvg},
\begin{align*}
\lim_{x\rightarrow\infty}x^\kappa\P(W_\infty\geq ax, \mathcal{M}_e\geq x)=&\lim_{x\rightarrow\infty}\P(e^{-\tM-\log x}\mathcal{W}^\tM\geq a\vert \tM+\log x\leq 0)\P(\tM\leq-\log x)x^\kappa\\
= &\P(e^U\mathcal{W}^{\tM}_\infty\geq a)\lim_{x\rightarrow\infty}\P(\tM\leq -\log x)x^\kappa\\
= & \gamma(a),
\end{align*}
with 
\begin{equation}\label{gammaa}
\gamma(a):=c_\tM \P(e^U\mathcal{W}^{\tM}_\infty\geq a)=c_\tM \E\left[\left(1\wedge \frac{\mathcal{W}^\tM_\infty}{a}\right)^{\kappa}\right], \forall a>0.
\end{equation}
Notice that by Lebesgue's dominated convergence thereom, $\gamma$ is continuous on $\R_+^*$. For $a=0$, \eqref{tailM} implies directly that
\[
\P(\mathcal{M}_e\geq x)=\P(\tM\leq -\log x)\sim c_\tM x^{-\kappa}.
\]
Here $\gamma(0)=c_\tM=\lim_{a\to 0}\gamma(a)$ by Lemma \ref{BRWtailM} hence $\gamma$ is also continuous in $0$.
\end{proof}

\begin{rem}
Note that as $c_\tM=\lim_{a\to 0}\gamma(a)$, Lemma \ref{BRWtailM} also implies that $\P(\mathcal{W}_\infty^\tM>0)=1$.
\end{rem}

\subsection{Proof of Theorem \ref{BRWcvg}}
First we state the following lemma which gives a rough estimate on the tail of $\tM$.
\begin{lem}\label{BRWroughbd}
Under the assumptions of Theorem \ref{BRWcvg}, there exist $0<c_1\leq 1$ such that 
\begin{equation}\label{BRWrhbdM}
c_1e^{-\kappa x}\leq \P(\tM\leq -x)\leq e^{-\kappa x}, \forall x\geq1.
\end{equation}
\end{lem}
Its proof is postponed to Section~\ref{TechLem}.

Recall that $\mathcal{W}^{\tM}=e^{\tM}W_\infty=\sum_{k=1}^{|u^*|} \sum_{z\in\Omega(u^*_k)}e^{\tM-V(z)}W_\infty^{(z)}+W_\infty^{(u^*)}$.
Its truncated version is defined for $0\leq t<|u^*|$ by
\begin{align}\label{truncatedW}
\mathcal{W}^{u^*,\leq t}:=&\sum_{k=|u^*|-t}^{|u^*|} \sum_{z\in\Omega(u^*_k)}e^{\tM-V(z)}W_\infty^{(z)}+W_\infty^{(u^*)}\nonumber\\
=&\sum_{k=|u^*|-t}^{|u^*|} e^{V(u^*)-V(u^*_{k-1})}\sum_{z\in\Omega(u^*_k)}e^{-[ V(z)-V(u^*_{k-1})]}W_\infty^{(z)}+W_\infty^{(u^*)}.
\end{align}

Let us state the following results for the truncated random variable.

\begin{lem}\label{BRWmaincvg}
Let $t$ be a fixed integer. Under the assumptions of Theorem \ref{BRWcvg}, for any continuous and bounded function $\phi: \R\rightarrow\R_+$, the following limit exists.
\begin{equation}
\lim_{x\rightarrow\infty} e^{\kappa x}\E\left[\phi(\mathcal{W}^{u^*,\leq t})\ind{\tM\leq -x}\right]:=\mathcal{E}_t(\phi)
\end{equation}
The explicit expression of $\mathcal{E}_t(\phi)$ will be given in Section~\ref{lemBRW} by equation~\eqref{eq:defE}.
\end{lem}

\begin{lem}\label{BRWrest}
Under the assumptions of Theorem \ref{BRWcvg}, for any $\delta>0$, we have
\begin{equation}
\lim_{t\rightarrow\infty}\sup_{x\in\R_+}\P\left(\mathcal{W}^{\tM}-\mathcal{W}^{u^*,\leq t}\geq \delta\vert \tM\leq-x\right)=0.
\end{equation}
\end{lem}

The next lemma states the tightness of the law of $\mathcal{W}^{\tM}$ conditionally on $\tM\leq -x$.
\begin{lem}\label{BRWtight}
Under the assumptions of Theorem \ref{BRWcvg}, we have
\begin{equation}
\lim_{M\rightarrow\infty}\sup_{x\in\R_+}\P[\mathcal{W}^{\tM}\geq M \vert \tM\leq -x]=0.
\end{equation}
\end{lem}

Let us prove Theorem \ref{BRWcvg} by these lemmas.

\begin{proof}[Proof of Theorem \ref{BRWcvg}]
By taking $\phi\equiv1$, the tail \eqref{tailM} of $\tM$ follows from Lemma \ref{BRWmaincvg} with $c_\tM=\mathcal{E}_0(1)$. We have $c_\tM\in(0,\infty)$ because of Lemma \ref{BRWroughbd}.

One sees from Lemmas \ref{BRWtight} and \ref{BRWroughbd} that the joint distribution of $(\mathcal{W}^{\tM}, \tM+x)$ conditionally on $\{\tM\leq -x\}$ is tight. By the classical L\'evy's theorem, it suffices to prove that for any $\theta_1,\theta_2\in\R_+$, the following limit
\begin{equation}
\lim_{x\rightarrow\infty}\E[e^{-\theta_1\mathcal{W}^\tM+\theta_2(\tM+x)}\vert \tM\leq -x]%=\lim_{t\rightarrow\infty}\frac{\mathcal{E}_t(e^{-\theta_1\cdot})}{\mathcal{E}_0(1)}\frac{\theta_2}{\kappa+\theta_2}
\end{equation}
exists. 

First, as given in \eqref{truncatedW}, $0\le\mathcal{W}^{u^*,\leq t}\leq \mathcal{W}^{\tM}$. So, the tightness of $(\mathcal{W}^{\tM}, \tM+x)$ conditionally on $\{\tM\leq -x\}$ yields also tightness of $(\mathcal{W}^{u^*,\leq t}, \tM+x)$ conditionally on $\{\tM\leq -x\}$. On the other hand, it follows from Lemma \ref{BRWmaincvg} that,
\begin{equation}\label{brwmaincvg}
\lim_{x\rightarrow\infty}\E[e^{-\theta_1\mathcal{W}^{u^*,\leq t}+\theta_2(\tM+x)}\vert \tM\leq -x]=\frac{\mathcal{E}_t(e^{-\theta_1\cdot})}{\mathcal{E}_t(1)}\frac{\kappa}{\kappa+\theta_2}=\frac{\mathcal{E}_t(e^{-\theta_1\cdot})}{\mathcal{E}_0(1)}\frac{\kappa}{\kappa+\theta_2},
\end{equation}
where $\mathcal{E}_t(e^{-\theta_1\cdot})>0$ by the tightness.
%Then by the classical L\'evy's theorem, for any $t\in\mathbb{N}^*$, one gets immediately the following convergence in law
%\begin{equation}
%(\mathcal{W}^{u^*,\leq t}, \tM+x)\Longrightarrow (\mathcal{W}^{\tM,\leq t}_\infty, -U),
%\end{equation}
%where $U$ is an exponential random variable with parameter $\kappa$, independent of $\mathcal{W}^{\tM,\leq t}_\infty$.  

Next, observe that $t\mapsto \mathcal{W}^{u^*,\leq t}$ is increasing. So $t\mapsto \mathcal{E}_t(e^{-\theta_1\cdot})$ is decreasing and positive. Consequently,
\begin{equation}\label{cvgbyt}
\lim_{t\rightarrow\infty}\frac{\mathcal{E}_t(e^{-\theta_1\cdot})}{\mathcal{E}_0(1)}\frac{\kappa}{\kappa+\theta_2}\textrm{ exists and is positive.}
\end{equation}
We then are going to show that
\begin{equation}
\lim_{x\rightarrow\infty}\E[e^{-\theta_1\mathcal{W}^\tM+\theta_2(\tM+x)}\vert \tM\leq -x]=\lim_{t\rightarrow\infty}\frac{\mathcal{E}_t(e^{-\theta_1\cdot})}{\mathcal{E}_0(1)}\frac{\kappa}{\kappa+\theta_2}.
\end{equation}
In fact, by \eqref{brwmaincvg} and \eqref{cvgbyt}, it suffices to show that
\begin{equation}
\lim_{t\rightarrow\infty}\limsup_{x\rightarrow\infty}\E[e^{-\theta_1\mathcal{W}^{u^*,\leq t}+\theta_2(\tM+x)}-e^{-\theta_1\mathcal{W}^\tM+\theta_2(\tM+x)}\vert \tM\leq -x]=0.
\end{equation}
Note that for any $\epsilon>0$,
\begin{align*}
0\leq&\E[e^{-\theta_1\mathcal{W}^{u^*,\leq t}+\theta_2(\tM+x)}-e^{-\theta_1\mathcal{W}^\tM+\theta_2(\tM+x)}\vert \tM\leq -x]\\
\leq &\E[(e^{-\theta_1\mathcal{W}^{u^*,\leq t}}-e^{-\theta_1\mathcal{W}^\tM})e^{\theta_2(\tM+x)}; \mathcal{W}^\tM-\mathcal{W}^{u^*,\leq t}\leq \varepsilon\vert \tM\leq -x]+\P(\mathcal{W}^\tM-\mathcal{W}^{u^*,\leq t}\geq \varepsilon\vert \tM\leq -x)\\
\leq & \theta_1\varepsilon+\P(\mathcal{W}^\tM-\mathcal{W}^{u^*,\leq t}\geq \varepsilon\vert \tM\leq -x).
\end{align*}
By Lemma \ref{BRWrest}, we obtain that for any $\varepsilon>0$,
\[
\limsup_{t\rightarrow\infty}\limsup_{x\rightarrow\infty}\E[e^{-\theta_1\mathcal{W}^{u^*,\leq t}+\theta_2(\tM+x)}-e^{-\theta_1\mathcal{W}^\tM+\theta_2(\tM+x)}\vert \tM\leq -x]\leq \theta_1\varepsilon,
\]
which is what we need to conclude the theorem.
\end{proof}

\section{Edge local times of the randomly biased random walk}
\label{tailRWRE}
As stated in Section~\ref{idea}, we are going to study the multi-type GW tree $\{\beta(u), u\in\T\}$ under $\p_n$, which describes the annealed distribution of edge local times up to $\tau_n$.

In what follows, we introduce a new probability measure $\hat{\p}^*$ on a marked multi-type GW tree. 

\subsection{Change of measures and spinal decomposition under $\widehat{\p}^*$}
\label{s:changeofmeasurep}

Let us define 
\[
Z_n:=\sum_{|u|=n}\beta(u), \forall n\geq0.
\]
Then under $\p_k$, $\{Z_n; n\geq0\}$ is a martingale with respect to the natural filtration $\{\nf^\beta_n; n\geq0\}$ where $\nf^\beta_n$ is the sigma-field generated by $\{(u,\beta(u)); |u|\leq n\}$. We are hence ready to define the new probability measure $\widehat{\p}^*$.

Recall that under $\p$, the offspring law of $(\T,\beta)$ is given by $\xi=\{\xi_i; i\geq1\}$. More precisely, for any $n\geq0$, for any $k_1,\cdots,k_n\in\N$, for any $u\in\T$,
\begin{multline}\label{offspringMtype}
\xi_i(k_1,\cdots, k_n)=\p\left(\{v:\overleftarrow{v}=u\}=\{u1,\cdots, un\}; \beta(u1)=k_1,\cdots,\beta(un)=k_n\Big\vert \beta(u)=i\right)\\
=\E\left[{\sum_{j=1}^nk_j+k-1 \choose k_1,\cdots,k_n,k-1}\frac{e^{-kV(u)}\prod_{j=1}^n e^{-k_j V(uj)}}{(e^{-V(u)}+\sum_{j=1}^n e^{-V(uj)})^{\sum_{j=1}^nk_j+k}}; \sum_{v:\overleftarrow{v}=u}1=n\right].
\end{multline}
As $\e_i[Z_1]=i$, we can define $\widehat{\xi}=\{\widehat{\xi}_i; i\geq1\}$ to be another collection of offsprings such that for any $i\geq1$, $\widehat{\xi}_i(k_1,\cdots,k_n)=\frac{\sum_{j=1}^nk_j}{i}\xi_i(k_1,\cdots,k_n)$.
The probability measure $\widehat{\p}^*_i$ on multi-type Galton-Watson tree with a marked ray $(\T,\beta, (w_n)_{n\geq0})$ is defined as follows. 
\begin{enumerate}
\item For the root $\rho$, let $\beta(\rho)=i$ and $w_0=\rho$. 
\item For any $n\geq0$, suppose that the process up to the $n$-th generation with the spine $(w_k)_{0\leq k\leq n}$ has been constructed. The vertex $w_n$ produces its children, independently of the others, according to the offspring $\widehat{\xi}_{\beta(w_n)}$. All other vertices $u$ of the $n$-th generation produce independently their children according to the offspring $\xi_{\beta(u)}$, respectively. The children of all vertices of the $n$-th generation form the $(n+1)$-th generation. We choose $w_{n+1}$ among the children of $w_n$, each $y$ child of $w_n$ being chosen with probability~$\frac{\beta(y)}{\sum_{v: \overleftarrow{v}=w_n}\beta(v)}$.
\end{enumerate}

Usually we call the marked ray $(w_n)_{n\ge0}$ the spine. Denote by $\widehat{\p}_i$ the marginal law of $(\T,\beta)$ constructed above. We state the following proposition from \cite{KLPP}.

\begin{prop}\label{prop:changeofmeasurep}
 Let $i\in\mathbb{N}^*$. Then $\{Z_n/i\}_{n\geq0}$ is a nonnegative $\p$-martingale, and the following assertions hold.
\begin{enumerate}
\item 
\[
\frac{d\widehat{\p}_i}{d\p_i}\Big\vert_{\nf^\beta_n}= \frac{Z_n}{i}, \forall n\geq0.
\]
\item\label{item:changemeasure} For any $u\in\T$ of the $n$-th generation,
\begin{equation}\label{pchoosew}
\widehat{\p}^*_i(w_n=u\vert\nf^\beta_n)=\frac{\beta(u)}{Z_n}.
\end{equation}
\item Under $\widehat{\p}^*_i$, $\{\beta(w_k); k\geq0\}$ is a recurrent Markov chain taking values in $\N^*$, started from $i$,  with transition probabilities $p_{i,j}$ such that
\begin{equation}\label{MC}
p_{i,j}={i+j-1 \choose i}\E\left[\sum_{|u|=1}\frac{e^{-jV(u)}}{(1+e^{-V(u)})^{i+j}}\right], \forall i,j\geq1.
\end{equation}
\end{enumerate}

\end{prop} 

Moreover, we observe that this Markov chain $(\beta(w_k);k\geq 0)$ admits an invariant law $(\pi_j)_{j\geq 1}$ whose expression can be found in Section~6.1 of~~\cite{AdR15}. Point~\ref{item:changemeasure} of Proposition~~\ref{prop:changeofmeasurep} yields the \textit{multi-type many-to-one lemma} as follows:
\begin{lem}\label{lemma:many-to-one-multi-type}
For all $n\in\N$, let $g:{\N}^{n+1} \to \R_+$ be a positive measurable function and $X_n$ a positive ${\mathcal{F}}^\beta_n$-measurable random variable, then
\begin{equation*}
\e_{i}\Big[\sum_{|u|=n}\beta(u)g(\beta(\rho),\beta(u_1),\beta(u_2),\dots,\beta(u_n)) X_n \Big]=i\widehat{\e}^*_i\Big[g(\beta(w_0),\beta(w_1),\beta(w_2),\dots,\beta(w_n)) X_n \Big]. 
\end{equation*}
\end{lem}

\subsection{Second construction of  $\widehat{\p}^*$}
In this subsection, we introduce another construction of  $\widehat{\p}^*$ which was borrowed from \cite{dR16}. Recall that the environment $\en=\{(u, V(u)); u\in\T\}$ is given by a branching random walk for which $W_n:=\sum_{|u|=n}e^{-V(u)}$ is a $\P$-martingale with respect to the filtration $\{\nf^V_n; n\geq0\}$. We first define another probability $\Q^*$ on branching random walk with a marked spine ${\en^*}:=(\T, V, (w_n)_{n\geq0})$. Then on the new environment ${\en^*}$, we introduce the associated biased random walks and their edge local times to reconstruct the marked multi-type GW tree under  $\widehat{\p}^*$.

\subsubsection{Change of measures and spinal decomposition: %$\{(V(u), u\in\T), (w_n)_{n\geq0}\}$ under 
$\Q^*$}\label{changeofmeasureBRW}

Recall that under $\P$, the branching random walk is constructed by use of the point process $\C$. Let us introduce a probability measure $\Q^*$ of a branching random walk with a spine: $\{(V(u); u\in\T), (w_n, V(w_n))_{n\geq0}\}$. First, as $\E[\int e^{-x}\C(dx)]=1$, let $\widehat{\C}$ be a point process with Radon-Nykodim derivative $\int e^{-x}\C(dx)$ with respect to the law of $\C$. We use $\widehat{\C}$ and $\C$ to construct  $\{(V(u); u\in\T), (w_n, V(w_n))_{n\geq0}\}$ under $\Q^*_x$ for any $x\in\R$ as follows.
\begin{enumerate}
\item For the root $\rho$, let $V(\rho)=x$ and $w_0=\rho$. $w_0$ gives birth to its children according to the point process $\widehat{\C}$ (i.e., the relative positions of its children with respect to $V(w_0)$ are distributed as $\widehat{\C}$).
\item For any $n\geq0$, suppose that the process with the spine $(w_k)_{0\leq k\leq n}$ has been constructed up to the $n$-th generation. All vertices of the $n$-th generation, except $w_n$, produce independently their children according to the law of $\C$. Yet, the vertex $w_n$ produces its children, independently of the others, according to the law of $\widehat{\C}$. All the children of the vertices of the $n$-th generation form the $(n+1)$-th generation, whose positions are denoted by $V(\cdot)$. And among the children of $w_n$, we choose $w_{n+1}=u$ with probability $\frac{e^{-V(u)}}{\sum_{z:\overleftarrow{z}=w_n}e^{-V(z)}}$.
\end{enumerate}
 
 We denote by $\Q_x$ the marginal distribution of $(\T, (V(u), u\in\T))$. For simplicity, write $\Q^*$ and $\Q$ for $\Q_0^*$ and $\Q_0$ respectively. Let us state the following proposition given in Lyons \cite{Lyons}. 
 \begin{prop}\label{BRWchangeofp}
 \begin{enumerate}
\item For any $n\geq0$, and $x\in\R$,
\[
\frac{d\Q_x}{d\P_x}\vert_{\mathcal{F}^V_n}=e^xW_n=\sum_{|u|=n}e^{-V(u)+x}, 
\]
where $\mathcal{F}^V_n$ denotes the sigma-field generated by $((u,V(u)); |u|\leq n)$. 
\item For any vertex $u\in\T$ of the $n$-th generation,
\[
\Q^*_x(w_n=u\vert\mathcal{F}^V_n)=\frac{e^{-V(u)}}{W_n}.
\]
\item Under $\Q^*_x$, $(V(w_n); n\geq0)$ is a random walk with i.i.d. increments and started from $x$.
\end{enumerate}
\end{prop}
Since according to~\cite{Big77} ,under our assumptions, the additive martingale $W_n$ converges in $L^1$ to $W_\infty$ under $\P$, one has $d\Q=W_\infty d\P$, and $W_\infty$ is also $\Q$-a.s. the limit of $W_n$.

\subsubsection{Reconstruction of $\widehat{\p}^*$ on biased environment $\Q^*$}
Let us introduce another interpretation of the multi-type Galton-Watson tree under $\widehat{\p}_1^*$ which was first given and proved in Proposition 5 of \cite{dR16}. Given the marked environment ${\en^*}:=(\T, V, (w_n)_{n\geq0})$, we denote by $\{X_k^{(1, w_i)}; k\geq0\}_{i\geq0}$ and $\{X_k^{(2, w_i)}; k\geq0\}_{i\geq0}$ two i.i.d. sequence of killed nearest-neighbour random walks as follows. For any $i\geq0$ fixed, $\{X_k^{(1, w_i)}; k\geq0\}$ is a random walk on $\T\cup\{\pa{\rho}\}$ started at $w_i$ such that before hitting $w_{i-1}$ (with $w_{-1}:=\rho$), the transition probabilities are
\[
\p^{{\en^*}}(X_{n+1}^{(1,w_i)}=y\vert X_n^{(1,w_i)}=x)=
\begin{cases}
\frac{e^{-V(x)}}{e^{-V(x)}+\sum_{z: \overleftarrow{z}=x}e^{-V(z)}}\textrm{ if } \overleftarrow{x}=y\\
\\
\frac{e^{-V(y)}}{e^{-V(x)}+\sum_{z: \overleftarrow{z}=x}e^{-V(z)}}\textrm{ if } \overleftarrow{y}=x.
\end{cases}
\]
When it reaches $w_{i-1}$, it is killed instantly. Let
\[
\tilde{\beta}^{j}_{i}(u):=\sum_{n\geq0}\ind{X_n^{(j,w_i)}=\overleftarrow{u}, X_{n+1}^{(j,w_i)}=u}, \forall j\in\{1,2\},
\]
and
\[
\tilde{\beta}^j(u)=\sum_{i\geq0}\beta^{j}_{i}(u), \forall j\in\{1,2\}.
\]
Finally, let 
\begin{equation}\label{eq:defbetatilde}
\tilde{\beta}(u)=\tilde{\beta}^1(u)+\tilde{\beta}^2(u)+\ind{\exists i\geq0: u=w_i}
\end{equation}
Then, according to Proposition~5 of~\cite{dR16}, the marked tree $\{\T, (\tilde{\beta}(u), u\in\T), (w_n)_{n\geq0}\}$ under $\p^{{\en^*}}\times\Q^*(d{\en^*})$ has the same law as $(\T,\beta, (w_n)_{n\geq0})$ under $\widehat{\p}^*_1$. With the convention, we still use $\widehat{\p}^*_1$ to represent the annealed probability  $\p^{{\en^*}}\times\Q^*(d{\en^*})$, and we will use the notation $\beta$ for both constructions. 

%and use $\L_1$, $L_1$, $M_1$, $\Omega(z)$ for the corresponding objects of $\{\T, (e(u), u\in\T), (w_n)_{n\geq0}\}$ as in \eqref{lineL}, \eqref{cardLM} and \eqref{defbro}. 

This fact sums up in the following diagram. \\

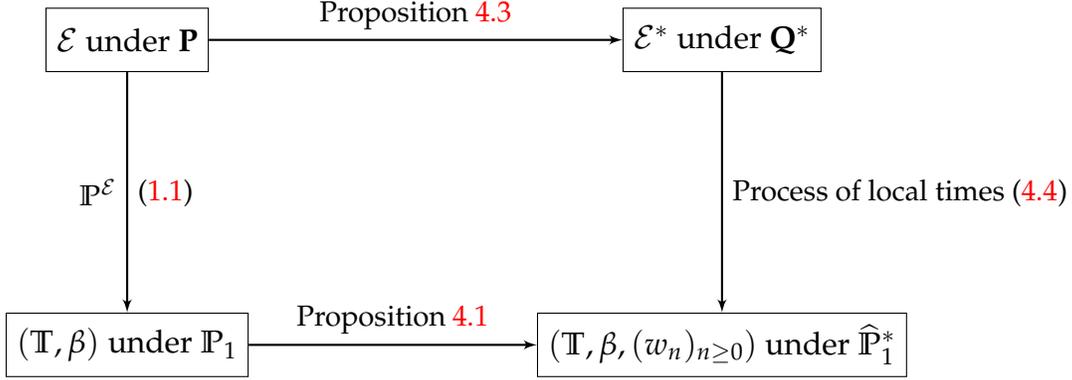
\tikzstyle{block} = [rectangle, draw, text centered, minimum height=2em]%rectangle, draw, text width=6em, text centered, rounded corners, minimum height=4em
\tikzstyle{line} = [draw, -latex', thick]
\tikzstyle{line2} = [draw, latex'-latex',thick]
\begin{center}\begin{tikzpicture}[node distance=1cm, auto]
\node (init) {};
\node [block] (A) {$\en$ under $\P$};
\node [block, below=3.2cm of A] (B) {$(\T,\beta)$ under $\p_1$};
\node [block, right=5.5cm of A] (C) {${\en^*}$ under $\Q^*$};
\node [block, below=3.2cm of C] (D) {$(\T,\beta,(w_n)_{n\geq 0})$ under $\widehat{\p}^*_1$};
\begin{small}
\path [line] (A) -- node [midway,left ] {$\p^\en$} node [midway,right] {\eqref{RW}} (B);
\path [line] (B) -- node [midway,above ] {Proposition~\ref{prop:changeofmeasurep}} (D);
\path [line] (A) -- node [midway,above ] {Proposition~\ref{BRWchangeofp}} (C);
\path [line] (C) -- node [midway,right ] {Process of local times~\eqref{eq:defbetatilde}} (D);
\end{small}
\end{tikzpicture}
\end{center}

\subsection{Proof of Proposition \ref{tailML}: joint tail of $(L_1,M_1)$}

In fact, the joint tail of $(L_1, M_1)$ under $\p_1$ follows from the following results.

\begin{lem}\label{tailsmallLM}
Under the assumptions of Proposition \ref{tailML}, one has
\begin{equation}
\lim_{a\downarrow0} \limsup_{x\rightarrow\infty}x^{\kappa}\p_1(L_1\leq ax, M_1\geq x)=0.
\end{equation}
\end{lem}

\begin{lem}\label{tailLMp}

%\begin{enumerate}
Under the assumptions of Proposition \ref{tailML}, there exist a decreasing continuous function $\eta_1:[0,\infty)\rightarrow (0,\infty)$ such that for any $a>0$, as $x\rightarrow\infty$,
\begin{equation*}
\widehat{\p}^*_1\left(L_1\geq ax, M_1\geq x\right)\sim \eta_1(a) x^{-(\kappa-1)};
\end{equation*}
%\item for any $a>0$ and $x>0$, 
%\begin{equation}
%\e_1[L_1; L_1\geq ax, M_1\geq x]= \widehat{\p}^*_1\left(L_1\geq ax, M_1\geq x\right);
%\end{equation}
%\item moreover, 
%\[
%\p_1(L_1\geq ax, M_1\geq x)\sim \eta(a)x^{-\kappa}, \textrm{ as } x\rightarrow\infty,
%\]
%where
%\[
%\eta(a)=\int_a^\infty \left(\eta_1(a)-\eta_1(u)\right)\frac{du}{u^2}\in [0,\infty).
%\]

%\end{enumerate}

\end{lem}

\begin{proof}[Proof of Proposition \ref{tailML}]
Let us first show that
\begin{equation}\label{eq:tailmoments}
\e_1[L_1; L_1\geq ax, M_1\geq x]= \widehat{\p}^*_1\left(L_1\geq ax, M_1\geq x\right).
\end{equation}\label{meantoprobab}
Observe that by Proposition~\ref{prop:changeofmeasurep},
\begin{align*}
\e_1[L_1; L_1\geq ax, M_1\geq x]=&\e_1\left[\Big(\sum_{k\geq1}\sum_{|u|=k}\ind{u\in\mathcal{L}_1}\Big)\ind{L_1\geq ax, M_1\geq x}\right]\\
=&\sum_{k\geq 1}\e_1\left[\sum_{|u|=k}\ind{u\in\mathcal{L}_1}\p(L_1\geq ax, M_1\geq x\vert \mathcal{F}^\beta_k)\right]\\
=&\sum_{k\geq1}\widehat{\e}_1\left[\sum_{|u|=k}\frac{1}{Z_n}\ind{u\in\mathcal{L}_1}\p(L_1\geq ax, M_1\geq x\vert \mathcal{F}^\beta_k)\right]
%=&\sum_{k\geq1}\widehat{\e}_1\left[\widehat{\e}_1[\ind{w_k\in\mathcal{L}_1}\vert\mathcal{F}^\beta_k]{\p}(L_1\geq ax, M_1\geq x\vert \mathcal{F}^\beta_k)\right]\\
%=&\sum_{k\geq1}\widehat{\e}_1\left[\ind{w_k\in\mathcal{L}_1}\p(L_1\geq ax, M_1\geq x\vert \mathcal{F}^\beta_k)\right]
\end{align*}
which by \eqref{pchoosew}, is equal to
\begin{align*}
&\sum_{k\geq1}\widehat{\e}_1\left[\widehat{\e}^*_1[\ind{w_k\in\mathcal{L}_1}\vert\mathcal{F}^\beta_k]{\p}(L_1\geq ax, M_1\geq x\vert \mathcal{F}^\beta_k)\right]\\
=&\sum_{k\geq1}\widehat{\e}^*_1\left[\ind{w_k\in\mathcal{L}_1}\p(L_1\geq ax, M_1\geq x\vert \mathcal{F}^\beta_k)\right]
\end{align*}
Note that given $\{w_k\in\mathcal{L}_1\}$, $\p(L_1\geq x, M_1\geq ax\vert \mathcal{F}^\beta_k)=\widehat{\p}^*(L_1\geq x, M_1\geq ax\vert \mathcal{G}^\beta_k)$ where $\mathcal{G}^\beta_k$ denotes the sigma-field generated by $((u,\beta(u))_{|u|\leq k}, (w_i)_{i\leq k})$. Therefore,
\begin{align*}
\e_1[L_1; L_1\geq ax, M_1\geq x]=&\sum_{k\geq1}\widehat{\e}^*_1\left[\ind{w_k\in\mathcal{L}_1}\widehat{\p}^*(L_1\geq ax, M_1\geq x\vert \mathcal{G}^\beta_k)\right]\\
=&\sum_{k\geq1}\widehat{\p}^*_1\left[w_k\in\mathcal{L}_1, L_1\geq ax, M_1\geq x\right]
\end{align*}
Recall also that the spine $(\beta(\omega_k); k\geq0)$ is a recurrent Markov chain under $\widehat{\p}^*_1$. So, $\sum_{k\geq1}\ind{w_1\in\L_1}=1$. We hence conclude \eqref{eq:tailmoments}.

By Lemma \ref{tailLMp}, one obtains that
for any $a>0$,
\[
\e_1[L_1; L_1\geq ax, M_1\geq x]\sim \eta_1(a) x^{-(\kappa-1)},
\]
as $x\rightarrow\infty$. 
Let $P_\alpha(ax,x):=\e_1[L_1^\alpha; L_1\geq ax, M_1\geq x]$ for any $\alpha\geq 0$. One sees that for $x>0$ and $a>0$,
\begin{align*}
\frac{1}{ax}P_1(ax,x)-P_0(ax,x)=&\e_1[(\frac{L_1}{ax}-1); L_1\geq ax, M_1\geq x]\\
%=&\e_1\left[L_1\Big(\frac{1}{ax}-\frac{1}{L_1}\Big); L_1\geq ax, M_1\geq x\right]\\
=&\e_1\left[L_1\int_{ax}^{L_1}\frac{dy}{y^2}; L_1\geq ax, M_1\geq x\right]\\
%=&\e_1\left[L_1\int_{ax}^\infty \frac{dy}{y^2}; L_1\geq y, M_1\geq x\right]=\int_{ax}^\infty P_1(y,x)\frac{dy}{y^2}\\
=& \int_a^\infty \frac{1}{x}P_1(xu, x)\frac{du}{u^2}
\end{align*}
which implies
\[
P_0(ax,x)=\frac{P_1(ax,x)}{ax}-\frac{1}{x}\int_a^\infty P_1(xu,x)\frac{du}{u^2}.
\]
Note that for all $x>0$, $x^{\kappa-1}P_1(ux,x)\rightarrow \eta_1(u)$. Moreover, for $x$ sufficiently large and $u\geq a>0$,
\[
x^{\kappa-1}P_1(ux,x)\frac{1}{u^2}\leq \frac{x^{\kappa-1}P_1(ax,x)}{u^2}\leq \frac{1+\eta_1(a)}{u^2},
\] 
which is integrable on $[a,\infty)$. By dominated convergence theorem, for any $a>0$, as $x\rightarrow\infty$,
\[
x^{\kappa}\p_1(L_1\geq ax, M_1\geq x)=x^{\kappa}P_0(ax,x)\rightarrow \int_a^\infty \left(\eta_1(a)-\eta_1(u)\right)\frac{du}{u^2}.
\]
Because of Lemma \ref{tailsmallLM}, we also obtain that $\p_1(M_1\geq x)\sim x^{-\kappa}\int_0^\infty \left(\eta_1(0)-\eta_1(u)\right)\frac{du}{u^2}$.
\end{proof}

In the next subsections, we prove Lemma \ref{tailLMp} and \ref{tailsmallLM}. 

\subsection{Joint tail of $(L_1,M_1)$ under $\widehat{\p}^*_1$: proof of Lemma \ref{tailLMp}}

We are going to show that for any $a>0$, as $x\rightarrow\infty$,
\begin{equation}\label{tailLMstar}
\widehat{\p}^*_1\left(L_1\geq ax, M_1\geq x\right)\sim \eta_1(a) x^{-(\kappa-1)}.
\end{equation}
The main idea follows that of Section~4.2 of \cite{dR16} where the author shows that $\widehat{\p}_1^*(L_1\geq x)\sim \frac{\kappa}{\kappa-1}c_L x^{-\kappa+1}$ for $\kappa\in (1,2]$. In fact, this joint tail largely depends on the joint tail of $(W_\infty,\mathcal{M}_e)$ on the  marked environment ${\en^*}$, which will be stated and proved in the following. We also prove Proposition \ref{cvgLM} which will be needed to obtain \eqref{tailLMstar}.

\subsubsection{Joint tail of $(W_\infty,\mathcal{M}_e)$ under $\Q^*$}

By Theorem \ref{thmtailWM}, one has $\P(W_\infty\geq ax,\mathcal{M}_e\geq x)\sim \gamma(a)x^{-\kappa}$ with $\gamma(a)=c_\tM\E[(1\wedge \frac{\mathcal{W}_\infty^\tM}{a})^\kappa]$. The change of measures given in Section~\ref{changeofmeasureBRW} and the non-triviality of $W_\infty$ imply that
\begin{align*}
\Q^*(W_\infty\geq ax, \mathcal{M}_e\geq x)=&\Q(W_\infty\geq ax, \mathcal{M}_e\geq x)\\
=&\E[W_\infty \ind{W_\infty\geq ax, \mathcal{M}_e\geq x}]\\
=&ax\P(W_\infty\geq ax, \mathcal{M}_e\geq x)+x\int_a^\infty \P(W_\infty \geq xu, \mathcal{M}_e\geq x)du
\end{align*}
Note that for $x$ large enough and $u\geq a>0$
\[
x^\kappa  \P(W_\infty \geq xu, \mathcal{M}_e\geq x)\leq x^\kappa  \P(W_\infty \geq xu)\leq (C_0+1)u^{-\kappa}.
\]
By dominated convergence theorem, one obtains that
\[
\lim_{x\rightarrow\infty}x^{\kappa-1}\Q^*(W_\infty\geq ax, \mathcal{M}_e\geq x)=a\gamma(a)+\int_a^\infty \gamma(u)du.
\]
Some simple calculations yield that  for any $a>0$, as $x\rightarrow\infty$,
\begin{equation}\label{tailWMQ}
\Q^*(W_\infty\geq ax, \mathcal{M}_e\geq x)\sim \mu(a)x^{-\kappa+1}
\end{equation} with $\mu(a):=c_\tM\frac{\kappa}{\kappa-1}\E[\mathcal{W}_\infty^\tM(1\wedge \frac{\mathcal{W}_\infty^\tM}{a})^{\kappa-1}]$ for any $a>0$.

\subsubsection{Proof of Proposition \ref{cvgLM}}

Let us state the following result on the moments of $L_1$ and $M_1$.

\begin{lem}\label{Mom}
%We assume Assumptions \ref{cond1}, \ref{cond2}, \ref{cond3}, \ref{cond4}. 
If $\kappa\in(1,\infty)$, for any $\alpha\in[0,\kappa-1)$, there exists a constant $C_{\alpha}\in (0,\infty)$ such that for any $i\geq1$,
\begin{equation}\label{MomL}
\e_i\left[L_1^{1+\alpha}\right]\leq C_\alpha i^{1+\alpha};
\end{equation}
and that
\begin{equation}\label{MomM}
\e_i\left[M_1^{1+\alpha}\right] \leq C_\alpha i^{1+\alpha}.
\end{equation}
Further, if $\kappa=\infty$, \eqref{MomM} holds for all $\alpha\geq0$ and \eqref{MomL} holds also for $\alpha\in[0,1]$.
\end{lem}

We postpone the proof of this lemma in Section~\ref{lemRWRE}. Now we are ready to prove Proposition \ref{cvgLM}, mainly inspired by the arguments in Section~4.2.1 of \cite{dR16}.

\begin{proof}[Proof of Proposition \ref{cvgLM}]
For any $M\in(0,\infty)$ and for any $k\geq 0$, let $\mathcal{M}_e^{\leq k}:=\sup_{|u|\leq k}e^{-V(u)}$ and $\mathcal{M}_e^{\leq k, M}:=\mathcal{M}_e^{\leq k}\wedge M$. Similarly, write 
$W_k^M:=(\sum_{|u|=k}e^{-V(u)})\wedge M$ and $M_1^{\leq k}:=\max_{|u|\leq k, u\leq \mathcal{L}_1}\beta(u)$. And for any $u\in \B_1$, let 
\begin{equation}\label{eq:defL1uM1u}
L_1^{(u)}:=\sum_{v: v\geq u}\ind{v\in\L_1}\quad \text{and}\quad M_1^{(u)}:=\max_{v: v\geq u}\beta(v)\ind{v\leq \L_1}. 
\end{equation}
Observe that for any $t, s>0$ and $k\geq \ell\geq1$,
\begin{align}\label{upbdcvgLMloi}
&\e_n\left[e^{-t(\frac{L_1}{n}-W_\ell^M)-s(\frac{M_1}{n}-\mathcal{M}_e^{\leq \ell,M})}\right]\leq  \e_n\left[e^{-t(\frac{L_1}{n}-W_\ell^M)-s(\frac{M_1^{\leq k}}{n}-\mathcal{M}_e^{\leq \ell,M})}\right]\\
=&\e_n\left[e^{-s\frac{M_1^{\leq k}}{n}}\e_n\left[1\wedge e^{\left\{-\frac{t}{n}\Big(\sum_{|u|=k}\ind{u<\mathcal{L}_1}L_1^{(u)}+\sum_{|u|\leq k}\ind{u\in\L_1}\Big)\right\}}\Big\vert \mathcal{F}_k^\beta,\mathcal{F}^V_k \right]e^{tW_\ell^M+s\mathcal{M}_e^{\leq \ell,M}}\right]\nonumber\\
%\leq &\e_n\left[e^{-s\frac{M_1^{\leq k}}{n}}\e_n\left[1\wedge\exp\left(-\frac{t}{n}\Big(\sum_{|u|=k}\ind{u<\mathcal{L}_1}L_1^{(u)}\Big)\right)\Big\vert \mathcal{F}_k,\mathcal{F}^V_k\right]e^{tW_\ell^M+s\mathcal{M}_e^{\leq \ell,M}}\right]\\
\leq&\e_n\left[e^{-s\frac{M_1^{\leq k}}{n}}\left(\prod_{|u|=k, u<\mathcal{L}_1}\e_n\left[\exp\left(-\frac{t}{n}L_1^{(u)}\right)\Big\vert \mathcal{F}^\beta_k,\mathcal{F}^V_k\right]\wedge 1\right) e^{tW_\ell^M+s\mathcal{M}_e^{\leq \ell,M}}\right].\nonumber
\end{align}
As $e^{-x}\leq 1-x+x^{1+\alpha}\leq e^{-x+x^{1+\alpha}}$ for $\alpha\in(0,\kappa-1)$ and $x\geq0$, one has
\begin{align*}
\e_n\left[\exp\left(-\frac{t}{n}L_1^{(u)}\right)\Big\vert \mathcal{F}^\beta_k,\mathcal{F}^V_k\right]\wedge 1\leq &\e_n\left[1-\frac{t}{n}L_1^{(u)}+(\frac{t}{n}L_1^{(u)})^{1+\alpha}\vert \mathcal{F}^\beta_k,\mathcal{F}^V_k\right]\wedge1\\
=&\left(1-\frac{t}{n}\e_n[L_1^{(u)}\vert \mathcal{F}^\beta_k,\mathcal{F}^V_k]+(\frac{t}{n})^{1+\alpha}\e_n[(L_1^{(u)})^{1+\alpha}\vert \mathcal{F}^\beta_k,\mathcal{F}^V_k]\right)\wedge 1\\
\leq & \left(1-\frac{t}{n}\beta(u)+\left(\frac{t}{n}\right)^{1+\alpha}C_\alpha \beta(u)^{1+\alpha}\right)\wedge 1.
\end{align*}
Plugging it into \eqref{upbdcvgLMloi} yields that
\begin{align*}
&\e_n\left[e^{-t(\frac{L_1}{n}-W_\ell^M)-s(\frac{M_1}{n}-\mathcal{M}_e^{\leq \ell,M})}\right]\\
\leq& \e_n\left[e^{-s\frac{M_1^{\leq k}}{n}}\left(\prod_{|u|=k, u<\mathcal{L}_1} \left(1-\frac{t}{n}\beta(u)+\left(\frac{t}{n}\right)^{1+\alpha}C_\alpha \beta(u)^{1+\alpha}\right)\wedge 1\right)e^{tW_\ell^M+s\mathcal{M}_e^{\leq \ell,M}}\right]\\
\leq & \e_n\left[e^{-s\frac{M_1^{\leq k}}{n}}\left(e^{\left(-\sum_{|u|=k}t\frac{\beta(u)}{n}\ind{u<\mathcal{L}_1}+t^{1+\alpha}C_\alpha\sum_{|u|=k} (\frac{\beta(u)}{n})^{1+\alpha}\ind{u<\mathcal{L}_1}\right)}\wedge 1\right)e^{tW_\ell^M+s\mathcal{M}_e^{\leq \ell,M}}\right].
\end{align*}
Note that given the environment $\en$, for any $k\geq 0$ fixed, by the law of large number,
\[
\sum_{|u|=k}\frac{\beta(u)}{n}\ind{u<\L_1}\xrightarrow{in\ \p^\en} \sum_{|u|=k}e^{-V(u)}.
\]
One can refer to the proof of Proposition~6 \cite{dR16} for more details for this convergence. Moreover,
\[
\sum_{|u|=k}(\frac{\beta(u)}{n})^{1+\alpha}\xrightarrow{in\ \p^\mathcal{E}} \sum_{|u|=k}e^{-(1+\alpha)V(u)}%\xrightarrow[k\rightarrow\infty]{\P-a.s.}0,
\]
and
\[
\frac{M_1^{\leq k}}{n}=\max_{|u|\leq k} \frac{\beta(u)}{n}\ind{u<\mathcal{L}_1}\xrightarrow{in\ \p^\mathcal{E}} \max_{|u|\leq k}e^{-V(u)}=\mathcal{M}_e^{\leq k}.
\]
By dominated convergence theorem, 
\begin{multline}
\limsup_{n\rightarrow\infty}\e_n\left[e^{-t(\frac{L_1}{n}-W_\ell^M)-s(\frac{M_1}{n}-\mathcal{M}_e^{\leq \ell,M})}\right]\\
\leq \E\left[e^{-s\mathcal{M}_e^{\leq k}}\left(\exp\left(-t W_k+C_\alpha t^{1+\alpha}\sum_{|u|=k}e^{-(1+\alpha)V(u)}\right)\wedge 1\right)e^{tW_\ell^M+s\mathcal{M}_e^{\leq \ell,M}}\right]
\end{multline}
Notice that by definition of $\kappa$ and thanks to the many-to-one lemma, $\sum_{|u|=k}e^{-(1+\alpha)V(u)}$ converges towards $0$ almost surely, for $\alpha\in(0;\kappa-1)$. Letting $k\rightarrow\infty$ implies that
\begin{equation}
\limsup_{n\rightarrow\infty}\e_n\left[e^{-t(\frac{L_1}{n}-W_\ell^M)-s(\frac{M_1}{n}-\mathcal{M}_e^{\leq \ell,M})}\right]\leq \E\left[e^{-s(\mathcal{M}_e-\mathcal{M}_e^{\leq \ell, M})-t(W_\infty-W_\ell^M)}\right]
\end{equation}

On the other hand, let us define
\[
\widetilde{M}_1^{\leq k}:=\max\{M_1^{\leq k}, \max_{|u|=k, u<\mathcal{L}_1}\frac{1}{2}\beta(u)e^{V(u)}\}, \textrm{ and }
E_k:=\cap_{|u|=k, u<\mathcal{L}_1}\{M_1^{(u)}\leq \beta(u)e^{V(u)}/2\}.
\]
Observe that on $E_k$, $M_1\leq \widetilde{M}_1^{\leq k}$. Consequently,
\begin{align}\label{lwbdcvgLMloi}
&\e_n\left[e^{-t(\frac{L_1}{n}-W_\ell^M)-s(\frac{M_1}{n}-\mathcal{M}_e^{\leq \ell,M})}\right]\geq  \e_n\left[e^{-t(\frac{L_1}{n}-W_\ell^M)-s(\frac{\widetilde{M}_1^{\leq k}}{n}-\mathcal{M}_e^{\leq \ell,M})}{\bf 1}_{E_k}\right]\nonumber\\
\geq &\e_n\left[e^{-t(\frac{L_1}{n}-W_\ell^M)-s(\frac{\widetilde{M}_1^{\leq k}}{n}-\mathcal{M}_e^{\leq \ell,M})}\right]-e^{tM+sM}\p_n(E_k^c).
\end{align}
Here, note that by Markov inequality, for any $\alpha\in(0,\kappa-1)$,
\begin{align*}
\p_n(E_k^c)\leq&\e_n\left[\sum_{|u|=k, u<\mathcal{L}_1}\ind{M_1^{(u)}\geq \beta(u)e^{V(u)}/2}\right]\\
=&\e_n\left[\sum_{|u|=k, u<\mathcal{L}_1}\p_n\left(M_1^{(u)}\geq \beta(u)e^{V(u)}/2\Big\vert \mathcal{F}^\beta_k,\mathcal{F}_k^V\right)\right]\\
\leq &\e_n\left[\sum_{|u|=k, u<\mathcal{L}_1}2^{1+\alpha}\frac{\e_{\beta(u)}[M_1^{1+\alpha}]}{\beta(u)^{1+\alpha}e^{(1+\alpha)V(u)}}\right]
%\leq&\e_n\left[\sum_{|u|=k, u<\mathcal{L}_1} 2^{1+\alpha}\frac{C_\alpha (\beta(u))^{1+\alpha}}{(\beta(u))^{1+\alpha}e^{(1+\alpha)V(u)}}\right]\\
%\leq & c\e_n\left[\sum_{|u|=k}e^{-(1+\alpha)V(u)}\right]=o_k(1).
\end{align*}
By \eqref{MomM}, one sees that
\[
\p_n(E_k^c)\leq c\E\left[\sum_{|u|=k}e^{-(1+\alpha)V(u)}\right]=o_k(1).
\]
Write $L_1^{\leq k}:=\sum_{|u|\leq k}\ind{u\in\L_1}$. Going back to \eqref{lwbdcvgLMloi}, by Jensen's inequality, for any $k\geq\ell$,
\begin{align*}
&\e_n\left[e^{-t(\frac{L_1}{n}-W_\ell^M)-s(\frac{M_1}{n}-\mathcal{M}_e^{\leq \ell,M})}\right]\\
\geq &\e_n\left[\e_n\left[e^{-\frac{t}{n}\left(L_1^{\leq k}+\sum_{|u|=k}\ind{u<\mathcal{L}_1}L_1^{(u)}\right)-\frac{s}{n}\widetilde{M}_1^{\leq k}}\Big\vert \mathcal{F}^\beta_k,\mathcal{F}_k^V\right]e^{tW_\ell^M+s\mathcal{M}_e^{\leq \ell,M}}\right]+o_k(1)\\
\geq &\e_n\left[e^{-\frac{1}{n}\e_n\left[t\sum_{|u|=k}\ind{u<\mathcal{L}_1}L_1^{(u)}\Big\vert \mathcal{F}^\beta_k,\mathcal{F}_k^V\right]}e^{-\frac{t}{n}L_1^{\leq k}-\frac{s}{n}\widetilde{M}_1^{\leq k}+tW_\ell^M+s\mathcal{M}_e^{\leq \ell,M}}\right]+o_k(1)\\
=& \e_n\left[e^{-\frac{t}{n}\sum_{|u|=k}\ind{u<\mathcal{L}_1}\beta(u)}e^{-\frac{t}{n}L_1^{\leq k}-\frac{s}{n}\widetilde{M}_1^{\leq k}+tW_\ell^M+s\mathcal{M}_e^{\leq \ell,M}}\right]+o_k(1)
\end{align*}
Observe that clearly, under $\p_n$, $\frac{M_1^{\leq k}}{n}\geq 1$. Note also that given the environment $\en$, 
\[
\max_{|u|=k, u<\mathcal{L}_1}\frac{1}{2}\frac{\beta(u)}{n}e^{V(u)}\xrightarrow{in\ \p_n^\mathcal{E}} \max_{|u|=k} \frac{1}{2}e^{-V(u)}e^{V(u)}=\frac{1}{2}
\]
and $\frac{M_1^{\leq k}}{n}\xrightarrow{in\ \p_n^\mathcal{E}} \max_{|u|\leq k}e^{-V(u)}=\mathcal{M}_e^{\leq k,M}$. So,
\[
\frac{\widetilde{M}_1^{\leq k}}{n}\xrightarrow{in\ \p_n^\mathcal{E}} \mathcal{M}_e^{\leq k,M}.
\]
Besides, $\frac{L_1^{\leq k}}{n}\xrightarrow{in\ \p_n^\mathcal{E}}0 $ according to the arguments of the proof of Proposition~6 of~\cite{dR16}. Again, by dominated convergence theorem,
\[
\liminf_{n\rightarrow\infty}\e_n\left[e^{-t(\frac{L_1}{n}-W_\ell^M)-s(\frac{M_1}{n}-\mathcal{M}_e^{\leq \ell,M})}\right]\geq \E\left[e^{-t W_k-s \M_e^{\le k, M}+tW_\ell^M+s\M_e^{\le \ell, M}}\right]+o_k(1)
\]
Letting $k\rightarrow\infty$ implies that
\begin{equation}
\liminf_{n\rightarrow\infty}\e_n\left[e^{-t(\frac{L_1}{n}-W_\ell^M)-s(\frac{M_1}{n}-\mathcal{M}_e^{\leq \ell,M})}\right]\geq \E\left[e^{-s(\mathcal{M}_e-\mathcal{M}_e^{\leq \ell, M})-t(W_\infty-W_\ell^M)}\right].
\end{equation}
We hence deduce the convergence in law of $(\frac{L_1}{n}-W_\ell^M, \frac{M_1}{n}-\mathcal{M}_e^{\leq \ell,M})$ towards $(W_\infty-W_\ell^M, \M_e-\M_e^{\le\ell,M})$. Repeating the arguments in \cite{dR16}, we then conclude from Lemma \ref{Mom} that 
\[
\lim_{n\rightarrow\infty}\e_n\left[|\frac{L_1}{n}-W_\infty|^{1+\alpha}+|\frac{M_1}{n}-\M_e|^{1+\alpha}\right]=0.
\]
\end{proof}

\subsubsection{Proof of Lemma \ref{tailLMp}}\label{jointtailLMp}
Let us consider $(\T, \beta, (w_n)_{n\geq0})$ defined on the biased environment ${\en^*}=(\T, V, (w_n)_{n\geq0})$.  Note that under $\widehat{\p}_1^*$, $\{\beta(w_n); n\ge 0\}$ is a recurrent Markov chain started from $1$. Let
\[
\widehat{\tau}_1:=\min\{k\geq1: \beta(w_k)=1\},\qquad \sigma_A:=\min\{k\geq0: \beta(w_k)>A\}, \forall A\geq1,
\]
and recall from~\eqref{eq:defL1uM1u} the notation
\[
L_1^{(u)}:=\sum_{v: u\leq v}\ind{v\in\L_1}, \quad M_1^{(u)}:=\max_{v: u\leq v}\beta(v)\ind{v\leq\L_1}, \forall u\in\B_1.
\]
Let also
\begin{equation*}
W_\infty^{(u)}:=\lim_{n\to\infty}\sum_{v: u\leq v, |v|=|u|+n}\e^{-(V(v)-V(u))}, \quad \mathcal{M}_e^{(u)}:=\max_{v: u\leq v}e^{-(V(v)-V(u))}, \forall u\in\T.
\end{equation*}
Moreover, for any $u\in\T$, let
\[
\beta^{\sigma_A}(u):=\sum_{k=0}^{\sigma_A-1}(\beta^1_k(u)+\beta^2_k(u)).
\]
Clearly, $\beta(w_{\sigma_A})=\beta^{\sigma_A}(w_{\sigma_A})+1$ and $\beta^{\sigma_A}(v)=0$ for any $v\geq w_{\widehat{\tau}_1}$. The decomposition along the spine $(w_k; k\geq0)$ gives
\[
L_1=\sum_{k=1}^{\widehat{\tau}_1}\sum_{u\in\Omega(w_k)}L_1^{(u)}+1,\textrm{ and } M_1=\max\{M_1^{(u)}; u\in\cup_{k=1}^{\widehat{\tau}_1}\Omega(w_k)\}\vee \max\{\beta(w_k); 1\leq k\leq \widehat{\tau}_1\}.
\]

Similarly as in Section~4.2.2 of \cite{dR16}, we will study
\begin{equation*}
P_1(ax,x)=\widehat{\p}_1^*(L_1\geq ax, M_1\geq x)
\end{equation*}
for $a\geq0$, $x\gg1$ and $\kappa\in(1,\infty)$ in three steps. In the end, we will obtain that $P_1(ax,x)\approx \widehat{\p}^*_1(\beta^{\sigma_A}(w_{\sigma_A})W_\infty^{(w_{\sigma_A})}\geq ax,\beta^{\sigma_A}(w_{\sigma_A}) \mathcal{M}_e^{(w_{\sigma_A})}\geq x,\ \sigma_A<\widehat{\tau}_1)$ where $(W_\infty^{(w_{\sigma_A})}, \mathcal{M}_e^{(w_{\sigma_A})})$ is independent of $\beta^{\sigma_A}(w_{\sigma_A})$ and distributed as $(W_\infty, \mathcal{M}_e)$ under $\Q^*$. This implies \eqref{tailLMstar}. 

Let us first state some facts which will be used in the proof.
\begin{lem}\label{bdness}
\begin{enumerate}
\item If $\kappa\in(1,\infty)$, for any $A\geq 1$ fixed,
\begin{equation}\label{MombdA}
\widehat{\e}_1^*\left[\left(\sum_{k=1}^{\widehat{\tau}_1}\ind{\beta(w_{k-1})<A}\sum_{u\in\Omega(w_k)}L_1^{(u)}\right)^{\kappa-1}\right]<\infty.
\end{equation}
\item If $\kappa\in(1,\infty)$, for any $A\geq 1$ fixed,
\begin{equation}\label{Mombdkappa}
\widehat{\e}_1^*\left[\left(\beta(w_{\sigma_A})\right)^{\kappa-1}\ind{\sigma_A<\widehat{\tau}_1}\right]\in(0,\infty).
\end{equation}
\item For $(\zeta_i)_{1\leq i\leq n}$ i.i.d. random variables with $\E[\zeta_1]=0$ and $\E[|\zeta_1|^\alpha]<\infty$ for some $\alpha\geq1$, then there exists $C_1(\alpha)>0$ depending only on $\alpha$, such that
\begin{equation}\label{MomofSum}
\E\left[|\sum_{i=1}^n \zeta_i|^\alpha\right]\leq 
\begin{cases}
C_1(\alpha) n^{\alpha/2}\E[|\zeta_1|^\alpha], \textrm{ if }\alpha\geq 2\\
2 n \E[|\zeta_1|^\alpha], \textrm{ if }\alpha\in[1,2].
\end{cases}
\end{equation}
\item If under $\P$, $(\zeta_i)_{1\leq i\leq N}$ is a random vector with negative multinomial distribution of parameters $n$ and $(\frac{A_i}{1+\sum_{j=1}^N A_j})_{1\leq i\leq N}$, then for any $\alpha\in[1,\infty)$, there exists $C_2(\alpha)>0$ depending only on $\alpha$, such that for any $(z_1,\cdots,z_N)\in \R_+^{N}$,
\begin{equation}\label{NBlaw}
\E\left[|\sum_{i=1}^N z_i \zeta_i - n\sum_{i=1}^N A_i z_i |^\alpha\right]\leq 
%C_\alpha n^{\alpha/2\vee 1} \left[(\sum_{i=1}^N  A_i z_i)^\alpha+ \sum_{i=1}^N A_i z_i^{\alpha}\right] \textrm{ if } \alpha\in[1,2]\\
C_2(\alpha) n^{\alpha/2\vee 1} \left[\sum_{k=0}^{\lfloor \alpha-1\rfloor}(\sum_{i=1}^NA_i z_i)^k(\sum_{i=1}^N A_i z_i^{\alpha-k})+(\sum_{i=1}^N  A_i z_i)^\alpha\right]. %\textrm{ if } \alpha>2
\end{equation}
\end{enumerate}

\end{lem}

Equation~\eqref{MomofSum} collects two inequalities from \cite{DFJ68} for $\alpha\geq2$ and from (2.6.20) of \cite{Pe95} for $\alpha\in[1,2]$. Note that it is immediate from \eqref{Mombdkappa} that for any $A\geq 1$ fixed,
\begin{equation}
K_A:=\widehat{\e}_1^*\left[\left(\beta^{\sigma_A}(w_{\sigma_A})\right)^{\kappa-1}\ind{\sigma_A<\widehat{\tau}_1}\right]<\infty.
\end{equation}
In \eqref{NBlaw}, we finally see that by convexity, $(\sum_{i=1}^N A_i z_i)^k\leq (\sum_{i=1}^N A_i)^{k-1}(\sum_{i=1}^NA_i z_i^{k})$ for any $k\geq1$. The proof of Lemma \ref{bdness} is postponed in Section~\ref{lemRWRE}.

Let us start analysing $P_1(ax,x)$. We let $\kappa':=\kappa-1$. 

\paragraph{Step1} Let 
\[
L_1^{>\sigma_A}:=\sum_{k=\sigma_A+1}^{\widehat{\tau}_1}\sum_{u\in\Omega(w_k)}L_1^{(u)}, %\textrm{ and } M_1^{> \sigma_A}=\max\{M_1^{(u)}; u\in\cup_{k=1+\sigma_A}^{\widehat{\tau}_1}\Omega(w_k)\}\vee \max\{\beta(w_k)\ind{\sigma_A\leq k\leq \widehat{\tau}_1}\}.
\]
and $M_1^{> \sigma_A}:=\max\{M_1^{> \sigma_A, \dagger}, M_1^{> \sigma_A,\star}\}$ with
\begin{equation}\label{eq:defM1geqsigma}
M_1^{> \sigma_A, \dagger}:=\max\{M_1^{(u)}; u\in\cup_{k=\sigma_A}^{\widehat{\tau}_1-1}\Omega(w_{k+1})\}, \quad M_1^{> \sigma_A,\star}:=\max\{\beta(w_k)\ind{\sigma_A\leq k\leq \widehat{\tau}_1}\}.
\end{equation}
\begin{lem}\label{lemma:LMstep1}
For any $\epsilon>0$, $A\geq 1$ and as $x\rightarrow\infty$,
\begin{multline}\label{LMstep1}
\widehat{\p}^*_1(L_1^{>\sigma_A}\geq ax, M_1^{> \sigma_A}\geq x,\ \sigma_A<\widehat{\tau}_1)\leq P_1(ax,x)\\
\leq\widehat{\p}^*_1(L_1^{>\sigma_A}\geq (a-\epsilon)x, M_1^{> \sigma_A}\geq x,\ \sigma_A<\widehat{\tau}_1)+o(x^{-\kappa'}).
\end{multline}
\end{lem}
\begin{proof}
Notice that 
\begin{equation*}
L_1-L_1^{>\sigma_A}= 1+\sum_{k=1}^{\sigma_A}\sum_{u\in\Omega(w_k)}L_1^{(u)}.
\end{equation*}
The proof of Lemma~16 of~\cite{dR16} applies to show that for any $\epsilon>0$, $A>0$, 
\begin{equation}\label{LMstep1eq1}
\widehat{\p}^*_1\big(L_1>\epsilon x,\widehat{\tau}_1\leq \sigma_A\big)=o(x^{-\kappa'})\quad \text{and}\quad \widehat\p^*_1\big(\sum_{k=1}^{\sigma_A}\sum_{u\in\Omega(w_k)}L_1^{(u)}>\epsilon x,\sigma_A<\widehat{\tau}_1\big)=o(x^{-\kappa'}), 
\end{equation}
excepted that for $\kappa>2$ we use \eqref{MombdA} of Lemma~\ref{bdness} to show the finiteness of the quantity $\widehat{\e}_1^*\left[\left(\sum_{k=1}^{\widehat{\tau}_1}\ind{\beta(w_{k-1})<A}\sum_{u\in\Omega(w_k)}L_1^{(u)}\right)^{\kappa'}\right]$. 
Now if we show that 
\begin{equation}\label{eq:M1o}
\widehat\p^*_1\big(M_1-M_1^{> \sigma_A}>\epsilon x,\sigma_A<\widehat{\tau}_1\big)=o(x^{-\kappa'}), 
\end{equation}
then the union bound together with~\eqref{LMstep1eq1} will conclude the lemma. To prove~\eqref{eq:M1o}, we simply adjust the proof of Lemma~16 of~\cite{dR16}. First, Markov's inequality yields
\begin{align*}
\widehat\p^*_1\big(M_1&-M_1^{> \sigma_A}>\epsilon x,\sigma_A<\widehat{\tau}_1\big)\leq (\epsilon x)^{-\kappa'}\widehat{\e}^*_1\left[ \big(M_1-M_1^{> \sigma_A}\big)^{\kappa'} \ind{\sigma_A<\widehat{\tau}_1}\ind{M_1-M_1^{> \sigma_A}>\epsilon x} \right]\\
 &\leq (\epsilon x)^{-\kappa'}\widehat{\e}^*_1\left[ \max_{1\leq k \leq \tauh_1}(\ind{\beta(w_{k-1})<A}\max(\beta(w_{k-1})^{\kappa'},\max_{u\in\Omega(w_k)} (M_1^{(u)})^{\kappa'}) \ind{M_1-M_1^{> \sigma_A}>\epsilon x} \right]\\
&\leq (\epsilon x)^{-\kappa'} \widehat{\e}^*_1\left[\Big(A^{\kappa'}+\big(\sum_{1\leq k \leq \tau_1}\ind{\beta(w_{k-1})<A}\sum_{u\in\Omega(w_k)}M_1^{(u)}\big)^{\kappa'}\Big) \ind{M_1-M_1^{> \sigma_A}>\epsilon x} \right]. 
\end{align*}
But since the indicator function in this last expectation tends to $0$ a.s.\ when $x\to\infty$, it is enough to show the finiteness of $\widehat{\e}^*_1\left[\big(\sum_{k=1}^{\tau_1}\ind{\beta(w_{k-1})<A}\sum_{u\in\Omega(w_k)}M_1^{(u)}\big)^{\kappa'}\right]$ to conclude. Proving this is similar to proving \eqref{MombdA}, with $M_1^{(u)}$ instead of $L_1^{(u)}$. Indeed the only fact we need on $L_1^{(u)}$ in this proof is that $\widehat{\e}^*_1[(L_1^{(u)})^{\kappa'}|\beta(u)]\leq C_{\kappa'}\beta(u)^{\kappa'}$ for a certain constant $C_{\kappa'}$, a domination which is also satisfied by the $M_1^{(u)}$ according to Lemma~\ref{Mom}. 
\end{proof}

\paragraph{Step 2}

Recall the notation of the beginning of this Subsection~\ref{jointtailLMp}. Let also for any $u\in\T$, 
\begin{equation*}
\Delta V(u):=V(u)-V(\pa{u}).
\end{equation*}
The aim of this step is to get the following lemma. 
\begin{lem}\label{lemma:LMstep2}
For all $\epsilon>0$ small enough, for all $A\geq A_\epsilon$ large enough, for all $x>0$ large enough, 
\begin{multline}\label{LMstep2}
\widehat{P}_1((a+3\epsilon)x, (1+5\epsilon)x;A)-8\epsilon x^{-\kappa'}K_A+o(x^{-\kappa'})\\
\leq P_1(ax,x)\leq\widehat{P}_1((a-3\epsilon)x, (1-5\epsilon)x; A)+8\epsilon x^{-\kappa'}K_A+o(x^{-\kappa'}),
\end{multline}
where $\widehat{P}_1(ax,x;A):=\widehat{\p}^*_1(\beta^{\sigma_A}(w_{\sigma_A})W_\infty^{(w_{\sigma_A})}\geq ax,\beta^{\sigma_A}(w_{\sigma_A}) \mathcal{M}_e^{(w_{\sigma_A})}\geq x,\ \sigma_A<\widehat{\tau}_1)$.
\end{lem}
Notice that the quantity $K_A=\widehat{\e}_1^*[(\beta^{\sigma_A}(w_{\sigma_A}))^{\kappa'}\ind{ \sigma_A<\widehat{\tau}_1}]$ for any $A\geq1$ is finite according to Lemma~\ref{bdness}.

Actually, we will divide the proof in two parts, one dealing with $L_1$ and the other with $M_1$. The following two lemmas will immediately yield Lemma~\ref{lemma:LMstep2}. 

\begin{lem}\label{lemma:LMstep2part1}
For all $\epsilon>0$ small enough, for all $A\geq A_\epsilon$ large enough, for all $x>0$ large enough, 
\begin{equation}\label{eq:ecartL1Winf}
\widehat{\p}^*_1\big( |L_1-\ty^{\sigma_A}(w_{\sigma_A})W_\infty^{(w_{\sigma_A})}|>3\epsilon x,\ \sigma_A<\widehat{\tau}_1 \big)\leq 3\epsilon x^{-\kappa} K_A +o(x^{-\kappa'}). 
\end{equation}
\end{lem}

\begin{lem}\label{lemma:LMstep2part2}
For all $\epsilon>0$ small enough, for all $A\geq A_\epsilon$ large enough, for all $x>0$ large enough, 
\begin{equation}\label{eq:ecartM1Winf}
\widehat{\p}^*_1\big( |M_1-\ty^{\sigma_A}(w_{\sigma_A})\mathcal{M}_e^{(w_{\sigma_A})}|>5\epsilon x,\ \sigma_A<\widehat{\tau}_1 \big)\leq 5\epsilon x^{-\kappa} K_A +o(x^{-\kappa'}). 
\end{equation}
\end{lem}

\begin{proof}[Proof of Lemma~\ref{lemma:LMstep2part1}]
For $\kappa\leq 2$, Lemmas~16,~18,~19 and~20 of~\cite{dR16} put together yield the result. Actually the proofs of Lemmas~16,~18,~19 and~20 of~\cite{dR16} can be adjusted to the case $\kappa> 2$. Let us detail this. 

$\bullet$ For Lemma~16, it is the lemma we just proved in Step~1, Lemma~\ref{lemma:LMstep1}. 

$\bullet$ Let us now follow the lines of Lemma~18. Let $\Omega:=\cup_{k=1}^{\tauh_1}\Omega(w_k)$, and $(Z_u)_{u\in\Omega}$ a family of i.i.d.\ random variables admitting a finite moment of order $\kappa'$. This ensures that the $\sum_{u\in\Omega(w_k)}e^{-\Delta V(u)}Z_u$ also have a finite moment of order $\kappa'$:
\begin{align}\label{eq:momentZu}
\widehat{\e}^*_1\left[\big(\sum_{u\in\Omega(w_k)}e^{-\Delta V(u)}Z_u\big)^{\kappa'}\right]&\leq \widehat{\e}^*_1\left[\sum_{u\in\Omega(w_k)}e^{-\Delta V(u)}(Z_u)^{\kappa'}\big(\sum_{u\in\Omega(w_k)}e^{-\Delta V(u)}\big)^{\kappa'-1}\right]\nonumber\\
&=\e_1\left[\sum_{|u|=1}e^{-V(u)}\big(\sum_{v\neq u} e^{-V(v)}\big)^{\kappa'}\right]\widehat{\e}^*_1\big[Z^{\kappa'}\big]<\infty, 
\end{align}
where the first inequality is obtained by convexity of $x\mapsto x^{\kappa'}$, and where $Z$ is a generic random variable of same law as the $(Z_u)_{u\in\Omega}$. Now we replace the inequality~(4.26) by the following (using~\eqref{NBlaw} with $N=1$ and $z_1=1$):
\begin{align*}
\widehat{\e}_1^*[|\beta^{\sigma_A}(w_\ell)-\beta^{\sigma_A}(w_{\ell-1})e^{-(V(w_{\ell})-V(w_{\ell-1}))}|^{\kappa'}\mid \beta^{\sigma_A}&(w_{\ell-1}),\Delta V(w_\ell)]\\ 
&\leq C_3 (\beta^{\sigma_A}(w_{\ell-1}))^{\kappa'/2}(e^{-\Delta V(w_{\ell})}+e^{-\kappa'\Delta V(w_{\ell})})
\end{align*}
for a  certain constant $C_3>0$, and we can reach the same conclusion on $f_A(x)$. % Plugging this into~(3.27) allows to dominate $f_A$: for all $x$ large enough, 
%\begin{align*}
%f_A(x)&\leq C'(\epsilon x)^{-\kappa}\widehat{\e}^*[\sum_{\ell\leq \sigma_A+1}(\ell-\sigma_A)^{2\kappa'}\ty^{\sigma_A}(w_{\ell-1})^{\kappa/2}(e^{-\Delta V(w_{\ell})}+e^{-\kappa'\Delta V(w_{\ell})})\ind{\sigma_A<\widehat{\tau}_1}]\\
%&=C'(\epsilon x)^{-\kappa}\widehat{\e}^*[\sum_{\ell=\sigma_A+1}^{\widehat{\tau}_1+1}(\ell-\sigma_A)^{2\kappa'}\ty^{\sigma_A}(w_{\ell-1})^{\kappa/2}\ind{\sigma_A<\widehat{\tau}_1}](\psi(2)+\psi(\kappa)). 
%\end{align*} 
%When conditioning on $\mathcal{F}^\beta_{\sigma_A}$, we get from equation~equation~\eqref{Domsum} that this last quantity is smaller than
%\begin{equation*}
%C''(\epsilon x)^{-\kappa}\widehat{\e}^*[\ty^{\sigma_A}(w_{\sigma_A})^{\kappa/2}\ind{\sigma_A<\widehat{\tau}_1}] 
%\end{equation*}
%for a certain constant $C''$, which allows us to get the conclusion of~(3.29) and ~(3.30) of~\cite{dR16}. 
Finally, we have to replace~(4.30) by 
\begin{align}\label{eq:domsum}
&\widehat{\e}^*_1\big[ |\sum_{u\in\Omega(w_k)} \big( \ty^{\sigma_A}(u)-\ty^{\sigma_A}(w_{k-1})e^{-\Delta V(u)}\big)Z_u |^{\kappa'} \mid \ty^{\sigma_A}(w_{k-1})\big]\\ 
\leq &C_4 (\ty^{\sigma_A}(w_{k-1}))^{\kappa'/2}\widehat{\e}^*_1\big[\sum_{k=0}^{\fl{\kappa'-1}}(\sum_{u\in\Omega(w_k)} e^{-\Delta V(u)}Z_u)^k(\sum_{u\in\Omega(w_k)} e^{-\Delta V(u)}Z_u^{\kappa'-k})+(\sum_{u\in\Omega(w_k)} e^{-\Delta V(u)}Z_u)^{\kappa'}\big]\nonumber\\
\leq &C_4 (\ty^{\sigma_A}(w_{k-1}))^{\kappa'/2}\widehat{\e}^*_1\left[\sum_{k=0}^{\fl{\kappa'-1}}(\sum_{u\in\Omega(w_k)} e^{-\Delta V(u)})^{k+1}+(\sum_{u\in\Omega(w_k)} e^{-\Delta V(u)})^{\kappa'}\right]\times\max_{0\leq k\leq \fl{\kappa'}}\widehat{\e}^*_1[Z^{k}]\widehat{\e}^*[Z^{\kappa'-k}].\nonumber
\end{align}
The second inequality comes from~\eqref{NBlaw}. The third one comes from the convexity of $x\mapsto x^k$ (as explained in the remark after Lemma~\ref{bdness}), and the linearity of the expectation. The finiteness of the first expectation comes from the many-to-one lemma and assumption~\ref{cond4}. We conclude following the lines of Lemma~18 of~\cite{dR16}, by use of~\eqref{eq:dompath}. \\

$\bullet$ The proof of Lemma~19 of~\cite{dR16} adjusts to the case $\kappa>2$, as its arguments boil down to those of the proof of Lemma~18 that we just showed adjust to the case $\kappa>2$. 

%$\bullet$ In order to adjust the proof of Lemma~19 of~\cite{dR16} to the case $\kappa>2$, the only change is in the proof of equation~(3.38) ; we rather dominate $f$ as follows: 
%\begin{align*}
%f(x)&\leq \widehat{\p}^*_1\Big( \sum_{k\geq 1}\big|\sum_{u\in\Omega(w_k)}\big(\ty_0^1(u)-\ty_0^1(w_{k-1})e^{-\Delta V(u)}\big) Z_u \big|>\sum_{k\geq 1}\frac{6}{\pi^2}\frac{1}{k^2} \frac{x}{2}\Big)\\
%&\leq  \sum_{k\geq 1}\widehat{\p}^*_1\Big(\big|\sum_{u\in\Omega(w_k)}\big(\ty_0^1(u)-\ty_0^1(w_{k-1})e^{-\Delta V(u)}\big) Z_u \big|>\frac{6}{\pi^2}\frac{1}{k^2} \frac{x}{2}\Big)\\
%&\leq C_5x^{-\kappa'}\sum_{k\geq 1}k^{2\kappa'}\widehat{\e}^*_1\Big[ \big|\sum_{u\in\Omega(w_k)}(\ty_0^1(u)-\ty_0^1(w_{k-1})e^{-\Delta V(u)})  Z_u\big|^{\kappa'} \Big]
%\end{align*}
%Using the same reasoning as in~\eqref{eq:domsum}, we get that this last quantity is smaller than
%\begin{equation*}
%C_6x^{-\kappa'}\sum_{k\geq 1}k^{2\kappa'}\widehat{\e}^*_1\big[ \ty_0^1(w_{k-1})^{\kappa'/2}\big].
%\end{equation*}
%As $\ty_0^1(w_{k-1})=0$ for $k\geq \widehat{\tau}_1$, we get that this is finally smaller than
%\begin{equation*}
%C_6x^{-\kappa'}\widehat{\e}^*_1\big[\sum_{k=1}^{\widehat{\tau}_1}k^{2\kappa'} \ty_0^1(w_{k-1})^{\kappa'/2} \big]
%\end{equation*}
%which is finite according to equation~\eqref{eq:dompath} ($\ty_0^1$ being dominated by $\ty$). 

$\bullet$ Finally, let us deal with Lemma~20 of~\cite{dR16}. The finiteness $\widehat{\e}^*_1\big[ \big(\sum_{u\in\Omega(w_k)} e^{-\Delta V(u)}Y_u\big)^{\kappa'} \big]$ is obtained similarly to~\eqref{eq:momentZu}. The last point is equation (4.37), which can be dominated as explained in the lines that follow it, let appart that we use the convexity inequality rather than Jensen's inequality. 
\end{proof}

\begin{proof}[Proof of Lemma~\ref{lemma:LMstep2part2}]
First, Lemma~\ref{lemma:LMstep1} yields that 
\begin{equation}\label{eq:domM11}
\widehat{\p}^*_1\big( \max_{1\leq k <\sigma_A}\beta(w_k)\vee\max_{1\leq k\leq \sigma_A,\ u\in\Omega(w_k)} M_1^{(u)}>\epsilon x,\sigma_A<\widehat{\tau}_1 \big)=o(x^{-\kappa'}). 
\end{equation}
Recall from~\eqref{eq:defM1geqsigma} the notation
\begin{equation*}
M_1^{> \sigma_A, \dagger}:=\max\{M_1^{(u)}; u\in\cup_{k=\sigma_A}^{\widehat{\tau}_1-1}\Omega(w_{k+1})\}, \quad M_1^{> \sigma_A,\star}:=\max\{\beta(w_k); \sigma_A\leq k\leq \widehat{\tau}_1\}.
\end{equation*}
We will first deal with vertices outside the spine. We first want to prove that
\begin{align}\label{eq:ecartM1Winf}
\widehat{\p}^*_1\big( |M_1^{>\sigma_A,\dagger}-(\beta^{\sigma_A}(w_{\sigma_A})\max_{\sigma_A+1\leq k\leq \tauh_1,\ u\in\Omega(w_k)} e^{-(V(u)-V(w_{\sigma_A})} \mathcal{M}_e^{(u)})|>3\epsilon x,\ \sigma_A<\widehat{\tau}_1 \big)\leq& 3\epsilon x^{-\kappa} K_A. 
\end{align}
As we have shown before, Lemmas~16,~18,~19 and~20 of~\cite{dR16} apply to any $\kappa>1$ for $L_1$. We will adjust them to get the same dominations on $M_1$. \\

$\bullet$ Since $\ty(u)-\ty^{\sigma_A}(u)\geq 0$ for any $u\in\T$, we have
\begin{equation*}
\sum_{k=\sigma_A+1}^{\widehat{\tau}_1}\sum_{u\in\Omega(w_k)}(\ty(u)-\ty^{\sigma_A}(u))\mathcal{M}_e^{(u)}\geq \max_{\sigma_A+1\leq k\leq \widehat{\tau}_1}\max_{u\in\Omega(w_k)}(\ty(u)-\ty^{\sigma_A}(u))\mathcal{M}_e^{(u)}, 
\end{equation*}
and therefore the proof of Lemma~19 of~\cite{dR16} (with $Z_u=\mathcal{M}_e^{(u)}$) yields
\begin{equation}\label{eq:domM12}
\widehat{\p}^*_1\left(\max_{\sigma_A+1\leq k\leq\widehat{\tau}_1}\max_{u\in\Omega(w_k)}(\ty(u)-\ty^{\sigma_A}(u))\mathcal{M}_e^{(u)}>\epsilon x,\quad \sigma_A<\widehat{\tau}_1\right)\leq \epsilon x^{-\kappa'}K_A. 
\end{equation}

$\bullet$ Lemma~20 of~\cite{dR16} can also be reformulated. Indeed, it is shown in the proof that the two terms of the sum in the right-part of the inequality~(4.34) can be made smaller than $\epsilon x^{-\kappa'}K_A$ for $A$ and $x$ large enough. Replacing $(L_1^{(u)},W_{\infty}^{(u)})$ by $(M_1^{(u)},\mathcal{M}_e^{(u)})$ (which satisfies the same hypotheses, as shown in Proposition~\ref{cvgLM}) does not change the result. As a sum of positive terms is larger than their maximum, we get that
\begin{align}\label{eq:domM13}
&\widehat{\p}^*_1\left( \max_{\sigma_A+1\leq k\leq \widehat{\tau}_1}\max_{u\in\Omega(w_k)}|\ty(u)\mathcal{M}_e^{(u)}-M_1^{(u)}|>\epsilon x,\sigma_A<\widehat{\tau}_1 \right)\nonumber\\ 
&\leq \widehat{\p}^*_1\left( \max_{\sigma_A+1\leq k\leq \widehat{\tau}_1}\max_{u\in\Omega(w_k)}\ty^{\sigma_A}(u)|\mathcal{M}_e^{(u)}-\frac{M_1^{(u)}}{\ty(u)}|>\epsilon x/2,\sigma_A<\widehat{\tau}_1 \right)\nonumber\\
&+\widehat{\p}^*_1\left( \max_{\sigma_A+1\leq k\leq \widehat{\tau}_1}\max_{u\in\Omega(w_k)}(\ty(u)-\ty^{\sigma_A}(u))|\mathcal{M}_e^{(u)}-\frac{M_1^{(u)}}{\ty(u)}|>\epsilon x/2,\sigma_A<\widehat{\tau}_1 \right)\nonumber\\
&\leq \epsilon x^{-\kappa'}K_A
\end{align} 
for $A$ and $x$ large enough. \\

$\bullet$ Finally, it remains to adjust Lemma~18. We will show that for $A$ and $x$ large enough, 
\begin{equation}\label{eq:domM14}
\widehat{\p}^*_1\left( |M_\eqref{eq:domM14}|>\epsilon x,\sigma_A<\ \widehat{\tau}_1 \right)
\leq \epsilon x^{-\kappa'}K_A. 
\end{equation}
where 
\[
M_\eqref{eq:domM14}:=\left(\max_{\sigma_A+1\leq k\leq \widehat{\tau}_1}\max_{u\in\Omega(w_k)}\ty^{\sigma_A}(u)\mathcal{M}_e^{(u)}\right)-\left(\ty^{\sigma_A}(w_{\sigma_A})\max_{\sigma_A+1\leq k\leq \tauh_1}\max_{\ u\in\Omega(w_k)}e^{-(V(u)-V(w_{\sigma_A}))}\mathcal{M}_e^{(u)}\right).
\]
To prove this, we follow the lines of the proof and replace the sums by maxima. The only delicate point is when bounding $g_A$ in equation~(4.29). We proceed as follows. We let
\begin{equation*}
g_A(x):=\widehat{\p}^*_1\left( \max_{k\geq \sigma_A+1}\max_{u\in\Omega(w_k)}|(\ty^{\sigma_A}(u)-\ty^{\sigma_A}(w_{k-1})e^{-\Delta V(u)})Z_u|>\frac{\epsilon}{2}x,\sigma_A<\widehat{\tau}_1 \right). 
\end{equation*}
If $\kappa\geq 2$, Markov's inequality yields
\begin{align*}
g_A(x)&\leq (\frac{\epsilon}{2}x)^{-\kappa'}\widehat{\e}^*_1\left[ \sum_{k\geq \sigma_A+1}\sum_{u\in\Omega(w_k)}|(\ty^{\sigma_A}(u)-\ty^{\sigma_A}(w_{k-1})e^{-\Delta V(u)})Z_u|^{\kappa'}\ind{\sigma_A<\widehat{\tau}_1} \right]\\
 &\leq(\frac{\epsilon}{2}x)^{-\kappa'}\widehat{\e}^*_1\left[ \sum_{k\geq \sigma_A+1} (\ty^{\sigma_A}(w_{k-1}))^{\kappa'/2}\ind{\sigma_A<\widehat{\tau}_1} \right]\widehat{\e}^*_1[(Z_u)^{\kappa'}]\\
 &\qquad\qquad\qquad\times C_7\widehat{\e}^*_1\left[(\sum_{u\in\Omega(w_k)}e^{-\Delta V(u)})+(\sum_{u\in\Omega(w_k)}e^{-\Delta V(u)})^{\kappa'}\right], 
\end{align*}
where the second inequality is due to~\eqref{NBlaw} with $N=1$. Using~\eqref{eq:dompath}, we get 
\begin{equation*}
g_A(x)\leq C_8 x^{-\kappa'}\widehat{\e}^*_1[(\ty^{\sigma_A}(w_{\sigma_A}))^{\kappa'/2}\ind{\sigma_A<\widehat{\tau}_1}]\leq C_8 A^{-\kappa/2} x^{-\kappa'}\widehat{\e}^*_1[(\ty^{\sigma_A}(w_{\sigma_A}))^{\kappa'}\ind{\sigma_A<\widehat{\tau}_1}] 
\end{equation*}
which is smaller than $\frac{\epsilon}{2}x^{-\kappa'}K_A$ for $A$ large enough. If $\kappa<2$, we write 
\begin{equation*}
g_A(x)\leq (\frac{\epsilon}{2}x)^{-\kappa'}\widehat{\e}^*_1\left[\widehat{\e}^*\left[ \sum_{k\geq \sigma_A+1}\sum_{u\in\Omega(w_k)}|(\ty^{\sigma_A}(u)-\ty^{\sigma_A}(w_{k-1})e^{-\Delta V(u)})Z_u|^{2}\ind{\sigma_A<\widehat{\tau}_1} \Big\vert \ty^{\sigma_A}(w_{\sigma_A)}\right]^{\kappa'/2}\right]
\end{equation*}
and we can use the same reasoning as in the case $\kappa\geq 2 $. \\

Equations \eqref{eq:domM12}, \eqref{eq:domM13} and \eqref{eq:domM14} yield \eqref{eq:ecartM1Winf}. \\

Let us now deal with the vertices on the spine. We want to prove that
\begin{equation}\label{eq:ecartM1infspine}
\widehat{\p}^*_1\big( |M_1^{>\sigma_A,\star}-(\beta^{\sigma_A}(w_{\sigma_A})\max_{\sigma_A\leq k\leq \tauh_1}e^{-(V(w_k)-V(w_{\sigma_A}))})|>2\epsilon x,\ \sigma_A<\widehat{\tau}_1 \big)\leq 2\epsilon x^{-\kappa'} K_A. 
\end{equation}

First, the same reasoning as that at the beginning of the proof of Lemma~19 of~\cite{dR16} yields that if one shows the existence of a constant $C_9\in(0;\infty)$ such that 
\begin{equation}\label{eq:toshowspine1}
\widehat{\p}^*_1\Big(\sum_{k\geq 1} \beta_0^1(w_k)>x\Big)\leq C_9 x^{-\kappa'}, 
\end{equation}
then 
\begin{equation*}
\widehat{\p}^*_1\Big(\sum_{k\geq \sigma_A} (\beta(w_k)-\beta^{\sigma_A}(w_k))>\epsilon x,\ \sigma_A<\widehat{\tau}_1\Big)\leq \epsilon x^{-\kappa'}K_A, 
\end{equation*}
which would yield 
\begin{equation}\label{eq:shownspine1}
\widehat{\p}^*_1\Big(\max_{k\geq \sigma_A} (\beta(w_k)-\beta^{\sigma_A}(w_k))>\epsilon x,\ \sigma_A<\widehat{\tau}_1\Big)\leq \epsilon x^{-\kappa'}K_A
\end{equation}
as $\beta(w_k)-\beta^{\sigma_A}(w_k)\geq 0$ for all $k\geq \sigma_A$. But showing~\eqref{eq:toshowspine1} is similar to showing the domination of $g(x)$ in the proof of~Lemma~19 of~\cite{dR16}, after replacing $\sum_{u\in\Omega(w_k)}e^{-\Delta V(u)}Z_u$ by $1$, hence~\eqref{eq:shownspine1} stands. 

Finally, the same reasoning as that made for the domination of $f_A$ in the proof of~Lemma~18 of~\cite{dR16} but replacing $Z_{w_\ell}$ by $1$ gives
\begin{equation}\label{eq:shownspine2}
\widehat{\p}^*_1\left( \max_{\sigma_A\leq k \leq \tauh_1} \big|\beta^{\sigma_A}(w_k)-\beta^{\sigma_A}(w_{\sigma_A})e^{-(V(w_k)-V(w_{\sigma_A}))}\big|>\epsilon x,\ \sigma_A<\widehat{\tau}_1 \right)\leq \epsilon x^{-\kappa'}K_A. 
\end{equation}
Equation~\eqref{eq:shownspine1} together with~\eqref{eq:shownspine2} yield~\eqref{eq:ecartM1infspine}. 

Combining equations~\eqref{eq:domM11},~\eqref{eq:ecartM1infspine} and~\eqref{eq:ecartM1Winf} concludes the proof of the lemma. 

\end{proof}
\paragraph{Step 3}

Because of the independence between $(\beta^{\sigma_A}(w_{\sigma_A}), \ind{\sigma_A<\widehat{\tau}_1})$ and $(W_\infty^{(w_{\sigma_A})},\mathcal{M}_e^{(w_{\sigma_A})})$, we only need to consider the joint tail $\widehat{\p}^*_1(W_\infty^{(w_{\sigma_A})}\geq ax,\mathcal{M}_e^{(w_{\sigma_A})}\geq x)$, which equals $\Q^*(W_\infty\geq ax, \mathcal{M}_e\geq x)$. It then follows from \eqref{tailWMQ} that as $x\to\infty$, 
\begin{align*}
\widehat{P}_1(ax,x;A)=&\widehat{\p}^*_1(W_\infty^{(w_{\sigma_A})}\geq a\frac{x}{\beta^{\sigma_A}(w_{\sigma_A})},\mathcal{M}_e^{(w_{\sigma_A})}\geq \frac{x}{\beta^{\sigma_A}(w_{\sigma_A})}; \sigma_A<\widehat{\tau}_1)\\
\sim & \mu(a) x^{-\kappa'}K_A.
\end{align*}
So, letting $x\rightarrow\infty$, $\epsilon\downarrow0$ then $A\uparrow\infty$ in \eqref{LMstep2} implies that 
\[
P_1(ax,x)\sim \eta_1(a)x^{-\kappa'}=\mu(a) C_\infty x^{-\kappa'},
\]
where
\begin{equation}\label{eq:defCinf}
C_\infty:=\lim_{A\rightarrow\infty}K_A\in (0,\infty).
\end{equation} 
The positivity of $C_\infty$ can be obtained by use of Lemma 4.1 in \cite{CdRH16}. By Lemma \ref{tailLMp}, one sees that
\begin{equation}\label{tailLMp1}
\p_1(L_1\geq ax, M_1\geq x)\sim \eta(a)x^{-\kappa},
\end{equation}
where 
\[
\eta(a)=c_\tM C_\infty\E[(1\wedge \frac{\mathcal{W}_\infty^\tM}{a})^{\kappa}].
\]

If we only use the tail of $W_\infty$, similar arguments imply that $\p_1(L_1\geq x)\sim c_L x^{-\kappa}$ with $c_L=C_0C_\infty$. However, we could not deduce the tail of $M_1$ under $\p_1$ as $\widehat{\p}_1^*(M_1\geq x)=\widehat{\p}_1^*(L_1\geq1, M_1\geq x)=\e_1[L_1; L_1\geq 1, M\geq x]$. In the next subsection, we study the tail of $M_1$.

\subsection{Tail of $M_1$ under $\p_1$: proof of Lemma \ref{tailsmallLM}}
It is immediate from Lemma \ref{tailsmallLM} and \eqref{tailLMp1} that as $x\rightarrow\infty$.
\[
\p_1(M_1\geq x)\sim \eta(0) x^{-\kappa},
\]
with $\eta(0)=c_\tM C_\infty$. Let us prove Lemma \ref{tailsmallLM} here, which mainly follows from Lemma \ref{BRWtailM}.

\begin{proof}[Proof of Lemma \ref{tailsmallLM}]
Let us consider $\p_1(L_1\leq \epsilon r, M_1\geq r)$ for $\epsilon\in(0,1)$ small and $r\gg1$. Let $\Sigma_r:=\sum_{u\in\T}\ind{\beta(u)\geq r>\max_{\rho\leq v<u}\beta(v),  \min_{\rho<v<u}\beta(v)\geq 2}$ for any $r>2$. Obviously, $\{M_1\geq r\}=\{\Sigma_r\geq 1\}$. Moreover, for any $u<\mathcal{L}_1$, one has
\[
L_1=\sum_{j=1}^{|u|} \sum_{z\in\Omega(u_j)}L_1^{(z)}\ind{z\in\B_1}+L_1^{(u)}\geq L_1^{\slash u}:=\sum_{j=1}^{|u|} \sum_{z\in\Omega(u_j)}L_1^{(z)}\ind{z\in \B_1}.
\]
Then observe that
\begin{equation*}
\p_1(M_1\geq r, L_1\leq\epsilon r)\leq \e_1\left[\Sigma_r\ind{L_1\leq \epsilon r}\right]
\leq  \e_1\left[\sum_{u\in\T}\ind{\beta(u)\geq r>\max_{\rho\leq v<u}\beta(v), \min_{\rho<v<u}\beta(v)\geq 2}\ind{L_1^{\slash u}\leq \epsilon r}\right]
\end{equation*}
which by change of measures, is bounded by 
\[
 \sum_{n\geq 1}\widehat{\e}^*_1\left[\frac{1}{\beta(w_n)}\ind{n=\sigma_r<\widehat{\tau}_1}\ind{L_1^{\slash w_n}\leq \epsilon r}\right]
\]
So, one sees that for any $2\leq A<r$, 
\begin{align*}
\p_1(M_1\geq r, L_1\leq \epsilon r)
\leq & \sum_{n\geq 1}\widehat{\e}^*_1\left[\frac{1}{\beta(w_n)}\ind{n=\sigma_r<\widehat{\tau}_1}\ind{L_1^{\slash w_n}\leq \epsilon r}\right]\\
%=&\widehat{\e}^*_1\left[\frac{1}{\beta(w_{\sigma_r})}\ind{\sigma_r<\widehat{\tau}_1}}\ind{L_1^{\slash w_{\sigma_r}}\leq \epsilon r}\right]\\
\leq &\widehat{\e}^*_1\left[\frac{1}{\beta(w_{\sigma_r})}\ind{\sigma_r<\widehat{\tau}_1}\ind{\sum_{j=\sigma_A+1}^{\sigma_r}\sum_{z\in\Omega(w_j)}L_1^{(z)}\leq \epsilon r}\right]
\end{align*}
As in the previous subsection, under $\widehat{\p}^*_1$, we can approximate $L_1^{(z)}$ by $\beta^{\sigma_A}(w_{\sigma_A})e^{V(w_{\sigma_A})-V(z)}W^{(z)}_\infty$, and along the spine $(w_k)_{k\geq\sigma_A}$, $\beta(w_k)$ can be approximated by $\beta^{\sigma_A}(w_{\sigma_A})e^{-(V(w_k)-V(w_{\sigma_A}))}$. In fact, we have
\begin{align*}
&\p_1(M_1\geq r, L_1\leq \epsilon r)\\
\leq & \widehat{\e}^*_1\left[\frac{1}{\beta(w_{\sigma_r})}%\ind{\sum_{j=\sigma_A+1}^{\sigma_r}\sum_{z\in\Omega(w_j)}L_1^{(z)}\leq \epsilon r}
\ind{\sigma_r<\widehat{\tau}_1}\ind{|\sum_{j=\sigma_A+1}^{\sigma_r}\sum_{z\in\Omega(w_j)}[L_1^{(z)}-\beta^{\sigma_A}(w_{\sigma_A})e^{V(w_{\sigma_A})-V(z)}W^{(z)}_\infty]|\geq \epsilon r}\right]\\
&+ \widehat{\e}^*_1\left[\frac{1}{\beta(w_{\sigma_r})}\ind{\sigma_r<\tauh_1}%\ind{\sigma_r<\widehat{\tau}_1}
\ind{\max_{\sigma_A\leq k\leq \widehat{\tau}_1}|\beta(w_k)-\beta^{\sigma_A}(w_{\sigma_A})e^{-(V(w_k)-V(w_{\sigma_A}))}|\geq \epsilon r}\right]\\
&+ \widehat{\e}^*_1\left[\frac{1}{\beta(w_{\sigma_r})}\ind{\sigma_r<\widehat{\tau}_1}\ind{\beta^{\sigma_A}(w_{\sigma_A})\geq \sqrt{r}}%\ind{\beta^{\sigma_A}(w_{\sigma_A})e^{-(V(w_{\sigma_r})-V(w_{\sigma_A})}\geq (1-\epsilon) r}
\ind{\max_{\sigma_A\leq k\leq \widehat{\tau}_1}|\beta(w_k)-\beta^{\sigma_A}(w_{\sigma_A})e^{-(V(w_k)-V(w_{\sigma_A}))}|< \epsilon r}\right]+\widehat{\e}^*_{\eqref{keybd}},
\end{align*}
where
\begin{multline}\label{keybd}
\widehat{\e}^*_{\eqref{keybd}}:=\widehat{\e}^*_1\left[\frac{1}{\beta(w_{\sigma_r})}\ind{\sigma_r<\widehat{\tau}_1}\ind{\sum_{j=\sigma_A+1}^{\sigma_r}\sum_{z\in\Omega(w_j)}\beta^{\sigma_A}(w_{\sigma_A})e^{V(w_{\sigma_A})-V(z)}W^{(z)}_\infty\leq 2\epsilon r; \ \beta^{\sigma_A}(w_{\sigma_A})< \sqrt{r}}\right.\\
\times\left.\ind{\max_{\sigma_A\leq k\leq \widehat{\tau}_1}|\beta(w_k)-\beta^{\sigma_A}(w_{\sigma_A})e^{-(V(w_k)-V(w_{\sigma_A}))}|< \epsilon r}\right]
\end{multline}
Notice that the third term of the sum is smaller than 
\begin{align*}
&\widehat{\e}^*_1\left[\frac{1}{\beta(w_{\sigma_r})}\ind{\sigma_r<\widehat{\tau}_1}\ind{\beta^{\sigma_A}(w_{\sigma_A})\geq \sqrt{r}}\ind{\beta^{\sigma_A}(w_{\sigma_A})e^{-(V(w_{\sigma_r})-V(w_{\sigma_A})}\geq (1-\epsilon) r}\right]\\
&\le\frac{1}{r}\widehat{\p}^*_1\left( \beta^{\sigma_A}(w_{\sigma_A})\geq \sqrt r,\ \sigma_A<\tauh_1,\ \sum_{n\geq \sigma_A}e^{-(V(w_n)-V(w_{\sigma_A})}\geq (1-\epsilon)\frac{r}{\beta^{\sigma_A}(w_{\sigma_A})} ) \right)\\
&\leq \frac{1}{r}\widehat{\e}^*_1\left[ \ind{\beta^{\sigma_A}(w_{\sigma_A})\geq \sqrt r}\ind{\sigma_A<\tauh_1}C_{10}(\beta^{\sigma_A}(w_{\sigma_A}))^{\kappa'}r^{-\kappa'} \right], 
\end{align*}
where the inequality is due to Fact~2.2(1) of~\cite{CdRH16} and the Markov property applied at time $\sigma_A$. Moreover, one sees that
\begin{align*}
\widehat{\e}^*_{\eqref{keybd}}\le&\widehat{\e}^*_1\left[\sum_{n=\sigma_A+1}^\infty \frac{1}{\beta(w_{\sigma_r})}\ind{\sigma_r=n<\tauh_1,\ \beta^{\sigma_A}(w_{\sigma_A})<\sqrt{r}}\ind{\sum_{j=\sigma_A+1}^{n}\sum_{z\in\Omega(w_j)}\beta^{\sigma_A}(w_{\sigma_A})e^{V(w_{\sigma_A})-V(z)}W^{(z)}_\infty\leq 2\epsilon r}\right.\\
&\qquad\qquad\times\left.\ind{\max_{\sigma_A\leq k\leq \sigma_r}|\beta(w_k)-\beta^{\sigma_A}(w_{\sigma_A})e^{-(V(w_k)-V(w_{\sigma_A}))}|<  \epsilon r\le \epsilon \beta(w_{\sigma_r})}\right]\\
\leq &\widehat{\e}^*_1\left[\sum_{n=\sigma_A+1}^\infty \frac{1+\epsilon}{\beta^{\sigma_A}(w_{\sigma_A})e^{-(V(w_{n})-V(w_{\sigma_A}))}}\ind{\sigma_r=n<\tauh_1,\ \beta^{\sigma_A}(w_{\sigma_A})<\sqrt{r}}\ind{\beta^{\sigma_A}(w_{\sigma_A})e^{-(V(w_n)-V(w_{\sigma_A}))}\ge (1-\epsilon)r}\right.\\
&\times\left.\ind{\sum_{j=\sigma_A+1}^{n}\sum_{z\in\Omega(w_j)}\beta^{\sigma_A}(w_{\sigma_A})e^{V(w_{\sigma_A})-V(z)}W^{(z)}_\infty\leq 2\epsilon r}\ind{\max_{\sigma_A\leq k< n}\beta^{\sigma_A}(w_{\sigma_A})e^{-(V(w_k)-V(w_{\sigma_A}))}< (1+\epsilon)r}\right].
\end{align*}
Let $r_A:=\log(\frac{r}{\beta^{\sigma_A}(w_{\sigma_A})})$. We get that
\begin{align*}
\widehat{\e}^*_{\eqref{keybd}}\le&(1+\epsilon)\widehat{\e}^*_1\left[\ind{\beta^{\sigma_A}(w_{\sigma_A})< \sqrt{r},\ \sigma_A<\widehat{\tau}_1}\sum_{n\geq \sigma_A+1}\frac{e^{V(w_n)-V(w_{\sigma_A})}}{\beta^{\sigma_A}(w_{\sigma_A})}\ind{V(w_n)-V(w_{\sigma_A})\le -r_A-\log(1-\epsilon)}\right.\\
&\times\left.\ind{\sum_{j=\sigma_A+1}^{n}\sum_{z\in\Omega(w_j)}e^{-(V(z)-V(w_{\sigma_A}))}W^{(z)}_\infty\leq 2\epsilon e^{r_A}}\ind{\min_{\sigma_A\leq k< n}(V(w_k)-V(w_{\sigma_A}))> -r_A-\log(1+\epsilon)}\right]
\end{align*}
Consequently, for $\epsilon\in(0,1/3)$ sufficiently small, 
\begin{align}\label{tailMsmallL}
&\p_1(M_1\geq r, L_1\leq \epsilon r)\\
\leq & \frac{1}{r}\widehat{\p}_1^*\left(\sum_{j=\sigma_A+1}^{\tauh_1}|\sum_{z\in\Omega(w_j)}[L_1^{(z)}-\beta^{\sigma_A}(w_{\sigma_A})e^{V(w_{\sigma_A})-V(z)}W^{(z)}_\infty]|\geq \epsilon r\right)\nonumber\\
&+\frac{1}{r}\widehat{\p}_1^*\left(\max_{\sigma_A\leq k\leq \widehat{\tau}_1}|\beta(w_k)-\beta^{\sigma_A}(w_{\sigma_A})e^{-(V(w_k)-V(w_{\sigma_A}))}|\geq \epsilon r\right)\nonumber\\
&+\frac{C_{10}}{r^\kappa}\widehat{\e}_1^*\left[(\beta^{\sigma_A}(w_{\sigma_A}))^{\kappa-1}\ind{\beta^{\sigma_A}(w_{\sigma_A})\geq \sqrt{r}}\ind{\sigma_A<\widehat{\tau}_1}\right]+(1+\epsilon)\widehat{\e}^*_{\eqref{tailMsmallLp}}\nonumber
\end{align}
where
\begin{multline}\label{tailMsmallLp}
\widehat{\e}^*_{\eqref{tailMsmallLp}}=:\widehat{\e}^*\Bigg[\ind{\beta^{\sigma_A}(w_{\sigma_A})< \sqrt{r}; \sigma_A<\widehat{\tau}_1}\sum_{n\geq \sigma_A+1}\frac{e^{(V(w_n)-V(w_{\sigma_A}))}}{\beta^{\sigma_A}(w_{\sigma_A})}\ind{\sum_{j=\sigma_A+1}^n\sum_{z\in\Omega(w_j)}e^{-(V(z)-V(w_{\sigma_A}))}W_\infty^{(z)}\leq 2\epsilon \exp(r_A)}
\\\times\ind{\min_{\sigma_A\leq k<n}(V(w_k)-V(w_{\sigma_A}))> -r_A-\epsilon;\ (V(w_n)-V(w_{\sigma_A}))\leq-r_A+2\epsilon}\Bigg],
\end{multline}
with $r_A\geq \log(r)/2$. According to the Subsection \ref{jointtailLMp}, one sees from \eqref{tailMsmallL} that for $A\geq2$ fixed and $r\gg 1$,
\begin{equation}\label{upbdtailMsmallL}
\p_1(M_1\geq r, L_1\leq \epsilon r)\leq C_{11} \epsilon r^{-\kappa}+ o_r(1) r^{-\kappa}+ 2 \widehat{\e}^*_{\eqref{tailMsmallLp}}.
\end{equation}
Here for $\widehat{\e}^*_{\eqref{tailMsmallLp}}$, by Markov property at time $\sigma_A$, one has
\[
\widehat{\e}^*_{\eqref{tailMsmallLp}}=\widehat{\e}_1^*\left[\frac{\ind{\beta^{\sigma_A}(w_{\sigma_A})< \sqrt{r}; \sigma_A<\widehat{\tau}_1}}{\beta^{\sigma_A}(w_{\sigma_A})}\E_{\Q^*}^{\eqref{tailWsmallVP}}(\epsilon, r_A)\right]
\]
where for any $\delta_1,\delta_2>0$,
\begin{multline}\label{tailWsmallVP}
\E_{\Q^*}^{\eqref{tailWsmallVP}}(\epsilon, r_A)=: \sum_{n\geq1}\E_{\Q^*}\Bigg[e^{V(w_n)}\ind{V(w_n)\leq -r_A+2\epsilon,\ \min_{1\leq k\leq n-1}V(w_k)> -r_A-\epsilon}\\
\times\ind{\sum_{j=1}^n\sum_{z\in\Omega(w_j)}e^{-V(z)}W_\infty^{(z)}\leq 2\epsilon \exp(r_A)}\Bigg].
\end{multline}
Let us claim that 
\begin{equation}\label{cvgtailWsmallVP}
\limsup_{\epsilon\downarrow0}\limsup_{r\rightarrow\infty}e^{\kappa r}\E_{\Q^*}^{\eqref{tailWsmallVP}}(\epsilon, r)=0.
\end{equation}
It mainly follows from Lemma \ref{BRWtailM}. We postpone its proof in the Subsection \ref{lemRWRE}. Then \eqref{cvgtailWsmallVP} implies that
\[
\widehat{\e}^*_{\eqref{tailMsmallLp}}=\widehat{\e}_1^*\left[\frac{\ind{\beta^{\sigma_A}(w_{\sigma_A})< \sqrt{r}; \sigma_A<\widehat{\tau}_1}}{\beta^{\sigma_A}(w_{\sigma_A})} e^{-\kappa r_A}o_{\epsilon, r_A}(1)\right]\leq  \frac{o_{\epsilon, r}(1)}{r^\kappa} \E\left[(\beta^{\sigma_A}(w_{\sigma_A}))^{\kappa-1}; \sigma_A<\widehat{\tau}_1\right].
\]
Plugging it in \eqref{upbdtailMsmallL} yields that
\[
\limsup_{\epsilon\downarrow0}\limsup_{r\rightarrow\infty}r^\kappa \p_1(M_1\geq r, L_1\leq \epsilon r)=0.
\]

\end{proof}

\section{Proofs of the technical lemmas}
\label{TechLem}

\subsection{Proofs of the technical lemmas in Section~\ref{tailBRW}}
\label{lemBRW}

Recall that $\psi(1)=\psi(\kappa)=1$ and $\psi(-1-\delta)+\psi(\kappa+\delta)<\infty$ for some $\delta>0$. We have the well-known many-to-one lemma.
\begin{lem}[Many-to-one]\label{ManytoOne}
Assume \eqref{hyp1} and \eqref{hyp1+}. For any $n\geq 1$, $x\in\R$ and any measurable function $g:\mathbb{R}^n\rightarrow \R_+$,
\begin{equation}\label{manyto1}
\E_x\left[\sum_{|z|=n}g\Big(V(z_1),\cdots, V(z_n)\Big)\right]=\E_x\left[e^{S_n-x}g(S_1,\cdots,S_n)\right],
\end{equation}
where $(S_n)$ is a random walk with i.i.d. increments such that 
\[
\E[S_1]=\E\left[\sum_{|z|=1}V(z)e^{-V(z)}\right]>0,\quad \E[e^{-(\kappa-1)S_1}]=1,\quad \E[e^{\delta S_1}+e^{-(\kappa-1+\delta)S_1}]<\infty.
\]
In particular, under $\Q^*_x$, $(V(w_i); 1\leq i\leq n)$ has the same distribution as $(S_i; 1\leq i\leq n)$ under $\P_x$ where $\P_x(S_0=x)=1$.
Similarly, as $\psi(\kappa)=1$ and $\psi'(\kappa)>0$, we have
\begin{equation}\label{manyto1+}
\E_x\left[\sum_{|z|=n}g\Big(V(z_1),\cdots, V(z_n)\Big)\right]=\E\left[e^{\kappa (S^{(\kappa)}_n-x)}g(S^{(\kappa)}_1,\cdots,S^{(\kappa)}_n)\right],
\end{equation}
where $(S^{(\kappa)}_n)$ is a random walk with i.i.d. increments such that 
\[
\E[S^{(\kappa)}_1]=\E\left[\sum_{|z|=1}V(z)e^{-\kappa V(z)}\right]<0,\quad \E[e^{(\kappa-1+\delta) S^{(\kappa)}_1}+e^{-\delta S^{(\kappa)}_1}]<\infty.
\]
Moreover, 
\begin{equation}\label{rwschange}
\E_x\left[g(S_1,\cdots,S_n)\right]=\E_x\left[e^{(\kappa-1) (S^{(\kappa)}_n-x)}g(S^{(\kappa)}_1,\cdots,S^{(\kappa)}_n)\right]
\end{equation}
\end{lem}

Before proving the lemmas, let us state some classic results on random walks $(S_n)_{n\geq0}$ and $(S_n^{(\kappa)})_{n\geq0}$.

\subsubsection{Renewal theory for one-dimensional random walk}

For any $0\leq j< n$, let
\[
\MS_{[j,n]}:=\max_{j\leq k\leq n}{S_k}, \mS_{[j,n]}:=\min_{j\leq k\leq n}S_k,\textrm{ and } \overline{V}(w_{[j,n]}):=\max_{j\leq k\leq n}V(w_k), \underline{V}(w_{[j,n]}):=\min_{j\leq k\leq n}V(w_k).
\]

For the random walk $(S_n^{(\kappa)})_{n\geq0}$, we define the renewal measures $U_s^{(\kappa),\pm}$ corresponding to the strict ascending/descending ladder process by 
\begin{equation}
U_s^{(\kappa),\pm}([0,x]):=\E\left[\sum_{k=0}^{\tau^{(\kappa),\mp}-1}\ind{\pm S^{(\kappa)}_k\leq x}\right], \forall x\geq0.
\end{equation}
with $\tau^{(\kappa),+}:=\inf\{k\geq1: S^{(\kappa)}_k\geq0\}$ and $\tau^{(\kappa),-}:=\inf\{k\geq1: S^{(\kappa)}_k\leq 0\}$. And the renewal measures corresponding to the weak ascending/descending ladder process are defined as follows:
\[
U_w^{(\kappa),\pm}([0,x]):=\E\left[\sum_{k=0}^{\hat{\tau}^{(\kappa),\mp}-1}\ind{\pm S^{(\kappa)}_k\leq x}\right], \forall x\geq0.
\]
with $\hat{\tau}^{(\kappa),+}:=\inf\{k\geq1: S^{(\kappa)}_k>0\}$ and $\hat{\tau}^{(\kappa),-}:=\inf\{k\geq1: S^{(\kappa)}_k<0\}$.

As usual, we set the strict renewal functions to be $R_s^{(\kappa),\pm}(x):=U_s^\pm([0,x])$. Then it is known that there exist constants $C_s^\pm\in(0,\infty)$ such that for any $h>0$, as $x\rightarrow\infty$,
\[
R^{(\kappa),+}_s(x)\rightarrow C_s^{(\kappa),+}\textrm{ and } U^{(\kappa),-}_s((x-h,x])\rightarrow C_s^{(\kappa),-}h.
\]
As $\E[S_1^{(\kappa)}]<0$, its strict descending process, denoted by $(\hat{\tau}_n^{(\kappa),-}, \hat{H}_n^{(\kappa),-})_{n\geq0}$, is proper. For the strict ascending ladder process, define the epochs by
\[
\hat{\tau}_n^{(\kappa),+}:=\inf\{k >\hat{\tau}_{n-1}^{(\kappa),+}: S^{(\kappa)}_k > S^{(\kappa)}_{\hat{\tau}_{n-1}^{(\kappa),+}} \}, \ \forall n\geq1,
\]
with $\hat{\tau}_0^{(\kappa),+}:=0$. On $\{\hat{\tau}_n^{(\kappa),+}<\infty\}$, let $\hat{H}_n^{(\kappa),+}:=S^{(\kappa)}_{\hat{\tau}_n^{(\kappa),+}}$. Then
\[
R_s^{(\kappa),+}(x)=\E\left[\sum_{k\geq0}\ind{\hat{\tau}_k^{(\kappa),+}<\infty, \hat{H}_k^{(\kappa),+}\leq x}\right]\textrm{ and } R_s^{(\kappa),-}(x)=\E\left[\sum_{k\geq0}\ind{\hat{H}_k^{(\kappa),-}\geq -x}\right]
\]

Write $I(x)$ for $(-x-1,-x]$. 

\begin{lem}
Under the assumptions of Lemma \ref{ManytoOne}, for any $a>0$, as $x\rightarrow\infty$,
\begin{equation}\label{RWeScvg}
\E_{-a}\left[\sum_{k\geq0}e^{S_k+\kappa x}\ind{\MS_{[1,k]}<0, S_k\in I(x)}\right]\rightarrow e^{(\kappa-1)a}\frac{1-e^{-\kappa}}{\kappa}C_s^{(\kappa),-} U_w^{(\kappa),+}([0,a)),
\end{equation}
In addition, there exists a constant $c_2>0$ such that for any $x\geq0$ and $a>0$,
\begin{equation}\label{RWeSbd}
\E_{-a}\left[\sum_{k\geq0}e^{S_k+\kappa x}\ind{\MS_{[1,k]}<0, S_k\in I(x)}\right]\leq  c_2(1+a)e^{\kappa' a}. 
\end{equation}
\end{lem}

\begin{proof}
Let us consider 
\begin{equation}
R(x,a):=\sum_{j\geq0}\P_{-a}\left(\MS^{(\kappa)}_{[1,j]}<0, S^{(\kappa)}_j> -x\right), \forall a,x \geq0.
\end{equation}

Note that $(S^{(\kappa)}_i)_{1\leq i\leq j}$ has the same distribution as $(S^{(\kappa)}_j-S^{(\kappa)}_{j-i})_{1\leq i\leq j}$. As a consequence, for any $a>0$,
\begin{align*}
R(x,a)=&\sum_{j\geq0}\P(\MS^{(\kappa)}_{[1,j]}<a, S^{(\kappa)}_j> a-x)\\
=&\sum_{j\geq0}\P\left(S^{(\kappa)}_j< a+\mS^{(\kappa)}_{[0,j-1]}, S^{(\kappa)}_j> a-x\right)
\end{align*}
As the associated descending ladder process is proper, this implies that
\begin{align*}
R(x,a)=&\sum_{n\geq0}\E\left[\sum_{j=\hat{\tau}_n^{(\kappa),-}}^{\hat{\tau}_{n+1}^{(\kappa),-}-1}\ind{S^{(\kappa)}_j< a+\mS^{(\kappa)}_{[0,j-1]}, S^{(\kappa)}_j> a-x}\right]\\
=&\sum_{n\geq 0} \E\left[\ind{\hat{H}_n^{(\kappa),-}\geq -x}\sum_{j=\hat{\tau}_n^{(\kappa),-}}^{\hat{\tau}_{n+1}^{(\kappa),-}-1}\ind{a-x< S^{(\kappa)}_j<a +\hat{H}_n^{(\kappa),-}}\right],
\end{align*}
which by Markov property at time $\tau_n^{(\kappa),-}$ equals to
\begin{align*}
&\sum_{n\geq 0} \E\left[\ind{\hat{H}_n^{(\kappa),-}\geq-x}U_w^{(\kappa),+}\left((a-x-\hat{H}_n^{(\kappa),-}, a)\right)\right]\\
=&\sum_{n\geq 0} \E\left[\ind{\hat{H}_n^{(\kappa),-}\geq -x}U_w^{(\kappa),+}\left([0, a)\right)\right]-\sum_{n\geq 0} \E\left[\ind{\hat{H}_n^{(\kappa),-}\geq-x}U_w^{(\kappa),+}\left([0, a-x-\hat{H}_n^{(\kappa),-}]\right)\right]
\end{align*}
This means that for any $a>0$ and $x\geq0$, 
\begin{equation}
%=& \sum_{n\geq 0} \E\left[\ind{\hat{H}_n^{(\kappa),-}>-x}\left(\right)\right]\\
R(x,a)= U_w^{(\kappa),+}\left( [0, a)\right)R_s^{(\kappa),-}(x) -\int_{[(x-a)_+,x]} U_w^{(\kappa),+}\left([0, a-x+u]\right) U_s^{(\kappa),-}(du).
\end{equation}
As $U_w^{(\kappa),+}\left(\R\right)=C_w^{(\kappa),+}\in(0,\infty) $, one sees that there exists some constant $c_3>0$ such that for any $x\geq0$,
\begin{equation}\label{upperbd}
\sum_{k\geq0}\P_{-a}\left[\MS^{(\kappa)}_{[1,k]}<0, S^{(\kappa)}_k\in I(x)\right]=R(x+1,a)-R(x,a)\leq c_3 (1+a), \forall a>0.
\end{equation}
Moreover, for any $b>0$, as $x\rightarrow\infty$, 
\begin{equation}\label{cvgrenewal}
R(x+b,a)-R(x,a)\rightarrow U_w^{(\kappa),+}\left( [0, a)\right) C_s^{(\kappa),-}b.
\end{equation}
It follows from Lemma \ref{ManytoOne} that 
\[
\E_{-a}\left[\sum_{k\geq0}e^{S_k+\kappa x}\ind{\MS_{[1,k]}<0, S_k\in I(x)}\right]
=e^{(\kappa-1)a}\E_{-a}\left[\sum_{k\geq0}e^{\kappa(S_k^{(\kappa)}+ x)}\ind{\MS^{(\kappa)}_{[1,k]}<0, S^{(\kappa)}_k\in I(x)}\right]
\]
Clearly, we have
\[
\E_{-a}\left[\sum_{k\geq0}e^{S_k+\kappa x}\ind{\MS_{[1,k]}<0, S_k\in I(x)}\right]\leq e^{\kappa' a} (R(x+1,a)-R(x,a))\leq c_3(1+a)e^{\kappa' a} .
\]
Moreover, for any $a>0$,
\begin{align*}
&\E_{-a}\left[\sum_{k\geq0}e^{S_k+\kappa x}\ind{\MS_{[1,k]}<0, S_k\in I(x)}\right]\\
%=&e^{(\kappa-1)a}\E_{-a}\left[\sum_{k\geq0}e^{\kappa(S_k^{(\kappa)}+ x)}\ind{\MS^{(\kappa)}_{[1,k]}<0, S^{(\kappa)}_k\in I(x)}\right]\\
=&e^{(\kappa-1)a+\kappa x}\int_{-\infty}^{-x}\kappa e^{\kappa u}\sum_{k\geq0}\P_{-a}\left(S^{(\kappa)}_k>u, \MS^{(\kappa)}_{[1,k]}<0, S^{(\kappa)}_k\in (-x-1,-x]\right)du\\
=&e^{(\kappa-1)a-\kappa}(R(x+1,a)-R(x,a))+e^{(\kappa-1)a}\int_0^1\kappa e^{-\kappa t}(R(x+t,a)-R(x,a))dt
\end{align*}
which, by \eqref{cvgrenewal}, converges towards
\[
e^{(\kappa-1)a}\left(e^{-\kappa}+\int_0^1\kappa te^{-\kappa t}dt\right) C_s^{(\kappa),-} U_w^{(\kappa),+}([0,a)),
\]
as $x\rightarrow\infty$. Here $e^{-\kappa}+\int_0^1\kappa te^{-\kappa t}dt=(1-e^{-\kappa})/\kappa$.
\end{proof}

\subsubsection{Proofs of lemmas \ref{BRWroughbd}, \ref{BRWmaincvg}, \ref{BRWrest}, \ref{BRWtight} and \ref{BRWtailM}}\label{subsubsec:prooflemmas}

\begin{proof}[Proof of Lemma \ref{BRWroughbd}]
The upper bound is quite easy. In fact, observe immediately that $\{\tM\leq -x\}$ implies that $\{\sum_{u\in\T}\ind{V(u)\leq -x<\min_{\rho<v<u}V(v)}\geq1\}$
 So, by Markov inequality then by \eqref{manyto1+},
\begin{align*}
\P(\tM\leq -x)\leq& \E\left[\sum_{k\geq1}\sum_{|u|=k}\ind{V(u)\leq -x<\min_{\rho<v<u}V(v)}\right]=\sum_{k\geq1}\E\left[e^{\kappa S^{(\kappa)}_k}; S^{(\kappa)}_k\leq -x < \min_{1\leq i\leq k-1}S^{(\kappa)}_i\right]\\
\leq & e^{-\kappa x} \sum_{k\geq 1}\P(S^{(\kappa)}_k\leq -x < \min_{1\leq i\leq k-1}S^{(\kappa)}_i)\\
=&e^{-\kappa x}.
\end{align*}

For the lower bound, let us introduce the following events for any $u\in\T$:
\[
E_u:=\{V(u)\in I(x), V(u)<\min_{\rho\leq z<u}V(z)\},\textrm{ and } F^L_u:=\{\sum_{k=1}^{|u|}\sum_{z\in\Omega(u_k)}e^{-\kappa(V(z)+x)}\leq L\}, \forall L\geq1 .
\]
Define
\[
N(x):=\sum_{u\in\T}\ind{E_u}, \ N_L(x):=\sum_{u\in\T}\ind{E_u}\ind{F^L_u}.
\]
Then, by Paley-Zygmund inequality, one sees that
\begin{equation}\label{PaleyZygmund}
\P(\tM\leq -x)\geq \P(N_L(x)\geq1)\geq \frac{\E[N_L(x)]^2}{\E[N_L(x)^2]}.
\end{equation}
Let us estimates the first and the second moments of $N_L(x)$. Note that by many-to-one lemma,
\begin{align*}
\E[N(x)]=&\sum_{n\geq1}\E\left[\sum_{|u|=n}\ind{V(u)<\underline{V}(u_{[0,n-1]}), V(u)\in I(x)}\right]
=\sum_{n\geq1}\E\left[e^{\kappa S^{(\kappa)}_n}\ind{S^{(\kappa)}_n<\mS^{(\kappa)}_{[0,n-1]}, S^{(\kappa)}_n\in I(x)}\right]\\
= & \E\left[\sum_{k=0}^\infty e^{\kappa\hat{H}^{(\kappa),-}_k} \ind{\hat{H}^{(\kappa),-}_k\in I(x)}\right]%= c e^{-\kappa}U^{(\kappa),-}_s((-x-1,-x])
\end{align*}
So, 
\[
e^{-\kappa x-\kappa}U^{(\kappa),-}_s([x,x+1))\leq \E[N(x)]\leq e^{-\kappa x}U^{(\kappa),-}_s([x,x+1)).
\]
Consequently, there exist $0<c_4<c_5<\infty$ such that for any $x\geq0$,
\begin{equation}\label{meanNx}
c_4 e^{-\kappa x} \leq \E[N(x)]\leq c_5 e^{-\kappa x}.
\end{equation}
Note that $\E[N_L(x)]=\E[N(x)]-\E[\sum_{u\in\T}\ind{E_u}\ind{(F^L_u)^c}]$. Let us bound $\E[\sum_{u\in\T}\ind{E_u}\ind{(F^L_u)^c}]$. In fact, by Proposition \ref{BRWchangeofp},
\begin{align*}
\E\left[\sum_{u\in\T}\ind{E_u}\ind{(F^L_u)^c}\right]=&\sum_{n\geq1}\E_{\Q^*}\left[e^{V(w_n)}\ind{V(w_n)< \underline{V}(w_{[0,n-1]}), V(w_n)\in I(x)}\ind{\sum_{k=1}^n\sum_{u\in\Omega(w_k)}e^{-\kappa(V(u)+x)}>L}\right]
\end{align*}
Observe that $L\geq \sum_{k=1}^n\frac{6L}{\pi^2 k^2}=\sum_{k=1}^n\frac{6L}{\pi^2 (n-k+1)^2}$. It follows that
\begin{align*}
&\E\left[\sum_{u\in\T}\ind{E_u}\ind{(F^L_u)^c}\right]\\
\leq &\sum_{n\geq1}\E_{\Q^*}\left[e^{V(w_n)}\ind{V(w_n)< \underline{V}(w_{[0,n-1]}), V(w_n)\in I(x)}\sum_{k=1}^n\ind{\sum_{u\in\Omega(w_k)}e^{-\kappa(V(u)+x)}>\frac{6L}{\pi^2(n-k+1)^2}}\right]\\
\leq &\sum_{n=1}^\infty\sum_{k=1}^n \E_{\Q^*}\left[e^{V(w_n)}\ind{V(w_n)< \underline{V}(w_{[0,n-1]}), V(w_n)\in I(x)}; \Delta^e_k e^{-\kappa V(w_{k-1})}\geq \frac{6Le^{\kappa x}}{\pi^2(n-k+1)^2}\right]
\end{align*}
where $\Delta_k^e:=\sum_{u\in\Omega(w_k)}e^{-\kappa \Delta V(u)}$. From the fact that $(\Delta V(w_k), \Delta_k^e)_{1\leq k\leq n}$ have the same law as $(\Delta V(w_{n-k+1}), \Delta_{n-k+1}^e)_{1\leq k\leq n}$, one sees that $(V(w_k), \Delta_k^e)_{1\leq k\leq n}$ and $(V(w_n)-V(w_{n-k}), \Delta_{n+1-k}^e)_{1\leq k\leq n}$ have the same law. As a consequence, 
\begin{align*}
&\E\left[\sum_{u\in\T}\ind{E_u}\ind{(F^L_u)^c}\right]\\
\leq &\sum_{n\geq1}\sum_{k=1}^n\E_{\Q^*}\left[e^{V(w_n)}\ind{\overline{V}(w_{[1,n]})<0, V(w_n)\in I(x)}; \Delta^e_{n-k+1} e^{\kappa V(w_{n-k+1})-\kappa V(w_n)}\geq\frac{6Le^{\kappa x}}{\pi^2(n-k+1)^2}\right] \\
\leq & \sum_{n=1}^\infty\sum_{k=1}^n\E_{\Q^*}\left[e^{V(w_n)}\ind{\overline{V}(w_{[1,n]})<0, V(w_n)\in I(x)}; \Delta_k^e e^{\kappa V(w_k)}\geq \frac{6L}{\pi^2k^2}\right]\\
=& \sum_{k=1}^\infty\E_{\Q^*}\left[\ind{\overline{V}(w_{[1,k]})<0, -V(w_k)\leq \frac{\log \Delta_k^e+2\ln k-\ln (6L/\pi^2)}{\kappa}}\sum_{n=k}^\infty\E_{V(w_k)}\left[e^{S_{n-k}}\ind{\MS_{[1,n-k]}<0, S_{n-k}\in I(x)}\right]\right],
\end{align*}
where the last equality follows from the Markov property at $w_k$. Note that $\{V(w_i); 1\leq i\leq j-1\}$ which is distributed as $(S_i; 1\leq i\leq j-1)$, is independent of $(\Delta V(w_j),\Delta_j^e)$. Let us introduce a new couple $(\zeta, \Delta^e)$ which under $\P$ is distributed as $(\Delta V(w_1),\Delta_1^e)$ under $\Q^*$ and is independent of the random walk $(S_k)$. By \eqref{RWeSbd}, 
\begin{align}\label{newcouplebd}
&\E\left[\sum_{u\in\T}\ind{E_u}\ind{(F^L_u)^c}\right]\\
\leq &c_2 e^{-\kappa x} \sum_{k=1}^\infty\E_{\Q^*}\left[\ind{\overline{V}(w_{[1,k]})<0, -V(w_k)\leq \frac{\ln \Delta_k^e+2\ln k-\log (6L/\pi^2)}{\kappa}}(1-V(w_k))e^{-\kappa' V(w_k)}\right]\nonumber\\
 \leq & c_6 e^{-\kappa x} \sum_{k=1}^\infty\E\left[\ind{\MS_{[1,k-1]}<0, -S_{k-1}\leq \zeta+\frac{\ln \Delta_k^e+2\ln k-\ln (cL)}{\kappa}}(1+\ln k)(1+\ln_+ \frac{\Delta^e}{cL})e^{-\kappa' S_{k-1}}e^{-\kappa'\zeta}\right].\nonumber
 %& c e^{-\kappa x} \sum_{k=1}^\infty\E_{\Q^*}\left[\ind{\overline{V}(w_{[1,k]})<0, -V(w_k)\leq \frac{\log \Delta_k^e-\log (cL)}{\kappa}}(1-V(w_k))e^{-\kappa' V(w_k)}\right]\\
 %&+c e^{-\kappa x} \sum_{k=1}^\infty\E_{\Q^*}\left[\ind{\overline{V}(w_{[1,k]})<0, \frac{\log \Delta_k^e-\log (cL)}{\kappa}<-V(w_k)\leq \frac{\log \Delta_k^e+2\log k-\log (cL)}{\kappa}}(1-V(w_k))e^{-\kappa' V(w_k)}\right]
\end{align}
%can be seen as a consequence of Theorem 3(ii) in \cite{AHZ13}, where they considered the killed branching random walk and showed that there exists $c\in (0,\infty)$ such that for any $a\geq0$, as $x\rightarrow\infty$,
%\[
%\P_{-a}(\exists u\in \T: V(u)\leq -x-a, \max_{\rho<w<u}V(w)\leq 0 )\sim c R^-_s(a)e^{-\kappa x},
%\]
%with $R^-_s(a)\geq1$. %Somehow, the tail of $\tM$ without killing is hidden in their arguments. As we are more interested in the joint tail, we will mainly use the idea of \cite{Ma16}. 
Observe that for $\lambda\in(0,\kappa')$, we have $a_\lambda:=-\log\E[e^{\lambda S_1^{(\kappa)}}]>0$. Thus, for any $ a<b$, 
\begin{align*}
&\E\left[\ind{\MS_{[1,k]}<0,\ a<-S_k\leq b}e^{-\kappa' S_k}\right]\\
= & \P\left(\MS^{(\kappa)}_{[1,k]}<0, a<-S^{(\kappa)}_k\leq b\right)\leq e^{\lambda b}\E\left[e^{\lambda S_k^{(\kappa)}}\right]\leq e^{\lambda b-a_\lambda k}.
\end{align*}
It hence follows that
\begin{align*}
\E\left[\sum_{u\in\T}\ind{E_u}\ind{(F^L_u)^c}\right]\leq & c_7 e^{-\kappa x} \sum_{k=0}^\infty (1+\ln_+ k)(k+1)^{2\lambda/\kappa} e^{-a_{\lambda}k}\E\left[(1+\ln_+ \frac{\Delta^e}{cL})\left(\frac{\Delta^e}{cL}\right)^{\lambda/\kappa}e^{(\lambda-\kappa')\zeta}\right]\\
\leq & c_8 e^{-\kappa x}\E\left[(1+\ln_+ \frac{\Delta^e}{cL})\left(\frac{\Delta^e}{cL}\right)^{\lambda/\kappa}e^{(\lambda-\kappa')\zeta}\right].%=o_L(1)e^{-\kappa x}.
\end{align*}
Now observe that by Many-to-one lemma,
\begin{align*}
&\E\left[(1+\ln_+\Delta^e)(\Delta^e)^{\lambda/\kappa}e^{(\lambda-\kappa')\zeta}\right]\\
=&\E_{\Q^*}\left[e^{(\lambda-\kappa')V(w_1)}\left(\sum_{z\in\Omega(w_1)}e^{-\kappa V(z)}\right)^{\lambda/\kappa}\left(1+\ln_+\left(\sum_{z\in\Omega(w_1)}e^{-\kappa V(z)}\right)\right)\right]\\
\leq &\E\left[\sum_{|u|=1}e^{(\lambda-\kappa)V(u)}\left(\sum_{|u|=1}e^{-\kappa V(u)}\right)^{\lambda/\kappa}\left(1+\ln_+\left(\sum_{|u|=1}e^{-\kappa V(u)}\right)\right)\right]\\
\leq & \E\left[\left(\sum_{|u|=1}e^{-V(u)}\right)^{\kappa-\lambda}\left(\sum_{|u|=1}e^{-\kappa V(u)}\right)^{\lambda/\kappa}\left(1+\ln_+\left(\sum_{|u|=1}e^{-\kappa V(u)}\right)\right)\right]
\end{align*}
where the last inequality holds as $\kappa-\lambda>1$. By Cauchy-Schwartz inequality, one has hence
\begin{align*}
&\E\left[(1+\ln_+\Delta^e)(\Delta^e)^{\lambda/\kappa}e^{(\lambda-\kappa')\zeta}\right]\\
\leq & \E\left[\left(\sum_{|u|=1}e^{-V(u)}\right)^{\kappa}\right]^{(\kappa-\lambda)/\kappa}\E\left[\left(\sum_{|u|=1}e^{-\kappa V(u)}\right)\left(1+\ln_+\left(\sum_{|u|=1}e^{-\kappa V(u)}\right)\right)^{\kappa/\lambda}\right]^{\lambda/\kappa}<\infty,
\end{align*}
by Assumption \ref{cond4}. As a result, $\E\left[\sum_{u\in\T}\ind{E_u}\ind{(F^L_u)^c}\right]=o_L(1)e^{-\kappa x}$ and for $L\gg 1$, $x\geq1$,
\[
c_9 e^{-\kappa x}\leq \E[N_L(x)]\leq c_{10} e^{-\kappa x}.
\]
For the second moment,
\begin{align*}
&\E[N_L(x)^2]=\E[N_L(x)]+\E\left[\sum_{u,v\in\T; u\neq v}\ind{E_u\cap F_u^L}\ind{E_v\cap F_v^L}\right]\\
=&\E[N_L(x)]+\E\left[\sum_{u,v\in\T; u\neq v; u\wedge v\in\{u,v\}}\ind{E_u\cap F_u^L}\ind{E_v\cap F_v^L}\right]+\E\left[\sum_{u,v\in\T; u\neq v; u\wedge v\notin\{u, v\}}\ind{E_u\cap F_u^L}\ind{E_v\cap F_v^L}\right]\\
\leq & c_{10} e^{-\kappa x}+ 2 \E\left[\sum_{u\in\T}\ind{E_u\cap F_u^L}\left(\sum_{v: v>u}\ind{E_v}\right)\right]+ \E\left[\sum_{u\in\T}\ind{E_u\cap F_u^L}\left(\sum_{k=1}^{|u|}\sum_{z\in\Omega(u_k)}\sum_{v: v\geq z}\ind{E_v}\right)\right]\\
=& c_{10} e^{-\kappa x}+ 2 \sum_{n\geq1}\E\left[\sum_{|u|=n}\ind{E_u\cap F_u^L}\left(\sum_{v: v>u}\ind{E_v}\right)\right]+\sum_{n\geq1}\E\left[\sum_{|u|=n}\ind{E_u\cap F_u^L}\left(\sum_{k=1}^{n}\sum_{z\in\Omega(u_k)}\sum_{v: v\geq z}\ind{E_v}\right)\right]
\end{align*}
which conditionally on $\nf^V_n$ or on $\sigma((u_k, V(u_k))_{1\leq k\leq |u|})$, is bounded by
\begin{multline*}
 c_{10}e^{-\kappa x}+ 2 \sum_{n\geq1}\E\left[\sum_{|u|=n}\ind{E_u\cap F_u^L}\E[N(x+V(u))\vert V(u)]\right]\\
+\sum_{n\geq1}\E\left[\sum_{|u|=n}\ind{E_u\cap F_u^L}\left(\sum_{k=1}^{n}\sum_{z\in\Omega(u_k)}\E[N(x+V(z))\vert V(z)]\right)\right].
\end{multline*}
In view of \eqref{meanNx}, for $L\geq1$, we have
\begin{align*}
\E[N_L(x)^2]=&  c_{10}e^{-\kappa x}+ c_{11} \sum_{n\geq1}\E\left[\sum_{|u|=n}\ind{E_u\cap F_u^L}e^{-\kappa (V(u)+x)}\right]\\
&\quad +c_{11}\sum_{n\geq1}\E\left[\sum_{|u|=n}\ind{E_u\cap F_u^L}\left(\sum_{k=1}^{n}\sum_{z\in\Omega(u_k)}e^{-\kappa V(z)-\kappa x}\right)\right]\\
\leq & c_{10}e^{-\kappa x}+ c_{11} \E[N(x)]+ c_{11}L \E[N(x)]\leq c_{12} L e^{-\kappa x}.
\end{align*}
Therefore, for $L$ sufficiently large and fixed, we conclude from \eqref{PaleyZygmund} that
\[
\P(\tM\leq -x)\geq c_{13} e^{-\kappa x}.
\]
\end{proof}

Let us introduce some notation here. For any $k\geq0$, let
\[
\tM_k:=\inf_{|u|\leq k}V(u).
\]
 For any $u\in\T\setminus\{\rho\}$, recall that $\Delta V(u):=V(u)-V(\pa{u})$. Let
\[
\tM^{(u)}:=\inf_{v: v\geq u}(V(v)-V(u)),\textrm{ and } \tM^{(u)}_k:=\inf_{v: v\geq u, |v|\leq |u|+k}(V(v)-V(u)).
\]
Write $\kappa'$ for $\kappa-1$.

\begin{proof}[Proof of Lemma \ref{BRWmaincvg}] In fact, we only need to show the convergence of $e^{\kappa x}\E\left[\phi(\mathcal{W}^{u^*,\leq t})\ind{\tM\in I(x)}\right]$.

Recall that $u^*$ is chosen at random among the youngest individuals attaining $\tM$. Then, observe that
\begin{align}\label{spineuptomin}
&\E\left[\phi(\mathcal{W}^{u^*,\leq t})\ind{\tM\in I(x)}\right]=\sum_{k\geq1}\E\left[\frac{1}{\sum_{|v|=k}\ind{V(v)=\tM}}\sum_{|u|=k}\ind{V(u)=\tM, V(u)\in I(x)}\phi(\mathcal{W}^{u, \leq t})\right]\nonumber\\
=&\sum_{k\geq1}\E\left[\sum_{|u|=k}\frac{1}{\sum_{|v|=k}\ind{V(v)=V(u)}}\ind{V(u)=\tM_k<\tM_{k-1}, V(u)\in I(x)}\ind{\inf_{|z|>k}V(z)\geq V(u)}\phi(\mathcal{W}^{u,\leq t})\right]\nonumber\\
=&\sum_{k\geq1}\E\left[\sum_{|u|=k}\frac{1}{\sum_{|v|=k}\ind{V(v)=V(u)}}\ind{V(u)=\tM_k<\tM_{k-1}, V(u)\in I(x)}\E\left[\ind{\inf_{|z|>k}V(z)\geq V(u)}\phi(\mathcal{W}^{u,\leq t})\mid \mathcal{F}_k^V\right]\right]. 
\end{align}
Notice that
\[
\mathcal{W}^{u,\leq t}=\mathcal{W}_0^{u,\leq t}+W^{(u)}_\infty,
\]
where $\mathcal{W}_0^{u,\leq t}:=\sum_{j=k-t}^k e^{V(u)-V(u_{j-1})}\sum_{z\in\Omega(u_j)}e^{-\Delta V(z)}W^{(z)}_\infty$ with $\Delta V(z)=V(z)-V(\overleftarrow{z})$. This yields
\begin{align*}
\E\left[\ind{\inf_{|z|>k}V(z)\geq V(u)}\phi(\mathcal{W}^{u,\leq t})\mid \mathcal{F}_k^V\right]&=\E[\phi(\mathcal{W}_0^{u,\leq t}+W_\infty^{(u)})\ind{\tM^{(u)}\geq 0}\mid \mathcal{F}_k^V]\times\prod_{|z|=k,z\neq u}\overline{F}(V(u)-V(z))\\
&=\Phi(\mathcal{W}_0^{u,\leq t})\times\prod_{|v|=k,v\neq u}\overline{F}(V(u)-V(v)), 
\end{align*}
where $\overline{F}(t)=\P(\tM\geq t)$, and
\[
\Phi(a):=\E[\phi(a+W_\infty)\ind{\tM\geq0}],\ \forall a\geq0.
\]
Since $\phi$ is continuous and bounded, so is $\Phi$. By the many-to-one lemma, equation~\eqref{spineuptomin} becomes
\begin{align}\label{eq:spineuptomin2}
&\E\left[\phi(\mathcal{W}^{u^*,\leq t})\ind{\tM\in I(x)}\right]\\
=&\sum_{k\geq1}\E_{\Q^*}\left[\frac{e^{V(w_k)}}{\sum_{|v|=k}\ind{V(v)=V(w_k)}}\ind{V(w_k)=\tM_k<\tM_{k-1}, V(w_k)\in I(x)}\Phi(\mathcal{W}_0^{w_k,\leq t})\times\prod_{|v|=k,v\neq w_k}\overline{F}(V(w_k)-V(v))\right] \nonumber\\
=&\sum_{k\geq1}\E_{\Q^*}\left[\frac{e^{V(w_k)}}{\sum_{|v|=k}\ind{V(v)=V(w_k)}}\ind{V(w_k)=\tM_k<\tM_{k-1}, V(w_k)\in I(x)}\Phi(\mathcal{W}_0^{w_k,\leq t})\times\prod_{|v|=k,v\neq w_k}\ind{\tM^{(v)}+V(v)\geq V(w_k)}\right]. \nonumber
\end{align}
%
%et
%which by the change of measures $d\Q^{(k)}\otimes\P=W_kd\P$, is
%\begin{equation}\label{spineuptomin}
%\sum_{k\geq1}\E_{\Q^{(k)}\otimes\P}\left[\frac{e^{V(w_k)}}{\sum_{|v|=k}\ind{V(v)=V(w_k)}}\phi(\mathcal{W}^{w_k,\leq t}); V(w_k)=\tM<\tM_{k-1}, V(w_k)\in I(x)\right]
%\end{equation}
Here, one sees that
\[
\sum_{|v|=k}\ind{V(v)=V(w_k)}=1+\sum_{j=1}^k\sum_{z\in\Omega(w_j)}\sum_{|v|=k, v\geq z}\ind{V(v)=V(w_k)},
\]
The event $\{V(w_k)=\tM_k<\tM_{k-1}\}\cap\bigcap_{|v|=k,v\neq w_k}\{\tM^{(v)}+V(v)\geq V(w_k)\}$ can be rewritten as
\begin{multline*}
\{V(w_k)<\underline{V}(w_{[0,k-1]})\}\cap\left(\bigcap_{z\in\cup_{j=1}^k\Omega(w_j)}\{V(z)+\tM_{k-j-1}^{(z)}>V(w_k), V(z)+\tM^{(z)}\geq V(w_k)\}\right).
\end{multline*}
Then, \eqref{eq:spineuptomin2} becomes
\[
\E\left[\phi(\mathcal{W}^{u^*,\leq t})\ind{\tM\in I(x)}\right]=\sum_{k\geq1}\chi_k,
\]
where 
\begin{multline*}
\chi_k:=\E_{\Q^*}\left[\frac{e^{V(w_k)}\ind{V(w_k)\in I(x)}\ind{ V(w_k)<\underline{V}(w_{[0,k-1]})}}{1+\sum_{j=1}^k\sum_{y\in\Omega(w_j)}\sum_{|v|=k, v\geq y}\ind{V(v)=V(w_k)}}\Phi(\mathcal{W}_0^{w_k,\leq t})\right.\\
\left.\times\prod_{j=1}^k\prod_{z\in\Omega(w_j)}\ind{\Delta V(z)+M_{k-j-1}^{(z)}>V(w_k)-V(w_{j-1}), \Delta V(z)+M^{(z)}\geq V(w_k)-V(w_{j-1})}\right].
\end{multline*}
To simplify this expression, we observe that only sufficiently large $k$ make contributions to the sum. Let us first prove the following result. 
For any $b\geq1$, 
\begin{equation}\label{smallk}
\lim_{x\rightarrow\infty}e^{\kappa x}\sum_{k\leq b}\E_{\Q^*}\left[e^{V(w_k)}; V(w_k)\in I(x),V(w_k)<\underline{V}(w_{[0,k-1]}) \right]=0,
\end{equation}
which means that as $x\rightarrow\infty$,
\begin{equation}
\E\left[\phi(\mathcal{W}^{u^*,\leq t})\ind{\tM\in I(x)}\right]=\sum_{k\geq b}\chi_k+o(e^{-\kappa x}).
\end{equation}
In fact, by many-to-one lemma,
\begin{align*}
&e^{\kappa x}\sum_{k\leq b}\E_{\Q^*}\left[e^{V(w_k)}; V(w_k)\in I(x),V(w_k)<\underline{V}(w_{[0,k-1]})\right]\\
%=&e^{\kappa x}\sum_{k\leq b}\E\left[\sum_{|u|=k}\ind{V(u)\leq -x, V(u)<\min_{\rho\leq v\leq u}V(v)}\right]\\
\leq &e^{\kappa x}\sum_{k\leq b}\E\left[e^{\kappa S_k^{(\kappa)}}; S_k^{(\kappa)}\leq -x, S_k^{(\kappa)}<\mS_{[0,k-1]}^{(\kappa)}\right]
\leq \sum_{k\leq b}\P\left(S_k^{(\kappa)}\leq -x\right),
\end{align*}
Recall that $\E[e^{-\delta S_1^{(\kappa)}}]<\infty$ for some $\delta>0$. Then, by Markov inequality, 
\[
e^{\kappa x}\sum_{k\leq b}\E_{\Q^*}\left[e^{V(w_k)}; V(w_k)\leq I(x),V(w_k)<\underline{V}(w_{[0,k-1]})\right]\leq\sum_{k\leq b}e^{-\delta x}\E[e^{-\delta S_1^{(\kappa)}}]^k=o_x(1),
\]
as $x\rightarrow\infty$ for any $b\geq1$ fixed. This implies that 
\begin{equation}\label{smallk0}
\sum_{k\leq b}e^{\kappa x}\chi_k=o_x(1).
\end{equation}
So it suffices to study $\chi_k$ for $k$ sufficiently large. Next, for any integer $k\geq b_1\geq1$, let us introduce the event
\[
E_k(b_1):=\{\forall j\leq k-b_1, \forall z\in \Omega(w_j), V(z)+\tM^{(z)}\geq V(w_k)+1\}.
\]
We claim that
\begin{equation}\label{awayspine}
\lim_{b_2\geq b_1\rightarrow\infty}\lim_{x\rightarrow\infty}e^{\kappa x}\sum_{k\geq b_2}\E_{\Q^*}\left[e^{V(w_k)}; V(w_k)\in I(x),V(w_k)<\underline{V}(w_{[0,k-1]}); E_k(b_1)^c\right]=0.
\end{equation}
Let $\E_{k}^{\eqref{awayspine}}$ denote the expectation in \eqref{awayspine}, observe that
\begin{align*}
&\E_{k}^{\eqref{awayspine}}\leq\sum_{j=1}^{k-b_1}\E_{\Q^*}\left[e^{V(w_k)}\ind{V(w_k)\in I(x),V(w_k)<\underline{V}(w_{[0,k-1]})}; \exists z\in \Omega(w_j), V(z)+\tM^{(z)}< V(w_k)+1\right]\\
=& \sum_{j=1}^{k-b_1}\E_{\Q^*}\left[e^{V(w_k)}\ind{V(w_k)\in I(x),V(w_k)<\underline{V}(w_{[0,k-1]})}\Q^*\left(\exists z\in \Omega(w_j), V(z)+\tM^{(z)}< V(w_k)+1\vert \nf_k^{w}\right)\right]
\end{align*}
where $\nf_k^{w}:=\sigma\{ (w_i, V(w_i))_{1\leq i\leq k}, (V(u), u\in\Omega(w_i))_{1\leq i \leq k}\}$. Given the spine, one sees that
\[
\Q^*\left(\exists z\in \Omega(w_j), V(z)+\tM^{(z)}< V(w_k)+1 \Big\vert \nf_k^w\right)\leq 1\wedge\left( \sum_{z\in \Omega(w_j)}\P(\tM\leq x)\vert_{x=V(w_k)-V(z)+1}\right)%\leq 1\wedge \sum_{z\in\Omega(w_j)}e^{\kappa (V(w_k)-V(z)+1)}=1\wedge e^{\kappa(V(w_k)-V(w_{j-1}))+1}\Delta_j^e,
\]
which by Lemma \ref{BRWroughbd} is bounded by 
\[
 1\wedge \left(\sum_{z\in\Omega(w_j)}e^{\kappa (V(w_k)-V(z)+1)}\right)=1\wedge \left(e^{\kappa(V(w_k)-V(w_{j-1}))+1}\Delta_j^e\right),
\]
where $\Delta_j^e=\sum_{z\in\Omega(w_j)}e^{-\kappa\Delta V(z)}$. It then follows that
\begin{align*}
\sum_{k\geq b_2}\E_{k}^{\eqref{awayspine}}\leq&\sum_{k\geq b_2}\sum_{j=1}^{k-b_1}\E_{\Q^*}\left[\left(1\wedge e^{\kappa(V(w_k)-V(w_{j-1}))}\Delta_j^e\right)e^{V(w_k)}\ind{V(w_k)\in I(x),V(w_k)<\underline{V}(w_{[0,k-1]})}\right]\\
%\leq & e^{-x}\sum_{k\geq b_2}\sum_{j=1}^{k-b_1}\E_\Q\left[\left(1\wedge e^{\kappa(V(w_k)-V(w_{j-1}))}\Delta_j^e\right); V(w_k)\in I(x),V(w_k)<\min_{0\leq j\leq k-1}V(w_j)\right]\\
=& \sum_{k\geq b_2}\sum_{j=1}^{k-b_1}\E_{\Q^*}\left[\left(1\wedge e^{\kappa V(w_{k-j+1})}\Delta_{k-j+1}^e\right)e^{V(w_k)}\ind{V(w_k)\in I(x), \overline{V}(w_{[1,k]})<0}\right]\\
=& \sum_{k\geq b_2}\sum_{j=b_1+1}^k\E_{\Q^*}\left[\left(1\wedge e^{\kappa V(w_{j})}\Delta_{j}^e\right)e^{V(w_k)}; V(w_k)\in I(x),  \overline{V}(w_{[1,k]})<0\right]
\end{align*}
where the equality follows from the fact that $(\Delta V(w_j), \sum_{z\in\Omega(w_j)}\delta_{\Delta V(z)})_{j=1,\cdots, k}$ are i.i.d. and has the same distribution as $(\Delta V(w_j), \sum_{z\in\Omega(w_j)}\delta_{\Delta V(z)})_{j=k,k-1,\cdots, 1}$. By the Markov property at time $j$ and~\eqref{RWeSbd}, one sees that
\begin{align*}
&\sum_{k\geq b_2}\E_{k}^{\eqref{awayspine}}\leq \sum_{k\geq b_2}\sum_{j=b_1+1}^k\E_{\Q^*}\left[\left(1\wedge e^{\kappa V(w_{j})}\Delta_{j}^e\right)\ind{ \overline{V}(w_{[1,j]})<0}\E_{V(w_j)}\left[e^{S_{k-j}}\ind{S_{k-j}\in I(x), \MS_{[1,k-j]}<0}\right]\right]\\
&\leq \sum_{j=b_1+1}^\infty\E_{\Q^*}\left[\left(1\wedge e^{\kappa V(w_{j})}\Delta_{j}^e\right)\ind{ \overline{V}(w_{[1,j]})<0}\sum_{k\geq j}\E_{V(w_j)}\left[e^{S_{k-j}}\ind{S_{k-j}\in I(x), \MS_{[1,k-j]}<0}\right]\right]\\
&\leq c_2e^{-\kappa x}\sum_{j=b_1+1}^\infty\E_{\Q^*}\left[\left(1\wedge e^{\kappa V(w_{j})}\Delta_{j}^e\right)\ind{ \overline{V}(w_{[1,j]})<0}e^{-\kappa'V(w_j)}(1-V(w_j))\right]
\end{align*}
Moreover, observe that for any $L>0$ and $\alpha>0$, $1\wedge e^{\kappa V(w_{j})}\Delta_{j}^e \leq e^{L+(\kappa-\alpha) V(w_j)}+\ind{\ln \Delta_j^e\geq L-\alpha V(w_j)}$. So, 
\begin{align*}
&\left(1\wedge e^{\kappa V(w_{j})}\Delta_{j}^e\right)e^{-\kappa' V(w_j)}(1-V(w_j)) \\
\leq  & e^{L+(1-\alpha) V(w_j)}(1-V(w_j))+(1-V(w_j))e^{-\kappa' V(w_j)}\ind{\ln \Delta_j^e\geq L-\alpha V(w_j)}
%\leq & e^{L+(1-\alpha) V(w_j)}(1-V(w_j))+e^{-\kappa'\Delta V(w_j)}(1+\frac{1}{\alpha}\ln \Delta_j^e)_+e^{-\kappa'V(w_{j-1})}\ind{-\alpha V(w_{j-1})\leq \alpha\Delta V(w_j)+\ln \Delta_j^e-L}
\end{align*}
It hence follows that
\begin{align}\label{upbdawayspine}
\sum_{k\geq b_2}e^{\kappa x}\E_{k}^{\eqref{awayspine}}\leq&  c_2 e^L\sum_{j=b_1+1}^\infty\E_{\Q^*}\left[(1-V(w_j))e^{(1-\alpha)V(w_j)}\ind{\overline{V}(w_{[1,j]})<0}\right]\\
&+c_2\sum_{j=b_1+1}^\infty\E_{\Q^*}\left[(1-V(w_j))e^{-\kappa' V(w_j)}\ind{-V(w_j)\leq (\ln \Delta_j^e-L)/\alpha;\overline{V}(w_{[1,j]})<0}\right]\nonumber
\end{align}
Note that $\{V(w_i); 1\leq i\leq j-1\}$ which is distributed as $(S_i; 1\leq i\leq j-1)$, is independent of $(\Delta V(w_j),\Delta_j^e)$. Therefore, the first sum on the right hand side of \eqref{upbdawayspine} is
\begin{align*}
&e^L\sum_{j=b_1+1}^\infty\E_{\Q^*}\left[(1-V(w_j))e^{(1-\alpha)V(w_j)}\ind{\overline{V}(w_{[1,j]})<0}\right]\\
=&e^L\sum_{j=b_1+1}^\infty\E\left[(1-S_j)e^{(1-\alpha)S_j}\ind{\MS_{[1,j]}<0}\right]=e^Lo_{b_1}(1),
\end{align*}
because $\sum_{j=0}^\infty\E\left[(1-S_j)e^{(1-\alpha)S_j}\ind{\MS_{[1,j]}<0}\right]=\int_0^\infty(1+x)e^{-(\kappa-\alpha)x}U_s^{(\kappa),-}(dx)<\infty$ for any $\alpha\in(0,\kappa)$. Let us take $\alpha=1$. Then, similarly as \eqref{newcouplebd}, the second sum on the right hand side of \eqref{upbdawayspine} is
\begin{align*}
&\sum_{j=b_1+1}^\infty\E_{\Q^*}\left[(1-V(w_j))e^{-\kappa' V(w_j)}\ind{-V(w_j)\leq (\ln \Delta_j^e-L)/\alpha;\overline{V}(w_{[1,j]})<0}\right]\\
\leq& c\E\left[e^{-\kappa'\zeta}(1+\ln \Delta^e)\ind{\ln \Delta^e\geq L}\E\left[\sum_{j\geq0}e^{-\kappa' S_j}\ind{-S_j\leq \zeta+\ln \Delta^e-L; \MS_{[1,j]}<0}\Big\vert \zeta, \Delta^e\right]\right]
%\leq & e^L o_{b_1}(1)+ce^{-\frac{\kappa'}{\alpha}L}\E_\Q\left[(\Delta^e)^{\kappa'/\alpha}(1+\frac{1}{\alpha}\ln \Delta^e)\ind{\alpha\zeta+\ln\Delta^e\geq L}\right],
\end{align*}
Here  $\E[\sum_{j\geq0}e^{-\kappa' S_j}\ind{-S_j\leq x, \MS_{[1,j]}<0}]=U_s^{(\kappa),-}[0,x]\leq \cb_R(1+x)$ for any $x\geq0$ and $\cb_R>0$. As a consequence,
\begin{align}\label{intermediateupbd}
&\sum_{j=b_1+1}^\infty\E_{\Q^*}\left[(1-V(w_j))e^{-\kappa' V(w_j)}\ind{-V(w_j)\leq (\ln \Delta_j^e-L)/\alpha;\overline{V}(w_{[1,j]})<0}\right]\\
\leq &c_{14}\E\left[e^{-\kappa'\zeta}(1+\ln \Delta^e)\ind{\ln \Delta^e\geq L}(1+\zeta_+ +\ln\Delta^e-L)\right]\nonumber\\
\leq &c_{14} \E\left[e^{-\kappa'\zeta}(1+(\ln \Delta^e)^2+\zeta^2)\ind{\ln \Delta^e\geq L}\right]\nonumber
\end{align}
which is
\begin{align*}
&c_{14}\E_{\Q^*}\left[e^{-\kappa'V(w_1)}(1+(\ln(\sum_{u\in\Omega(w_1)}e^{-\kappa V(u)}))^2+V(w_1)^2)\ind{\ln(\sum_{u\in\Omega(w_1)}e^{-\kappa V(u)})\geq L}\right]\\
\leq & c_{14}\E\left[\Delta_+^e(1+\ln \Delta_+^e)^2\ind{\ln \Delta_+^e\geq L}\right]+\E\left[(S_1^{(\kappa)})^2\ind{\ln\Delta_+^e\geq L}\right]
\end{align*}
where $\Delta^e_+:=\sum_{|u|=1}e^{-\kappa V(u)}$. By Assumption \ref{cond4}, $\E\left[\Delta_+^e(1+\ln \Delta_+^e)^2\right]+\E\left[(S_1^{(\kappa)})^2\right]<\infty$. We hence deduce that
\[
\sum_{j=b_1+1}^\infty\E_{\Q^*}\left[(1-V(w_j))e^{-\kappa' V(w_j)}\ind{-V(w_j)\leq (\ln \Delta_j^e-L)/\alpha;\overline{V}(w_{[1,j]})<0}\right]=o_L(1).
\]
Therefore, $\sum_{k\geq b_2}e^{\kappa x}\E_{k}^{\eqref{awayspine}}\leq e^L o_{b_1}(1)+o_L(1)$. Recall that $b_2\geq b_1$. Letting $b_1\rightarrow\infty$ then $L\rightarrow\infty$ leads to \eqref{awayspine}, i.e.,
\[
\lim_{b_2\geq b_1\rightarrow\infty}\sum_{k\geq b_2}e^{\kappa x}\E_{k}^{\eqref{awayspine}}=0.
\]
Let 
\begin{align*}
\chi_k(b_1):=&\E_{\Q}\left[\frac{e^{V(w_k)}\ind{V(w_k)\in I(x)}}{1+\sum_{j=k-b_1+1}^k\sum_{y\in\Omega(w_j)}\sum_{|v|=k, v\geq y}\ind{V(v)=V(w_k)}}\Phi(\mathcal{W}_0^{w_k,\leq t})\ind{ V(w_k)<\overline{V}(w_{[0,k-1]})}\right.\\
&\qquad\left.\times\prod_{j=k-b_1+1}^k\prod_{z\in\Omega(w_j)}\ind{\Delta V(z)+M_{k-j-1}^{(z)}>V(w_k)-V(w_{j-1}), \Delta V(z)+M^{(z)}\geq V(w_k)-V(w_{j-1})}\right].
\end{align*}
In view of \eqref{smallk0} and \eqref{awayspine}, it suffices to study the convergence of
\begin{equation}
\sum_{k\geq b_2}e^{\kappa x}\chi_k(b_1).
\end{equation}
Let for $u\in\T$, $\en(u):=(v,V(v)-V(u))_{v\geq u}$. Observe that for $b_1>t$, as the random variables $(\Delta V(w_i), \sum_{u\in\Omega(w_i)}\delta_{(\Delta V(u), \en(u))})_{1\leq i\leq k}$ are i.i.d. hence exchangeable, 
\begin{align*}
\chi_k(b_1)=&\E_\Q\left[\frac{e^{V(w_k)}\ind{\overline{V}(w_{[1,k]})<0, V(w_k)\in I(x)}}{1+\sum_{j=1}^{b_1}\sum_{y\in\Omega(w_j)}\sum_{v\in\T^{(y)}_{j-1}}\ind{V^{(y)}(v)=V(w_j)-\Delta V(y)}}\Phi(\overleftarrow{\mathcal{W}}_0^{\leq t})\ind{\Xi_{b_1}}\right]
\end{align*}
where
\[
\overleftarrow{\mathcal{W}}_0^{\leq t}=\sum_{j=1}^{t+1}e^{V(w_j)}\sum_{z\in\Omega(w_j)}e^{-\Delta V(z)}W_\infty^{(z)}, \textrm{ and } V_y(v)=: V(v)-V(y), \forall v\geq y,
\]
and
\[
\T_j^{(y)}:=\{v\in\T\ :\ v\geq y,\ |v|=|y|+j\},\ \text{and}\ \ind{\Xi_{b_1}}:=\prod_{j=1}^{b_1}\prod_{z\in\Omega(w_j)}\ind{\tM^{(z)}\geq V(w_j)-\Delta V(z), \tM^{(z)}_{j-2}>V(w_j)-\Delta V(z)}.
\]
By the Markov property at time $b_2\geq b_1$ and Lemma \ref{ManytoOne}, one gets that
\begin{align*} 
\sum_{k\geq b_2}e^{\kappa x}\chi_k(b_1)=&\E_\Q\left[\frac{e^{V(w_{b_2})}\ind{\overline{V}(w_{[1,b_2]}<0)}\E_{V(w_{b_2})}\left[\sum_{k\geq0}e^{S_k+\kappa x}\ind{\MS_{[1,k]}<0, S_k\in I(x)}\right]}{1+\sum_{j=1}^{b_1}\sum_{y\in\Omega(w_j)}\sum_{v\in\T^{(y)}_{j-1}}\ind{V^{(y)}(v)=V(w_j)-\Delta V(y)}} \Phi(\overleftarrow{\mathcal{W}}_0^{\leq t})\ind{\Xi_{b_1}}\right],
\end{align*}
where by \eqref{RWeScvg},
\[
\E_{V(w_{b_2})}\left[\sum_{k\geq0}e^{S_k+\kappa x}\ind{\MS_{[1,k]}<0, S_k\in I(x)}\right]\rightarrow e^{-\kappa' V(w_{b_2})}\left(\frac{1-e^{-\kappa}}{\kappa}\right) C_s^{(\kappa),-} U_w^{(\kappa),+}([0,-V(w_{b_2}))),
\]
as $x\rightarrow\infty$. Therefore, with $C_\kappa:=\left(\frac{1-e^{-\kappa}}{\kappa}\right) C_s^{(\kappa),-} $,
\begin{align*}
\lim_{x\rightarrow\infty}\sum_{k\geq b_2}e^{\kappa x}\chi_k(b_1)=&C_\kappa\E_\Q\left[\frac{e^{(2-\kappa) V(w_{b_2})}U_w^{(\kappa),+}([0,-V(w_{b_2})))\ind{\overline{V}(w_{[1,b_2]}<0)}}{1+\sum_{j=1}^{b_1}\sum_{y\in\Omega(w_j)}\sum_{v\in\T^{(y)}_{j-1}}\ind{V^{(y)}(v)=V(w_j)-\Delta V(y)}} \Phi(\overleftarrow{\mathcal{W}}_0^{\leq t})\ind{\Xi_{b_1}}\right]
\end{align*}
Note that $\sum_{k\geq b_2}e^{\kappa x}\chi_k(b_1)$ is non-increasing in $b_1$ and $b_2$. We thus deduce that
\begin{equation}
\lim_{x\rightarrow\infty}e^{\kappa x}\E\left[\phi(\mathcal{W}^{u^*,\leq t})\ind{\tM\in I(x)}\right]= (1-e^{-\kappa}) \mathcal{E}_t(\phi)
\end{equation}
with
\begin{align}\label{eq:defE}
\mathcal{E}_t(\phi):=&\lim_{b_1\rightarrow\infty}\lim_{b_2\rightarrow\infty} \frac{C_\kappa}{1-e^{-\kappa}} \E_\Q\left[\frac{e^{(2-\kappa) V(w_{b_2})}U_w^{(\kappa),+}([0,-V(w_{b_2})))\ind{\overline{V}(w_{[1,b_2]}<0)}}{1+\sum_{j=1}^{b_1}\sum_{y\in\Omega(w_j)}\sum_{v\in\T^{(y)}_{j-1}}\ind{V^{(y)}(v)=V(w_j)-\Delta V(y)}}\right.\nonumber\\
&\qquad\qquad\qquad\times \Phi(\overleftarrow{\mathcal{W}}_0^{\leq t})\left. \prod_{j=1}^{b_1}\prod_{z\in\Omega(w_j)}\ind{\tM^{(z)}\geq V(w_j)-\Delta V(z), \tM^{(z)}_{j-2}>V(w_j)-\Delta V(z)}\right]
\end{align}
Similarly, for any $j\geq0$, one could prove that
\[
\lim_{x\rightarrow\infty}e^{\kappa x}\E\left[\phi(\mathcal{W}^{u^*,\leq t})\ind{\tM\in I(x+j)}\right]= (1-e^{-\kappa}) e^{-\kappa j} \mathcal{E}_t(\phi)
\]
which implies Lemma \ref{BRWmaincvg}.

\end{proof}

\begin{proof}[Proof of Lemma \ref{BRWrest}]
In fact, it suffices to show that for any $\delta>0$,
\begin{equation}
\lim_{t\rightarrow\infty}\sup_{x\in\R_+} e^{\kappa x}\P\left(\sum_{k=1}^{|u^*|-t}\sum_{z\in\Omega(u^*_k)}e^{\tM-V(z)}W_\infty^{(z)}\geq \delta, \tM\in I(x)\right)=0.
\end{equation}
Again by change of measures, one has
\begin{align*}
&e^{\kappa x}\P\left(\sum_{k=1}^{|u^*|-t}\sum_{z\in\Omega(u^*_k)}e^{\tM-V(z)}W_\infty^{(z)}\geq \delta, \tM\in I(x)\right)\\
\leq & \sum_{k\geq t+1}e^{\kappa x}\E\left[\sum_{|u|=k}\ind{u=u^*, V(u)\in I(x)}\ind{\sum_{j=1}^{k-t}\sum_{z\in\Omega(u_j)}e^{\tM-V(z)}W_\infty^{(z)}\geq \delta}\right]\\
%\leq &\sum_{k\geq t+1}\E_\Q\left[e^{V(w_k)}\ind{V(w_k)<\tM_{k-1}, V(w_k)\in I(x)}\ind{\sum_{j=1}^{k-t}\sum_{z\in\Omega(w_j)}e^{V(w_k)-V(z)}W_\infty^{(z)}\geq \delta}\right]\\
\leq &\sum_{k\geq t+1}e^{\kappa x}\E_{\Q^*}\left[e^{V(w_k)}\ind{V(w_k)<\underline{V}(w_{[0,k-1]}), V(w_k)\in I(x)}\ind{\sum_{j=1}^{k-t}\sum_{z\in\Omega(w_j)}e^{V(w_k)-V(z)}W_\infty^{(z)}\geq \delta}\right].
\end{align*}
One sees that conditionally on the spine $\nf_k^w$, by Markov's inequality,
\[
\Q^*\left(\sum_{j=1}^{k-t}\sum_{z\in\Omega(w_j)}e^{V(w_k)-V(z)}W_\infty^{(z)}\geq \delta\vert \nf_k^w\right)\leq 1\wedge \left(\sum_{j=1}^{k-t}\sum_{z\in\Omega(w_j)}e^{V(w_k)-V(z)}\frac{1}{\delta}\right),
\]
since $\E[W_\infty]=1$. Note that if we write $\tilde{\Delta}^e_j=\sum_{z\in\Omega(w_j)}e^{-\Delta V(z)}$, then $\sum_{z\in\Omega(w_j)}e^{V(w_k)-V(z)}=e^{V(w_k)-V(w_{j-1})}\tilde{\Delta}^e_j$ with $V(w_k)\in I(x)$. So,
\begin{align*}
1\wedge \left(\sum_{j=1}^{k-t}\sum_{z\in\Omega(w_j)}e^{V(w_k)-V(z)}\frac{1}{\delta}\right)\leq & \sum_{j=1}^{k-t}\left(\frac{e^{V(w_k)-V(w_{j-1})}\tilde{\Delta}^e_j}{\delta}\wedge 1\right)\\
\leq & \sum_{j=1}^{k-t}\frac{1}{\delta} e^{-(V(w_{j-1})-V(w_k))(1-\alpha)}+\sum_{j=1}^{k-t}\ind{\ln \tilde{\Delta}^e_j\geq \alpha(V(w_{j-1})-V(w_k))}.
\end{align*}
It follows that
\begin{align}\label{rgbdrest}
&e^{\kappa x}\P\left(\sum_{k=1}^{|u^*|-t}\sum_{z\in\Omega(u^*_k)}e^{\tM-V(z)}W_\infty^{(z)}\geq \delta, \tM\in I(x)\right)\\
\leq & \sum_{k\geq t+1}e^{\kappa x}\E_{\Q^*}\left[e^{V(w_k)}\ind{V(w_k)<\underline{V}(w_{[0,k-1]}), V(w_k)\in I(x)} \sum_{j=1}^{k-t}\frac{1}{\delta} e^{-(V(w_{j-1})-V(w_k))(1-\alpha)}\right]\nonumber\\
&+ \sum_{k\geq t+1}e^{\kappa x}\E_{\Q^*}\left[e^{V(w_k)}\ind{V(w_k)<\underline{V}(w_{[0,k-1]}), V(w_k)\in I(x)} \sum_{j=1}^{k-t}\ind{\ln \tilde{\Delta}^e_j\geq \alpha(V(w_{j-1})-V(w_k))}\right]\nonumber
\end{align}
We only need to prove that the right hand side of \eqref{rgbdrest} is $o_t(1)$ as $t\rightarrow\infty$. For the first sum of the right hand side, by time reversing then the Markov property,
\begin{align*}
&\sum_{k\geq t+1}e^{\kappa x}\E_{\Q^*}\left[e^{V(w_k)}\ind{V(w_k)<\underline{V}(w_{[0,k-1]}), V(w_k)\in I(x)} \sum_{j=1}^{k-t}\frac{1}{\delta} e^{-(V(w_{j-1})-V(w_k))(1-\alpha)}\right]\\
%=&\sum_{k\geq t+1}\sum_{j=1}^{k-t}\frac{1}{\delta}\E_\Q\left[e^{V(w_k)}\ind{V(w_k)<\underline{V}(w_{[0,k-1]}), V(w_k)\in I(x)} e^{-(V(w_{j-1})-V(w_k))(1-\alpha)}\right]\\
=&\sum_{k\geq t+1}\sum_{j=t+1}^{k}\frac{1}{\delta}\E\left[e^{S_k+\kappa x}\ind{\MS_{[1,k]}<0, S_k\in I(x)} e^{(1-\alpha)S_{j}}\right]\\
=&\sum_{j\geq t+1}\frac{1}{\delta}\E\left[e^{(1-\alpha)S_j}\ind{\MS_{[1,j]}<0}\sum_{k\geq0}\E_{S_j}[e^{S_k+\kappa x}\ind{\MS_{[1,k]}<0, S_k\in I(x)}]\right],
\end{align*}  
which by \eqref{RWeSbd} and for $\alpha\in(0,1)$ is bounded by 
\begin{align*}
%&\sum_{k\geq t+1}e^{\kappa x}\E_{\Q^*}\left[e^{V(w_k)}\ind{V(w_k)<\underline{V}(w_{[0,k-1]}), V(w_k)\in I(x)} \sum_{j=1}^{k-t}\frac{1}{\delta} e^{-(V(w_{j-1})-V(w_k))(1-\alpha)}\right]\\
\frac{c_2}{\delta}\sum_{j\geq t+1}\E\left[e^{(1-\alpha)S_j}\ind{\MS_{[1,j]}<0}(1-S_j)e^{-\kappa' S_j}\right]=o_t(1),
\end{align*}
since $\sum_{j=0}^\infty\E\left[(1-S_j)e^{(1-\alpha)S_j}e^{-\kappa' S_j}\ind{\MS_{[1,j]}<0}\right]=\int_0^\infty(1+x)e^{-(1-\alpha)x}U_s^{(\kappa),-}(dx)<\infty$.

On the other hand, for the second sum on the right hand side of \eqref{rgbdrest},
\begin{align*}
& \sum_{k\geq t+1}e^{\kappa x}\E_{\Q^*}\left[e^{V(w_k)}\ind{V(w_k)<\underline{V}(w_{[0,k-1]}), V(w_k)\in I(x)} \sum_{j=1}^{k-t}\ind{\ln \tilde{\Delta}^e_j\geq \alpha(-V(w_k)+V(w_{j-1}))}\right]\\
%=&\sum_{k\geq t+1} \sum_{j=1}^{k-t}e^{\kappa x}\E_{\Q^*}\left[e^{V(w_k)}\ind{V(w_k)<\underline{V}(w_{[0,k-1]}), V(w_k)\in I(x)} \ind{\ln \tilde{\Delta}^e_j\geq -\alpha(V(w_k)-V(w_{j-1}))}\right]\\
=&\sum_{k\geq t+1} \sum_{j=t+1}^{k}\E_{\Q^*}\left[e^{V(w_k)}\ind{\overline{V}(w_{[1,k]})<0, V(w_k)\in I(x)}\ind{\ln \tilde{\Delta}^e_{j}\geq -\alpha V(w_{j})}\right]\\
=&\sum_{j\geq t+1}\E_{\Q^*}\left[\ind{\overline{V}(w_{[1,j]})<0,\ln \tilde{\Delta}^e_{j}\geq -\alpha V(w_{j})}\sum_{k\geq0}\E_{V(w_j)}\left[e^{S_k}\ind{\MS_{[1,k]}<0, S_k\in I(x)}\right]\right]
\end{align*}
Again by \eqref{RWeSbd}, one gets that
\begin{align*}
& \sum_{k\geq t+1}e^{\kappa x}\E_{\Q^*}\left[e^{V(w_k)}\ind{V(w_k)<\underline{V}(w_{[0,k-1]}), V(w_k)\in I(x)} \sum_{j=1}^{k-t}\ind{\ln \tilde{\Delta}^e_j\geq \alpha(-V(w_k)+V(w_{j-1}))}\right]\\
\leq & \sum_{j\geq t+1}\E_{\Q^*}\left[\ind{\overline{V}(w_{[1,j]})<0,\ln \tilde{\Delta}^e_{j}\geq -\alpha V(w_{j})} (1-V(w_j))e^{-\kappa' V(w_j)}\right]
%=&\sum_{j\geq t+1}\E\left[ (\tilde{\Delta}^e)^{\kappa'/\alpha}(1+\ln \tilde{\Delta}^e)\ind{\MS_{[1,j-1]}<0, -\frac{1}{\alpha}\ln\tilde{\Delta}^e-\zeta\leq S_{j-1}}\right]\\
%= &o_t(1)
\end{align*}
Suppose that under $\P$, $(\zeta, \tilde{\Delta}^e)$ is independent of the random walk $(S_j)$ and is distributed as $(\Delta V(w_j), \tilde{\Delta}_j^e)$ under $\Q^*$. Similarly as \eqref{intermediateupbd}, one sees that
\begin{align*}
\sum_{j\geq 1}\E_{\Q^*}\left[\ind{\overline{V}(w_{[1,j]})<0,\ln \tilde{\Delta}^e_{j}\geq -\alpha V(w_{j})} (1-V(w_j))e^{-\kappa' V(w_j)}\right]
\leq &c_{15}\E\left[e^{-\kappa'\zeta}(1+(\ln\tilde{\Delta}^e)^2+\zeta^2)\right].
\end{align*}
Here $\E\left[e^{-\kappa'\zeta}(1+\frac{1}{\alpha^2}(\ln \tilde{\Delta}^e)^2+\zeta^2)\right]<\infty$
by Assumption \ref{cond4}. This means that 
\[
\sum_{k\geq t+1}e^{\kappa x}\E_{\Q^*}\left[e^{V(w_k)}\ind{V(w_k)<\underline{V}(w_{[0,k-1]}), V(w_k)\in I(x)} \sum_{j=1}^{k-t}\ind{\ln \tilde{\Delta}^e_j\geq \alpha(-V(w_k)+V(w_{j-1}))}\right]=o_t(1).
\]
This suffices to conclude Lemma \ref{BRWrest}.
\end{proof}
\begin{proof}[Proof of Lemma \ref{BRWtight}]
Observe that for any $M\geq1$ and $x\geq0$,
\begin{align*}
\P[\mathcal{W}^{\tM}\geq M \vert \tM\leq -x]\leq & c_{16} e^{\kappa x}\P\left(e^{\tM}W_\infty\geq M, \tM\leq -x\right)\\
\leq & c_{16} e^{\kappa x}\P(W_\infty \geq Me^x)
\end{align*}
Recall that $\P(W_\infty\geq r)\sim C_0 r^{-\kappa}$. It is immediate that $\sup_{x\geq 0}\P[\mathcal{W}^{\tM}\geq M \vert \tM\leq -x]=o_M(1)$.
\end{proof}

\begin{proof}[Proof of Lemma \ref{BRWtailM}]
For small $\varepsilon\in(0,1)$, let us consider $P_\varepsilon(r):=\P(W_\infty\leq \varepsilon e^r, \tM\leq -r)$ for $r\gg1$ and prove that
\[
\limsup_{\epsilon\downarrow0}\limsup_{r\rightarrow\infty}e^{\kappa r}P_\varepsilon(r)=0.
\]
Observe that by change of measure (Proposition~\ref{BRWchangeofp}),
\begin{align}\label{Er1+2}
&e^{\kappa r}P_\varepsilon(r)\leq \sum_{n\geq1}e^{\kappa r}\E\left[\sum_{|u|=n}\ind{V(u)=\tM<\tM_{n-1}, V(u)\leq -r}\ind{W_\infty \leq \varepsilon e^{r}}\right]\\
\leq & \sum_{n=1}^b e^{\kappa r} \E_{\Q^*}\left[e^{V(w_n)}\ind{V(w_n)\leq -r, V(w_n)<\underline{V}(w_{[1,n-1]})}\right]\nonumber\\
+&\sum_{n\geq b+1}e^{\kappa r}\E_{\Q^*}\left[e^{V(w_n)}\ind{V(w_n)\leq -r, V(w_n)<\underline{V}(w_{[1,n-1]})}\ind{\sum_{j=1}^n \sum_{z\in\Omega(w_j)}e^{-V(z)}W_\infty^{(z)}\leq \varepsilon e^{r}}\right]\nonumber\\
=:&E_1(r)+E_2(r).\nonumber
\end{align}
It is proved in \eqref{smallk} that $E_1(r)=o_r(1)$ for any fixed $b\geq1$. Note that $q:=\P(W_\infty>0)>0$, and that $\P(\min_{|u|=1}V(u)\leq K)>0$ for any $K\in\R$ sufficiently large. Consequently,
\begin{align}\label{ErE+S}
&E_2(r)\leq   \sum_{n\geq b+1}e^{\kappa r}\E_{\Q^*}\left[e^{V(w_n)}\ind{V(w_n)\leq -r, V(w_n)<\underline{V}(w_{[1,n-1]})}\prod_{j=n-b}^n\ind{\nexists z\in\Omega(w_j): \Delta V(z)\leq K\textrm{ and } W_\infty^{(z)}>0}\right]\\
&+\sum_{n\geq b+1}e^{\kappa r}\E_{\Q^*}\left[e^{V(w_n)}\ind{V(w_n)\leq -r, V(w_n)<\underline{V}(w_{[1,n-1]})}\ind{W_\infty^+\min_{n-b\leq j\leq n}e^{V(w_n)-V(w_{j-1})-K}\leq \varepsilon }\right]\nonumber\\
&=: E_2(r, extinction)+E_2(r, survival)\nonumber
\end{align}
where $W_\infty^+$ is distributed as $\P(W_\infty\in\cdot\vert W_\infty>0)$ and is independent of $\nf^w_n$. First by time reversing then by the Markov property, one obtains that
\begin{align*}
&E_2(r, extinction)\\
=&\sum_{i\geq 0} \sum_{n\geq b+1}e^{-\kappa i} \E_{\Q^*}\left[e^{V(w_n)+\kappa (r+i)}\ind{V(w_n)\in I(r+i), \overline{V}(w_{[1,n]})<0}\prod_{j=1}^{b+1}\ind{\nexists z\in\Omega(w_j): \Delta V(z)\leq K\textrm{ and } W_\infty^{(z)}>0}\right]\\
=& \sum_{i\geq0}\E_{\Q^*}\left[\prod_{j=1}^{b+1}\prod_{z\in\Omega(w_j)}\big(1-\ind{\Delta V(z)\leq K}q\big)\ind{\MV{w_{[0,b+1]}}<0}\right.\\
&\hspace{2cm}\times \left.\sum_{n\geq b+1}\E_{V(w_{b+1})}[e^{S_{n-b-1}+\kappa(r+i)}\ind{\MS_{[1,n-b-1]}<0, S_{n-b-1}\in I(r+i)}]\right]
%\leq & \E_\Q\left[q^{\sum_{j=1}^{b+1}\#\Omega(w_j)}(1-S_{b+1})e^{-\kappa' S_{b+1}}\ind{\MS_{[1,b+1]}<0}\right]\\
%\leq  & \E\left[ q^{T_{b+1}}(1-S_{b+1}^{(\kappa)})\ind{S_{b+1}^{(\kappa)<0}}\right]\\
%\leq & C\E\left[q^{2T_{b+1}}\right]^{1/2}\E[1+(S_{b+1}^{(\kappa)})^2]^{1/2}\leq C(1+b) \E[q^{2T_1}]^{\frac{b+1}{2}}=o_b(1).
\end{align*}
which by \eqref{RWeSbd} is bounded by 
\begin{align*}
&\E_{\Q^*}\left[e^{-q\sum_{j=1}^{b+1}\sum_{z\in\Omega(w_j)}\ind{\Delta V(z)\leq K}}(1-V(w_{b+1}))e^{-\kappa' V(w_{b+1})}\ind{\MV(w_{[1,b+1]})<0}\right]\\
&\leq \E_{\Q^*}\left[e^{-2q\sum_{j=1}^{b+1}\sum_{z\in\Omega(w_j)}\ind{\Delta V(z)\leq K}}e^{-\kappa' V(w_{b+1})}\right]^{1/2}\E_{\Q^*}\left[(1-V(w_{b+1}))^2e^{-\kappa' V(w_{b+1})}\ind{\MV(w_{[1,b+1]})<0}\right]^{1/2}\\
&= \E_{\Q^*}\left[e^{-2q\sum_{z\in\Omega(w_1)}\ind{\Delta V(z)\leq K}}e^{-\kappa' V(w_{1})}\right]^{(b+1)/2}\E_{\Q^*}\left[(1-V(w_{b+1}))^2e^{-\kappa' V(w_{b+1})}\ind{\MV(w_{[1,b+1]})<0}\right]^{1/2}\\
&=\E\left[\sum_{|u|=1}e^{-\kappa V(u)}e^{-2q\sum_{|z|=1,z\neq u}\ind{\Delta V(z)\leq K}}\right]^{(b+1)/2}\E\left[(1-S_{b+1}^{(\kappa)})^2\ind{S_{b+1}^{(\kappa)}<0}\right]\\
&\leq  e^{-c_{17}(b+1)}\E\left[(1-S_{b+1}^{(\kappa)})^2\ind{S_{b+1}^{(\kappa)}<0}\right]\leq c_{18}(1+b)^2e^{-c_{17} (b+1)}.
%\sqrt{\E\left[\sum_{|u|=b+1}e^{-\kappa V(u)}(1-V(u))^2\ind{\max_{\rho<z\leq u} V(z)<0}\right]\E\left[\sum_{|u|=b+1}e^{-\kappa V(u)}q^{2\sum_{\rho<z\leq u}\#\Omega(z)}\right]}
%\leq  & \E\left[ q^{T_{b+1}}(1-S_{b+1}^{(\kappa)})\ind{S_{b+1}^{(\kappa)}<0}\right]\\
%\leq & C\E\left[q^{2T_{b+1}}\right]^{1/2}\E[1+(S_{b+1}^{(\kappa)})^2]^{1/2}\leq C(1+b) \E[q^{2T_1}]^{\frac{b+1}{2}}=o_b(1).
\end{align*}
since $\E[(S_1^{(\kappa)})^2]<\infty$. The constant $c$ is positive because $\E\left[\sum_{|u|=1}e^{-\kappa V(u)}e^{-2q\sum_{|z|=1,z\neq u}\ind{\Delta V(z)\leq K}}\right]<\E\left[\sum_{|u|=1}e^{-\kappa V(u)}\right]=1$ Therefore, 
\begin{equation}\label{ErE}
E_2(r, extinction)=o_b(1).
\end{equation}
It remains to bound $E_{2}(r, survival)$. For convenience, under $\P$, let $W_\infty^+$ be still distributed as $\P(W_\infty\in\cdot\vert W_\infty>0)$ and is independent of the random walk $(S_n)_{n\geq0}$. Then by time reversing, one has
\begin{align*}
E_2(r, survival)=& \sum_{n\geq b+1}e^{\kappa r}\E\left[e^{S_n}\ind{S_n\leq -r, S_n<\mS_{[1,n-1]}}\ind{ W_\infty^+\min_{n-b\leq j\leq n}e^{S_n-S_{j-1}}\leq \varepsilon e^{K}}\right]\\
=&\sum_{j\geq 0} \sum_{n\geq b+1}e^{-\kappa j} \E\left[e^{S_n+\kappa (r+j)}\ind{\MS_{[1,n]}<0, S_n\in I(r+j)}\ind{W_\infty^+ e^{\mS_{[1,b+1]}}\leq \varepsilon e^{K}}\right]\\
=&\sum_{j\geq 0}e^{-\kappa j}\E\left[\ind{W_\infty^+ e^{\mS_{[1,b+1]}}\leq \varepsilon e^{K}}\ind{\MS_{[1,b+1]}<0}\E_{S_{b+1}}\left[\sum_{n\geq0}e^{S_n+\kappa (r+j)}\ind{\MS_{[1,n]}<0, S_n\in I(r+j)}\right]\right]
%\leq &\E\left[(1-S_{b+1})e^{-\kappa' S_{b+1}}\ind{W_\infty^+ e^{\mS_{[1,b+1]}}\leq \varepsilon e^{K}}\right]\\
%\leq & \E\left[(1-S_{b+1})e^{-\kappa' S_{b+1}}\ind{W_\infty^+ \leq \varepsilon e^{2K}}\right]+\E\left[(1-S_{b+1})e^{-\kappa' S_{b+1}}\ind{\mS_{[1,b+1]}\leq -K}\right]\\
%= & o_\varepsilon(1) C_b +o_K(1) C_b
\end{align*}
by the Markov property at time $b+1$. It follows from \eqref{RWeSbd} that
\begin{align*}
E_2(r, survival)\leq &c\E\left[(1-S_{b+1})e^{-\kappa' S_{b+1}}\ind{\MS_{[1,b+1]}<0}\ind{W_\infty^+ e^{\mS_{[1,b+1]}}\leq \varepsilon e^{K}}\right]\\
\leq & \E\left[(1+|S_{b+1}|)e^{-\kappa' S_{b+1}}\ind{W_\infty^+ \leq \varepsilon e^{2K}}\right]+\E\left[(1+|S_{b+1}|)e^{-\kappa' S_{b+1}}\ind{\mS_{[1,b+1]}\leq -K}\right]\\
= & \E\left[(1+|S^{(\kappa)}_{b+1}|)\right]\P(W_\infty^+\leq \epsilon e^{2K})+\E\left[(1+|S_{b+1}^{(\kappa)}|)\ind{\mS_{[1,b+1]}\leq -K}\right],
\end{align*}
where $\E\left[(1+|S^{(\kappa)}_{b+1}|)\right]\leq c_{19}(1+b)$. By Cauchy-Schwartz inequality,
\[
\E\left[(1+|S_{b+1}^{(\kappa)}|)\ind{\mS_{[1,b+1]}\leq -K}\right]\leq \sqrt{\E\left[(1+|S_{b+1}^{(\kappa)}|)^2\right]}\sqrt{\P\left(\mS^{(\kappa)}_{[1,b+1]}\leq -K\right)},
\]
with $\sqrt{\E\left[(1+|S_{b+1}^{(\kappa)}|)^2\right]}\leq c_{20}(1+b)$. We hence deduce that
\[
E_2(r, survival)\leq c_{21}(1+b) [o_\epsilon(1)+\sqrt{\P\left(\mS^{(\kappa)}_{[1,b+1]}\leq -K\right)}].
\]
Plugging it and \eqref{ErE} into \eqref{ErE+S} implies that
\[
E_2(r)=o_b(1)+c_{21}(1+b) [o_\epsilon(1)+\sqrt{\P\left(\mS^{(\kappa)}_{[1,b+1]}\leq -K\right)}].
\]
It then follows from \eqref{Er1+2} that 
\[
\limsup_{\varepsilon\downarrow0}\limsup_{r\rightarrow\infty}e^{\kappa r}P_\varepsilon(r)\leq o_b(1)+c_{21}(1+b) \sqrt{\P\left(\mS^{(\kappa)}_{[1,b+1]}\leq -K\right)}.
\]
Letting $K\rightarrow\infty$ then $b\rightarrow\infty$ yields what we need.

\end{proof}

\subsection{Proofs of the technical lemmas in Section~\ref{tailRWRE}}
\label{lemRWRE}

Recall that under $\widehat{\p}^*$, the spine $\beta(w_k), k\geq0$ is a Markov chain on $\mathbb{N}^*$, with transition probabilities given by
\[
p_{i,j}={i+j-1 \choose i}\E\left[\sum_{|u|=1}\frac{e^{-j V(u)}}{(1+e^{-V(u)})^{i+j}}\right],  \forall i, j\geq1.
\]
For the function $f(i):=\frac{\Gamma(i+\gamma)}{i}$ with $i\geq 1$ and $\gamma\in(0,\kappa-1)$, it is known (in Appendix of \cite{CdRH16}) that there exists $d\in(0,1)$ such that for any $i\geq i_0$ large enough,
\begin{equation}\label{Domsum}
\sum_{j=1}^\infty p_{i, j}f(j)\leq d f(i).
\end{equation}
As a consequence of Theorem 15.3.3 of \cite{MT93}, there exists $C_f>0$ such that for any $r\in(1,\frac{1}{d})$ and for any $i\geq1$, 
\[
\widehat{\e}_i^*\left[\sum_{k=1}^{\widehat{\tau}_1}f(\beta(w_k))r^k\right]\leq C_f f(i),
\]
where $\widehat{\tau}_1=\inf\{k\geq1: \beta(w_k)=1\}<\infty$.
As we can take $d$ sufficiently close to $\psi(1+\gamma)$ and $f(i)\sim i^\gamma$ for $i\rightarrow\infty$, we conclude that for any $\gamma\in(0,\kappa-1)$, there exists $\cb_\gamma>0$ such that for any $r\in(1, \frac{1}{\psi(1+\gamma)})$ and $i\geq1$,
\begin{equation}\label{eq:dompath}
\widehat{\e}_i^*\left[\sum_{k=1}^{\widehat{\tau}_1}(\beta(w_k))^\gamma r^k\right]\leq \cb_\gamma i^\gamma.
\end{equation}
Further, Jensen's inequality implies that for any $p>0$, for any $i\geq1$,
\[
\widehat{\e}_i^*[\widehat{\tau}_1^p]\leq c_{22} \log^p(1+i).
\]
Let us turn to prove the technical facts in Section~\ref{tailRWRE}.

\textbf{Proof of Lemma \ref{Mom}.}
In fact, \eqref{MomL} has been proven in \cite{dR16} (see Lemmas 11 and 14) when $\kappa\leq 2$. Let us prove \eqref{MomL} for $\kappa>2$ using the idea borrowed from the proof of Lemma 5 in \cite{AHZ13}.

\begin{proof}[Proof of \eqref{MomL}]
We treat  the $\e_i[L_1^{1+\alpha}]$ for integer $\alpha$ and non-integer $\alpha$ separately. First for $m\in\N^*$ satisfying $m<\kappa$, let us show by recurrence that there exists some constant $C_m>0$ such that $\e_i[L_1^m]\leq C_m i^m$ for any $i\geq1$. Recall that for $m=1$, $\e_i[L_1]=i$. Suppose that $\forall 1\leq k\leq m$ and $\forall i\geq1$,
\[
\e_i[L_1^{k}]\leq C_k i^k.
\]
Let us bound $\E_i[L_1^{1+m}]$ for $1+m<\kappa$. Note that by change of measures,
\begin{align*}
\e_i[L_1^{1+m}]=& i\widehat{\e}^*_i[L_1^m]\\
=&i \widehat{\e}^*_i\left[\left(1+\sum_{k=1}^{\tauh_1}\sum_{u\in\Omega(w_k)}L_1^{(u)}\right)^m\right]
\end{align*}
where $\tauh_1=\inf\{k\geq1: \beta(w_k)=1\}$ and $L_1^{(u)}=\sum_{v: v\geq u}\ind{v\in\L_1}$ for any $u\in \B_1$. Set $\Sigma_k:=\sum_{u\in\Omega(w_k)}L_1^{(u)}$ for $k\geq1$ and $\Sigma_0:=1$. And write
\[
\chi_l:=\sum_{k=l}^{\tauh_1}\Sigma_k, \forall 0\leq l\leq \tauh_1,
\]
with $\chi_{\tauh_1+1}:=0$. Apparently, $L_1=\chi_0$ under $\widehat{\p}_i^*$. As $\chi_0^m=\sum_{l=0}^{\tauh_1}(\chi_l^m-\chi_{l+1}^m)$, it follows that
\begin{align}
\e_i[L_1^{1+m}]=& i\widehat{\e}^*_i\left[\sum_{l=0}^{\tauh_1}(\chi_l^m-\chi_{l+1}^m)\right]= i\widehat{\e}^*_i\left[\sum_{l=0}^{\tauh_1}((\chi_{l+1}+\Sigma_l)^m-\chi_{l+1}^m)\right]\nonumber\\
=&i\widehat{\e}^*_i\left[\sum_{l=0}^{\tauh_1}\sum_{k=0}^{m-1}{m \choose k}\chi_{l+1}^k\Sigma_l^{m-k}\right]\leq  c_{23} i \sum_{k=0}^{m-1}\widehat{\e}^*_i\left[\sum_{l=0}^{\tauh_1} \chi_{l+1}^k\Sigma_l^{m-k}\right].\label{byrec}
\end{align}
Here by the Markov property at time $l$, one sees that for any $0\leq k\leq m-1$ and $l\geq1$,
\begin{align*}
\widehat{\e}^*_i\left[\ind{l\leq \tauh_1}\chi_{l+1}^k\Sigma_l^{m-k}\right]=&\widehat{\e}^*_i\left[\ind{l-1< \tauh_1}\Sigma_l^{m-k}\widehat{\e}^*_{\beta(w_l)}[\chi_{1}^k]\right]\\
\leq & \widehat{\e}^*_i\left[\ind{l-1< \tauh_1}\Sigma_l^{m-k}\widehat{\e}^*_{\beta(w_l)}[L_1^k]\right]\\
\leq & C_k \widehat{\e}^*_i\left[\ind{l-1< \tauh_1}\Sigma_l^{m-k} \beta(w_l)^k\right],
\end{align*}
where $\widehat{\e}^*_i[L_1^k]=\frac{1}{i}\e_i[L_1^k]\leq C_k i^{k-1}$ for $k\geq1$ and $\widehat{\e}^*_i[L_1^k]=1$ for $k=0$. Again, by the Markov property at time $l-1$, for any $l\geq1$, one has
\begin{equation}\label{oneterm}
\widehat{\e}^*_i\left[\ind{l\leq \tauh_1}\chi_{l+1}^k\Sigma_l^{m-k}\right]\leq C_k\widehat{\e}_i^*\left[\ind{l-1< \tauh_1}\widehat{\e}^*_{\beta(w_{l-1})}\left(\Sigma_1^{m-k}\beta(w_1)^k\right)\right]
\end{equation}
Now for any $\ell\geq1$, let us estimate $\widehat{\e}_\ell^*\left[\Sigma_1^{m-k}\beta(w_1)^k\right]$. Observe that given $(\beta(w_1); \beta(u), u\in\Omega(w_1))$, $(L_1^{(u)}, u\in\Omega(w_1))$ are independent and $L_1^{(u)}$ is distributed as $L_1$ under $\p_{\beta(u)}$. For any $0\leq k\leq m-1$ (i.e., $1\leq m-k\leq m$), by convexity, 
\[
\Sigma_1^{m-k}=\left(\sum_{u\in\Omega(w_1)}\beta(u)\frac{L_1^{(u)}}{\beta(u)}\right)^{m-k}\leq \left(\sum_{u\in\Omega(w_1)}\beta(u)\right)^{m-1-k} \sum_{u\in\Omega(w_1)}\beta(u)\left(\frac{L_1^{(u)}}{\beta(u)}\right)^{m-k}
\]
As a consequence, 
\begin{align*}
\widehat{\e}_\ell^*\left[\Sigma_1^{m-k}\beta(w_1)^k\right]=&\widehat{\e}_\ell^*\left[\beta(w_1)^k\widehat{\e}_\ell^*\left[\Sigma_1^{m-k}\Big\vert (w_1, \beta(w_1)); (u, \beta(u))_{ u\in\Omega(w_i)}\right]\right]\\
\leq & \widehat{\e}_\ell^*\left[\beta(w_1)^k \left(\sum_{u\in\Omega(w_1)}\beta(u)\right)^{m-1-k} \sum_{u\in\Omega(w_1)}\beta(u) \e_{\beta(u)}\left[\left(\frac{L_1}{\beta(u)}\right)^{m-k}\right]\right],
\end{align*}
where $\e_i[L_1^{m-k}]\leq C_{m-k} i^{m-k}$ for any $i\geq1$ as supposed above. Therefore, by Lemma \ref{momentofw1},  for any $\ell\geq1$ and $0\leq k\leq m-1$,
\begin{align*}
&\widehat{\e}_\ell^*\left[\Sigma_1^{m-k}\beta(w_1)^k\right]\leq  C_{m-k} \widehat{\e}_\ell^*\left[\beta(w_1)^k \left(\sum_{u\in\Omega(w_1)}\beta(u)\right)^{m-k}\right]\leq c_{25} \ell^m.
%=& \frac{C_{m-k}}{\ell}\e_\ell\left[\sum_{|z|=1}\beta(z)^{k+1}\left(\sum_{|u|=1, u\neq z}\beta(u)\right)^{m-k}\right]\leq \frac{C_{m-k}}{\ell}\e_\ell\left[\left(\sum_{|z|=1}\beta(z)\right)^{m+1}\right].
\end{align*}
Applying it in \eqref{oneterm} implies that for any $i\geq 1$ and $l\geq1$,
\[
\widehat{\e}^*_i\left[\ind{l\leq \tauh_1}\chi_{l+1}^k\Sigma_l^{m-k}\right]\leq c_{26} \widehat{\e}_i^*\left[\ind{l-1< \tauh_1} \beta(w_{l-1})^m\right]
\]
where for $l=0$, $\widehat{\e}^*_i\left[\ind{l\leq \tauh_1}\chi_{l+1}^k\Sigma_l^{m-k}\right]=\widehat{\e}^*_i\left[\chi_{1}^k\right]\leq C_k i^{k}$ by hypothesis. Plugging this inequality into \eqref{byrec} yields that
\begin{align*}
\e_i[L_1^{1+m}]\leq &  c_{27} i \sum_{k=0}^{m-1}\widehat{\e}^*_i\left[\sum_{l=0}^{\infty}\ind{l\leq \tauh_1} \chi_{l+1}^k\Sigma_l^{m-k}\right]\\
\leq & c_{27} i \sum_{k=0}^{m-1} C_k i^k+ c_{27}i \sum_{k=0}^{m-1} \widehat{\e}^*_i\left[\sum_{l=1}^{\infty}\ind{l-1<\tauh_1}\beta(w_{l-1})^m\right]\\
\leq & c_{28} i^{m+1} + c_{27} im \widehat{\e}^*_i\left[\sum_{l=0}^{\tauh_1-1}\beta(w_l)^m\right].
\end{align*}
%We claim that under $\widehat{\p}^*$, $(\beta(w_k))_{k\ge0}$ is a recurrent Markov chain such that for any $\alpha\in(0,\kappa-1)$ and $r\in (1, \frac{1}{\psi(1+\alpha)})$,
%\begin{equation}\label{spinekey}
%\widehat{\e}_i^*\left[\sum_{k=1}^{\sigma^+_1}\beta(w_k)^\alpha r^k\right]\leq C_\alpha i^\alpha, \forall i\geq1.
%\end{equation}
%When $\kappa\in (1,2]$, this is (3.7) in \cite{dR16}. When $\kappa>2$, (4.4) in \cite{CdRH16} suffices to conclude \eqref{spinekey} by Theorem 15.3.3 of \cite{MT93}. 
By \eqref{eq:dompath}, we hence end up with
\[
\e_i[L_1^{1+m}]\leq C_{m+1} i^{m+1},
\]
as long as $m+1<\kappa$. We deduce by recurrence that for any integer $1\leq m\leq \lceil \kappa \rceil -1 $, 
\[
\e_i[L_1^m]\leq C_m i^m, \forall i\geq1.
\]
It remains to consider $\e_i[L_1^{m+\delta}]$ for $2\leq m\leq \lceil \kappa \rceil -1$ and $\delta\in(0,(\kappa-m)\wedge 1)$. Write $m-1+\delta=m(1-\eta)$ for some $\eta\in(0,1)$. Observe that by \eqref{byrec},
\begin{align}\label{byrec+}
\e_i[L_1^{m+\delta}]=& i\widehat{\e}_i^*\left[L_1^{m-1+\delta}\right]=i\widehat{\e}_i^*\left[(L_1^{m})^{1-\eta}\right]=i\widehat{\e}^*_i\left[\left(\sum_{l=0}^{\tauh_1}\sum_{k=0}^{m-1}{m \choose k}\chi_{l+1}^k\Sigma_l^{m-k}\right)^{1-\eta}\right]\nonumber\\
\leq & c_{29}\sum_{k=0}^{m-1} i\widehat{\e}_i^*\left[\sum_{l=0}^{\tauh_1}\chi_{l+1}^{k(1-\eta)}\Sigma_l^{(m-k)(1-\eta)}\right].
\end{align}
It is proven for integer $0\leq k\leq m-1$ that for any $i\geq1$,
\[
\widehat{\e}^*_{i}[\chi_1^k]\leq \widehat{\e}^*_{i}[L_1^k]\leq C_k i^k.
\]
When $l=0$, for any $0\leq k\leq m-1$,
\[
\widehat{\e}^*_i\left[\ind{l\leq \tauh_1}\chi_{l+1}^{k(1-\eta)}\Sigma_l^{(m-k)(1-\eta)}\right]\leq \widehat{\e}^*_i\left[\chi_{1}^{k(1-\eta)}\right]\leq\widehat{\e}^*_i\left[\chi_{1}^{k}\right] ^{(1-\eta)}\leq c_{30} i^{k(1-\eta)}.
\]
For $l\geq1$, by Markov property at time $l$ then by Jensen's inequality, one sees that for any $0\leq k\leq m-1$,
\begin{align*}
\widehat{\e}^*_i\left[\ind{l\leq \tauh_1}\chi_{l+1}^{k(1-\eta)}\Sigma_l^{(m-k)(1-\eta)}\right]=&\widehat{\e}^*_i\left[\ind{l-1 < \tauh_1}\Sigma_l^{(m-k)(1-\eta)}\widehat{\e}^*_{\beta(w_l)}[\chi_{1}^{k(1-\eta)}]\right]\\
\leq & \widehat{\e}^*_i\left[\ind{l-1< \tauh_1}\Sigma_l^{(m-k)(1-\eta)}\widehat{\e}^*_{\beta(w_l)}[\chi_1^k]^{1-\eta}\right]\\
\leq & c_{30}\widehat{\e}^*_i\left[\ind{l-1< \tauh_1}\Sigma_l^{(m-k)(1-\eta)} (\beta(w_l))^{k(1-\eta)}\right].
%\leq & C_k \widehat{\e}^*_i\left[\ind{l< \tauh_1}\Sigma_l^{m-k} \beta(w_l)^k\right],
\end{align*}
As a consequence,
\begin{align*}
\e_i[L_1^{m+\delta}]\leq &c_{31}i^{m+\delta}+c_{31}\sum_{k=0}^{m-1} i\widehat{\e}^*_i\left[\sum_{l\geq1}\ind{l-1< \tauh_1}(\beta(w_l))^{k(1-\eta)}\Sigma_l^{(m-k)(1-\eta)}\right]
\end{align*}
For the integer $k$ such that $(m-k)(1-\eta)\geq1$, by convexity, the similar arguments as above show that
\[
\widehat{\e}^*_i\left[\ind{l-1< \tauh_1}(\beta(w_l))^{k(1-\eta)}\Sigma_l^{(m-k)(1-\eta)}\right]\leq c_{32} \widehat{\e}_i^*\left[\ind{l-1< \tauh_1} \beta(w_{l-1})^{m(1-\eta)}\right]
\]
For the integer $k$ such that $(m-k)(1-\eta)\leq 1$, by Markov property at time $l-1$,
\[
\widehat{\e}^*_i\left[\ind{l< \tauh_1}(\beta(w_l))^{k(1-\eta)}\Sigma_l^{(m-k)(1-\eta)}\right]
\leq C_k\widehat{\e}_i^*\left[\ind{l-1< \tauh_1}\widehat{\e}^*_{\beta(w_{l-1})}\left(\Sigma_1^{(m-k)(1-\eta)}\beta(w_1)^{k(1-\eta)}\right)\right]
\]
Here applying Jensen's inequality to $\widehat{\e}_\ell^*\left[\Sigma_1^{(m-k)(1-\eta)}\beta(w_1)^{k(1-\eta)}\right]$ yields that
\begin{align*}
\widehat{\e}_\ell^*\left[\Sigma_1^{(m-k)(1-\eta)}\beta(w_1)^{k(1-\eta)}\right]=&\widehat{\e}_\ell^*\left[\beta(w_1)^{k(1-\eta)}\widehat{\e}_\ell^*\left[\Sigma_1^{(m-k)(1-\eta)}\Big\vert (w_1, \beta(w_1)); (u, \beta(u))_{ u\in\Omega(w_i)}\right]\right]\\
\leq &\widehat{\e}_\ell^*\left[\beta(w_1)^{k(1-\eta)}\widehat{\e}_\ell^*\left[\Sigma_1\Big\vert (w_1, \beta(w_1)); (u, \beta(u))_{ u\in\Omega(w_i)}\right]^{(m-k)(1-\eta)}\right]\\
\leq &\widehat{\e}_\ell^*\left[\beta(w_1)^{k(1-\eta)}\left(\sum_{u\in\Omega(w_1)}\beta(u)\right)^{(m-k)(1-\eta)}\right]\leq c_{33} \ell^{m(1-\eta)}.
\end{align*}
Still, we see that for the integer $k$ such that $(m-k)(1-\eta)\leq 1$,
\[
\widehat{\e}^*_i\left[\ind{l< \tauh_1}(\beta(w_l))^{k(1-\eta)}\Sigma_l^{(m-k)(1-\eta)}\right]\leq c_{34} \widehat{\e}_i^*\left[\ind{l-1< \tauh_1} \beta(w_{l-1})^{m(1-\eta)}\right].
\]
It follows that
\begin{align*}
\e_i[L_1^{m+\delta}]\leq &c_{31}i^{m+\delta}+c_{31}c_{34}\sum_{k=0}^{m-1} i\widehat{\e}^*_i\left[\sum_{l=0}^{\tauh_1}(\beta(w_l))^{m(1-\eta)}\right]\\
=& c_{31}i^{m+\delta}+c_{31}c_{34}\sum_{k=0}^{m-1} i\widehat{\e}^*_i\left[\sum_{l=0}^{\tauh_1}(\beta(w_l))^{m+\delta-1}\right]\\
\leq & c_{35} i^{m+\delta}
\end{align*}
since $m+\delta<\kappa$.
\end{proof}
\begin{proof}[Proof of \eqref{MomM}]
Note that for any $i\in\N$ and $r> i$,
\[
\p_i(M_1\geq r)=\p_i\left(\sum_{u\in\mathcal{B}_1}\ind{\beta(u)\geq r}\geq 1\right)
\]
which by Markov inequality yields that
\begin{align*}
\p_i(M_1\geq r)\leq &\e_i\left[\sum_{u\in\mathcal{B}_1}\ind{\beta(u)\geq r}\right]\\
=&\sum_{n=1}^\infty\e_i\left[\sum_{|u|=n}\ind{\beta(u)\geq r>\max_{\rho<v<u}\beta(v), \min_{\rho<v<u}\beta(v)\geq2}\right].
\end{align*}
By change of measures and Proposition \ref{prop:changeofmeasurep}, one sees that
\begin{align*}
\p_i(M_1\geq r)\leq &\sum_{n\geq1}\widehat{\e}^*_i\left[\frac{i}{\beta(w_n)}; \beta(w_n)\geq r>\max_{1\leq k\leq n-1}\beta(w_k), \min_{1\leq k\leq n-1}\beta(w_n)\geq2\right]\\
\leq&\frac{i}{r}\widehat{\p}^*_i\left(\max_{0\leq k<\tauh_1}\beta(w_k)\geq r\right),
\end{align*}
where $\tauh_1=\inf\{k\geq1: \beta(w_k)=1\}$. According to Lemma 4.2 of \cite{CdRH16}, for any $\alpha\in(0,\kappa-1)$, there exists some constant $c_\alpha\in(1,\infty)$,
\[
\widehat{\p}^*_i\left(\max_{0\leq k<\tauh_1}\beta(w_k)\geq r\right)\leq c_\alpha \left(\frac{i}{r}\right)^\alpha.
\]
(Notice that in fact, the proof of this inequality holds also for $\kappa=\infty$). As a result, for any $r\geq i$,
\[
\p_i(M_1\geq r)\leq c_\alpha\left(\frac{i}{r}\right)^{\alpha+1}.
\]
Therefore, for any $\alpha\in[0,\kappa-1)$, we take $\eta\in(\alpha,\kappa-1)$ so that $\p_i(M_1\geq r)\leq c_\eta (i/r)^{\eta+1}$. Consequently,
\begin{align*}
\e_i[M_1^{1+\alpha} ]=&\int_0^\infty (1+\alpha) r^\alpha\p_i(M_1\geq r)dr\\
\leq& \int_0^i(1+\alpha)r^\alpha dr +\int_i^\infty (1+\alpha)r^\alpha c_\eta \left(\frac{i}{r}\right)^{\eta+1}dr\\
\leq & C_\alpha i^{1+\alpha}.
\end{align*}

\end{proof}

\textbf{Proof of Lemma \ref{bdness}.} \eqref{MombdA} is a special case of \eqref{MomL} and it holds in a similar way.

\begin{proof}[Proof of \eqref{MombdA}]
We prove \eqref{MombdA} in a similar way as in the proof of \eqref{MomL}. Observe that 
\[
\widehat{\e}_1^*\left[\left(\sum_{k=1}^{\widehat{\tau}_1}\ind{\beta(w_{k-1})<A}\sum_{u\in\Omega(w_k)}L_1^{(u)}\right)^{\kappa-1}\right]=\widehat{\e}_1^*\left[\left(\sum_{k=1}^{\widehat{\tau}_1}\ind{\beta(w_{k-1})<A}\Sigma_k\right)^{\kappa-1}\right].
\]
Its finiteness has been proven in \cite{dR16} for $\kappa\in(1,2]$. Assume now that $\kappa>2$. Let us prove its finiteness by recurrence. 

It is known from \eqref{MomL} that for any $0\leq \alpha<\kappa-1$ and any $i\geq1$,
\[
\widehat{\e}_i^*\left[\left(\sum_{k=1}^{\widehat{\tau}_1}\Sigma_k\right)^{\alpha}\right]<C_\alpha i^\alpha.
\]
If $\kappa>2$ is an integer, It follows from \eqref{byrec} that for $m=\kappa-2$,
\begin{align}
\widehat{\e}_1^*\left[\left(\sum_{k=1}^{\widehat{\tau}_1}\ind{\beta(w_{k-1})<A}\Sigma_k\right)^{m+1}\right]\leq &c_{36}\sum_{k=0}^{m-1}\widehat{\e}_1^*\left[\sum_{l=1}^{\widehat{\tau}_1}\ind{\beta(w_{l-1})<A}\Sigma_l^{m-k}\left(\sum_{j=l+1}^{\widehat{\tau}_1}\Sigma_j\ind{\beta(w_{j-1})<A}\right)^k\right]\nonumber\\
\leq  &c_{36}\sum_{k=0}^{m-1}\widehat{\e}_1^*\left[\sum_{l=1}^{\infty}\ind{l\leq\widehat{\tau}_1}\ind{\beta(w_{l-1})<A}\Sigma_l^{m-k}\chi_{l+1}^k\right]
\end{align}
where
\[
\widehat{\e}_1^*\left[\ind{l\leq\widehat{\tau}_1}\ind{\beta(w_{l-1})<A}\Sigma_l^{m-k}\chi_{l+1}^k\right]\leq c_{37}\widehat{\e}_1^*\left[\ind{l-1<\widehat{\tau}_1}\ind{\beta(w_{l-1})<A}\beta(w_{l-1})^m\right].
\]
Consequently,
\begin{align*}
\widehat{\e}_1^*\left[\left(\sum_{k=1}^{\widehat{\tau}_1}\ind{\beta(w_{k-1})<A}\Sigma_k\right)^{m+1}\right]\leq &c_{38}\sum_{k=0}^{m-1} \widehat{\e}_1^*\left[\sum_{l=1}^\infty\ind{l-1<\widehat{\tau}_1}\ind{\beta(w_{l-1})<A}\beta(w_{l-1})^m\right]\\
=& c_{38}m \widehat{\e}_1^*\left[\sum_{l=0}^{\widehat{\tau}_1-1}\ind{\beta(w_{l})<A}\beta(w_{l})^m\right]<\infty.
\end{align*}

If $\kappa>2$ is not an integer, for $m=\lfloor \kappa\rfloor$ and $\delta=\kappa-m\in(0,1)$, write $\kappa-1=m-1+\delta=m(1-\eta)$ for some $\eta\in(0,1)$. Similarly as \eqref{byrec+}, one sees that
\begin{align*}
\widehat{\e}_1^*\left[\left(\sum_{k=1}^{\widehat{\tau}_1}\ind{\beta(w_{k-1})<A}\Sigma_k\right)^{\kappa-1}\right]=&\widehat{\e}_1^*\left[\left(\sum_{k=1}^{\widehat{\tau}_1}\ind{\beta(w_{k-1})<A}\Sigma_k\right)^{m(1-\eta)}\right]\\
\leq & c_{39}\sum_{k=0}^{m-1} \widehat{\e}_1^*\left[\sum_{l=1}^{\tauh_1}\ind{\beta(w_{l-1})<A}\chi_{l+1}^{k(1-\eta)}\Sigma_l^{(m-k)(1-\eta)}\right]
\end{align*}
where 
\[
\widehat{\e}_1^*\left[\ind{l\leq\tauh_1}\ind{\beta(w_{l-1})<A}\chi_{l+1}^{k(1-\eta)}\Sigma_l^{(m-k)(1-\eta)}\right]\leq  c_{40} \widehat{\e}_1^*\left[\ind{l-1< \tauh_1}\ind{\beta(w_{l-1})<A} \beta(w_{l-1})^{m(1-\eta)}\right].
\]
Therefore,
\begin{align*}
\widehat{\e}_1^*\left[\left(\sum_{k=1}^{\widehat{\tau}_1}\ind{\beta(w_{k-1})<A}\Sigma_k\right)^{\kappa-1}\right]\leq &c_{41} m \widehat{\e}_1^*\left[\sum_{l=1}^{\widehat{\tau}_1}\ind{\beta(w_{l-1})<A}\beta(w_{l-1})^{m(1-\eta)}\right]\\
=&c_{41} m \widehat{\e}_1^*\left[\sum_{l=0}^{\widehat{\tau}_1-1}\ind{\beta(w_{l})<A}\beta(w_{l})^{\kappa-1}\right]\leq c_{42} m A^{\kappa-1}\widehat{\e}_1^*[\widehat{\tau}_1]<\infty.
\end{align*}
\end{proof}

\begin{proof}[Proof of \eqref{Mombdkappa}]
Let us prove the finiteness of $\widehat{\e}_1^*\left[\left(\beta(w_{\sigma_A})\right)^{\kappa-1}\ind{\sigma_A<\widehat{\tau}_1}\right]$. Note that
\[
\left(\beta(w_{\sigma_A})\right)^{\kappa-1}\ind{\sigma_A<\widehat{\tau}_1}\leq A^{\kappa-1}\sum_{k=1}^{\widehat{\tau}_1-1}\ind{\beta(w_{k-1})<A}\left(\frac{\beta(w_k)}{\beta(w_{k-1})}\right)^{\kappa-1}.
\]
It follows that 
\begin{align*}
\widehat{\e}_1^*\left[\left(\beta(w_{\sigma_A})\right)^{\kappa-1}\ind{\sigma_A<\widehat{\tau}_1}\right]\leq&A^{\kappa-1} \widehat{\e}_1^*\left[\sum_{k=1}^{\widehat{\tau}_1}\ind{\beta(w_{k-1})<A}\left(\frac{\beta(w_k)}{\beta(w_{k-1})}\right)^{\kappa-1}\right]\\
=& A^{\kappa-1} \widehat{\e}_1^*\left[\sum_{k=1}^\infty\ind{k-1<\widehat{\tau}_1}\frac{1}{\beta(w_{k-1})^{\kappa-1}}\widehat{\e}^*_{\beta(w_{k-1})}\left[\beta(w_1)^{\kappa-1}\right]\right]
\end{align*}
where by Lemma \ref{momentofw1},
\[
\frac{1}{\beta(w_{k-1})^{\kappa-1}}\widehat{\e}^*_{\beta(w_{k-1})}\left[\beta(w_1)^{\kappa-1}\right]\leq c_{43}<\infty.
\]
As $ \widehat{\e}_1^*[\widehat{\tau}_1]<\infty$, one ends up with
\[
\widehat{\e}_1^*\left[\left(\beta(w_{\sigma_A})\right)^{\kappa-1}\ind{\sigma_A<\widehat{\tau}_1}\right]\leq c_{43} A^{\kappa-1} \widehat{\e}_1^*[\widehat{\tau}_1]<\infty.
\]
\end{proof}

\begin{proof}[Proof of \eqref{NBlaw}]
Note that from the negative multinomial distribution, one has the generating function of $\sum_{i=1}^N z_i\zeta_i$ as follows
\[
\E[s^{\sum_{i=1}^N z_i\zeta_i}]=\left(\frac{1}{1+\sum_{i=1}^N A_i(1-s^{z_i})}\right)^n, \forall s\in[0,1].
\]
Apparently, $\sum_{i=1}^N z_i\zeta_i$ can be viewed as sum of $n$ i.i.d. random variables of mean $\sum_{i=1}^N A_i z_i$. According to \eqref{MomofSum}, it suffices to prove \eqref{NBlaw} for $n=1$. In fact we only need to show that 
\begin{equation}\label{NBMom}
\E\left[\left(\sum_{i=1}^N z_i \zeta_i \right)^\alpha\right]\leq 
%C_\alpha n^{\alpha/2\vee 1} \left[(\sum_{i=1}^N  A_i z_i)^\alpha+ \sum_{i=1}^N A_i z_i^{\alpha}\right] \textrm{ if } \alpha\in[1,2]\\
C_2(\alpha) \left[\sum_{k=0}^{\lfloor \alpha-1\rfloor}(\sum_{i=1}^NA_i z_i)^k(\sum_{i=1}^N A_i z_i^{\alpha-k})+(\sum_{i=1}^N  A_i z_i)^\alpha\right]. 
\end{equation}
Let us prove it by recurrence. Fix $n=1$, we now have
\[
\E\left[\prod_{i=1}^N s_i^{ z_i\zeta_i}\right]=\frac{1}{1+\sum_{i=1}^N A_i(1-s_i^{z_i})}, \E[\sum_{i=1}^N z_i\zeta_i]=\sum_{i=1}^N A_i z_i,
\]
and $Var(\sum_{i=1}^N z_i\zeta_i)=\sum_{i=1}^N A_i z_i^2+(\sum_{i=1}^N A_i z_i)^2$. So, \eqref{NBMom} holds for $\alpha=1$ and $\alpha=2$. For $\alpha\in(1,2)$, proving \eqref{NBMom} means proving 
\[
\E\left[\left(\sum_{i=1}^N z_i \zeta_i \right)^\alpha\right]\leq 
C_2(\alpha)  \left[(\sum_{i=1}^N  A_i z_i)^\alpha+ \sum_{i=1}^N A_i z_i^{\alpha}\right],
\]
Write $\alpha=1+\delta$ with some $\delta\in(0,1)$. Observe that
\begin{align}\label{NBMom1}
\E\left[\left(\sum_{i=1}^N z_i \zeta_i \right)^\alpha\right]=\E\left[\left(\sum_{i=1}^N z_i \zeta_i \right)\left(\sum_{i=1}^N z_i \zeta_i \right)^\delta\right]
\leq &\sum_{i=1}^Nz_i\E\left[ \zeta_i \left( z_i^\delta \zeta_i^\delta+ (\sum_{j\neq i}z_j \zeta_j)^\delta\right)\right]\nonumber\\
=& \sum_{i=1}^N z_i^{1+\delta} \E[\zeta_i\times\zeta_i^{\delta}]+\sum_{i=1}^N z_{i}\E\left[\zeta_i\left(\sum_{j\neq i}z_j\zeta_j\right)^\delta\right].
\end{align}
For any $i\in\{1,\cdots,N\}$, let us introduce a biased probability by 
\[
\frac{d\P^\dagger_i}{d\P}=\frac{\zeta_i}{\E[\zeta_i]}=\frac{\zeta_i}{A_i}.
\]
Then under $\P^\dagger_i$, the generating functions of $\zeta_i$ and of $\sum_{j\neq i}z_j \zeta_j$ are 
\[
\E^\dagger_i[s^{\zeta_i}]=\frac{1}{A_i}\E[\zeta_i s^\zeta_i], \textrm{ and } \E^\dagger_i[s^{\sum_{j\neq i}z_j\zeta_j}]=\frac{1}{A_i}\E\left[\zeta_i s^{\sum_{j\neq i}z_j\zeta_j}\right].
\]
By simple calculations, one sees that
\[
\E^\dagger_i[s^{\zeta_i}] =\frac{s}{(1+A_i(1-s))^2}=\E[s^{1+\zeta_i+\zeta_i^\dagger}]\textrm{ and } \E^\dagger_i[s^{\sum_{j\neq i}z_j\zeta_j}]=\E[s^{\sum_{j\neq i}z_j\zeta_j+\sum_{j\neq i}z_j\zeta_j^\dagger}]
\]
where under $\P$, $(\zeta_i^\dagger)_{1\leq i\leq N}$ is an independent copy of $(\zeta_i)_{1\leq i\leq N}$. Consequently, 
\begin{align*}
 \E[\zeta_i\times\zeta_i^{\delta}]=&A_i\E_i^\dagger[\zeta_i^\delta]=A_i\E[(1+\zeta_i+\zeta_i^\dagger)^\delta]\\
 \leq & A_i \E[1+\zeta_i+\zeta_i^\dagger]^\delta\leq A_i + 2 A_i^{1+\delta}. 
\end{align*}
Similarly, by Jensen's inequality,
\begin{align*}
\E\left[\zeta_i\left(\sum_{j\neq i}z_j\zeta_j\right)^\delta\right] =& A_i \E_i^\dagger\left[\left(\sum_{j\neq i}z_j\zeta_j\right)^\delta\right]=A_i\E\left[\left(\sum_{j\neq i}z_j\zeta_j+\sum_{j\neq i}z_j\zeta_j^\dagger\right)^\delta\right]\\
\leq & A_i\E\left[\left(\sum_{j\neq i}z_j\zeta_j+\sum_{j\neq i}z_j\zeta_j^\dagger\right)\right]^\delta \leq 2 A_i (\sum_{j\neq i}A_j z_j)^\delta
\end{align*}
Going back to \eqref{NBMom1}, we obtain that
\begin{align*}
\E\left[\left(\sum_{i=1}^N z_i \zeta_i \right)^\alpha\right]\leq &2\sum_{i=1}^N z_i^{1+\delta} (A_i+A_i^{1+\delta})+2\sum_{i=1}^N A_i z_i \left(\sum_{j\neq i}A_j z_j\right)^\delta\\
\leq & 4 \left(\sum_{i=1}^N A_i z_i\right)^{1+\delta} + 2 \sum_{i=1}^N A_i z_i^{1+\delta},
\end{align*}
as $\sum_{i\leq N} x_i^{1+\delta}\leq (\sum_{i\leq N}x_i)^{1+\delta}$ for $x_i\geq0$. 

Suppose now that for some $k\geq 2$, \eqref{NBMom} holds for any $\alpha\in[1,k]$. Let us prove \eqref{NBMom} for $1+\alpha$ with $\alpha\in[1,k]$. Similarly as above, observe that as $(x+y)^\alpha\leq 2^{\alpha-1}(x^\alpha+y^\alpha)$, 
\begin{align}\label{NBMomrec}
\E\left[\left(\sum_{i=1}^N z_i \zeta_i \right)^{1+\alpha}\right]=\E\left[\left(\sum_{i=1}^N z_i \zeta_i \right)\times \left(\sum_{i=1}^N z_i \zeta_i \right)^\alpha\right]
\leq & 2^{\alpha-1}\sum_{i=1}^N z_i \E\left[\zeta_i\times \left(z_i^\alpha\zeta_i^\alpha+\left(\sum_{j\neq i}z_j\zeta_j\right)^\alpha\right)\right]\nonumber\\
=& 2^{\alpha-1}\sum_{i=1}^N z_i A_i\E_i^\dagger\left[z_i^\alpha\zeta_i^\alpha+\left(\sum_{j\neq i}z_j\zeta_j\right)^\alpha\right]
\end{align}
where 
\begin{align*}
\E_i^\dagger\left[z_i^\alpha\zeta_i^\alpha+\left(\sum_{j\neq i}z_j\zeta_j\right)^\alpha\right]=&\E\left[z_i^\alpha(1+\zeta_i+\zeta_i^\dagger)^\alpha+\left(\sum_{j\neq i}z_j\zeta_j+\sum_{j\neq i}z_j \zeta_j\right)^\alpha\right]\\
\leq & c_{44} z_i^\alpha(1+\E[\zeta_i^\alpha])+ c_{44} \E\left[\left(\sum_{j\neq i}^N z_i \zeta_i \right)^{\alpha}\right].
\end{align*}
As \eqref{NBMom} is assumed to be true for $\alpha$, taking $z_i=0$ yields that
\[
\E\left[\left(\sum_{j\neq i}^N z_i \zeta_i \right)^{\alpha}\right]\leq C_2(\alpha)  \left[\sum_{k=0}^{\lfloor \alpha-1\rfloor}(\sum_{j\neq i}A_i z_i)^k(\sum_{j\neq i}^N A_i z_i^{\alpha-k})+(\sum_{j\neq i}^N  A_i z_i)^\alpha\right],
\]
while taking $z_j=0$ for any $j\neq i$ yields that
\[
1+\E[\zeta_i^\alpha]\leq 1+C_2(\alpha)[A_i+A_i^2+\cdots+A_i^{\lfloor\alpha\rfloor}+A_i^\alpha]\leq c_{45}(\alpha) (1+A_i^\alpha).
\]
As a result,
\begin{align*}
\E_i^\dagger\left[z_i^\alpha\zeta_i^\alpha+\left(\sum_{j\neq i}z_j\zeta_j\right)^\alpha\right]\leq c_{46}(\alpha)z_i^\alpha (1+A_i^\alpha)+ c_{46}(\alpha)\left[\sum_{k=0}^{\lfloor \alpha-1\rfloor}(\sum_{i=1}^NA_i z_i)^k(\sum_{i=1}^N A_i z_i^{\alpha-k})+(\sum_{i=1}^N  A_i z_i)^\alpha\right].
\end{align*}
Plugging it into \eqref{NBMomrec} implies that 
\begin{align*}
\E\left[\left(\sum_{i=1}^N z_i \zeta_i \right)^{1+\alpha}\right]\leq &c_{47}(\alpha)\left[ \sum_{i=1}^N z_i^{1+\alpha} A_i(1+A_i^\alpha)+ \sum_{k=0}^{\lfloor \alpha-1\rfloor}(\sum_{i=1}^NA_i z_i)^{k+1}(\sum_{i=1}^N A_i z_i^{\alpha-k})+(\sum_{i=1}^N  A_i z_i)^{1+\alpha}\right]\\
\leq & c_{48}(\alpha) \left[\sum_{k=0}^{\lfloor \alpha\rfloor}(\sum_{i=1}^NA_i z_i)^{k}(\sum_{i=1}^N A_i z_i^{\alpha+1-k})+(\sum_{i=1}^N  A_i z_i)^{1+\alpha}\right].
\end{align*}
We hence obtain \eqref{NBMom} for $1+\alpha$. By recurrence, we conclude \eqref{NBMom} for any $\alpha\geq1$.
\end{proof}

\begin{proof}[Proof of \eqref{cvgtailWsmallVP}]
Note that
\begin{align}
&e^{\kappa r_A}\E_{\Q^*}^{\eqref{tailWsmallVP}}(\epsilon, r_A)\leq \sum_{n\geq 1}\E_{\Q^*}\left[e^{V(w_n)+\kappa r_A}\ind{-r_A-\epsilon\leq V(w_n)\leq -r_A+2\epsilon,\ \min_{1\leq k\leq n-1}V(w_k)> -r_A-\epsilon}\right]\nonumber\\
&+ \sum_{n\geq1}e^{\kappa r_A}\E_{\Q^*}\Bigg[e^{V(w_n)}\ind{V(w_n)\leq -r_A-\epsilon,\ V(w_n)<\min_{1\leq k\leq n-1}V(w_k)}\ind{\sum_{j=1}^n\sum_{z\in\Omega(w_j)}e^{-V(z)}W_\infty^{(z)}\leq 3\epsilon r_A}\Bigg]\nonumber
\end{align}
The second sum of the righthand side is in fact $E_2(r)$ (see \eqref{Er1+2}), which has been treated in the the proof of Lemma \ref{BRWtailM}. We hence get that 
\[
 \sum_{n\geq1}e^{\kappa r_A}\E_{\Q^*}\Bigg[e^{V(w_n)}\ind{V(w_n)\leq -r_A-\epsilon,\ V(w_n)<\min_{1\leq k\leq n-1}V(w_k)}\ind{\sum_{j=1}^n\sum_{z\in\Omega(w_j)}e^{-V(z)}W_\infty^{(z)}\leq 3\epsilon r_A}\Bigg]=o_\epsilon(1).
\]
On the other hand, by Lemma \ref{ManytoOne} and time reversing for the random walk $(S^{(\kappa)}_k; 0\leq k\leq n)$, 
\begin{align*}
e^{\kappa r_A}\E_{\Q^*}^{\eqref{tailWsmallVP}}(\epsilon, r_A)%\leq &\sum_{n\geq 1}e^{\kappa r_A}\E_{\Q^*}\left[e^{V(w_n)}\ind{-r_A-\epsilon\leq V(w_n)\leq -r_A+2\epsilon,\ \min_{1\leq k\leq n-1}V(w_k)> -r_A-\epsilon}\right]+o_\epsilon(1)\\
\leq &\sum_{n\geq1}e^{\kappa r_A}\E\left[e^{\kappa S_n^{(\kappa)}}\ind{S^{(\kappa)}_n\in[-r_A-\epsilon, -r_A+2\epsilon], \min_{1\leq k\leq n}S^{(\kappa)}_k\geq -r_A-\epsilon}\right]+o_\epsilon(1)\\
%\leq &\sum_{n\geq 1}e^{2\kappa \epsilon}\E_{-3\epsilon}\left[\max_{1\leq k\leq n}S^{(\kappa)}_k<0, S^{(\kappa)}_n\in[-r_A-4\epsilon, -r_A-\epsilon]\right]+o_\epsilon(1)\\
\leq & \sum_{n\geq1}e^{2\kappa\epsilon}\P_{-3\epsilon}\left(\max_{1\leq k\leq n}S^{(\kappa)}_k<0, S_n\in(-r_A-5\epsilon, -r_A-\epsilon]\right)\\
\rightarrow&e^{2\kappa\epsilon}C_s^{(\kappa),-} 4\epsilon U_w^{(\kappa),+}([0,3\epsilon))+o_\epsilon(1),
\end{align*}
as $r_A\rightarrow\infty$ by \eqref{cvgrenewal}. It is then immediate to conclude \eqref{cvgtailWsmallVP}.
\end{proof}

\medskip

\appendix

\section{Appendix}
\begin{lem}\label{momentofw1}
For $\alpha\geq0,\ \beta\geq0$ such that $\alpha+\beta\leq \kappa'$, there exists some constant $c_{49}>0$ depending only on $\alpha+\beta$ such that  for any $\ell\geq1$, we have
\begin{equation}
 \widehat{\e}_\ell^*\left[\beta(w_1)^\alpha \left(\sum_{u\in\Omega(w_1)}\beta(u)\right)^{\beta}\right]\leq c_{49} \ell^{\alpha+\beta}.
\end{equation}
\end{lem}
 \begin{proof}
 Note that by change of measures and Proposition \ref{prop:changeofmeasurep}, 
 \begin{align}\label{momentofw1p}
  \ell\times\widehat{\e}_\ell^*\left[\beta(w_1)^\alpha \left(\sum_{u\in\Omega(w_1)}\beta(u)\right)^{\beta}\right]
  =&\e_\ell\left[\sum_{|u|=1}\beta(u)^{\alpha+1}\left(\sum_{|z|=1, z\neq u}\beta(z)\right)^{\beta}\right]\\
  \leq &\e_\ell\left[\left(\sum_{|u|=1}\beta(u)\right)^{\alpha+1}\left(\sum_{|z|=1}\beta(z)\right)^{\beta}\right]\nonumber\\
  =&\e_\ell\left[\left(\sum_{|u|=1}\beta(u)\right)^{\alpha+\beta+1}\right]\nonumber
 \end{align}
 Here under $\p_\ell$, $\sum_{|u|=1}\beta(u)$ is sum of $\ell$ random variables which are not independent but all distributed as  $\sum_{|u|=1}\beta(u)$ under $\p_1$. By convexity of $t\mapsto t^{\alpha+\beta+1}$, we have for any $t_i\geq0$,
 \[
 \left(\sum_{i=1}^\ell t_i\right)^{\alpha+\beta+1}\leq \ell^{\alpha+\beta}\left(\sum_{i=1}^\ell t_i^{\alpha+\beta+1}\right).
 \]
 Therefore,
 \begin{align*}
 \e_\ell\left[\left(\sum_{|u|=1}\beta(u)\right)^{\alpha+\beta+1}\right]\leq \ell^{\alpha+\beta+1}\e_1\left[\left(\sum_{|u|=1}\beta(u)\right)^{\alpha+\beta+1}\right].
 \end{align*}
 Plugging it into \eqref{momentofw1p} yields that
 \begin{equation}\label{momofw1p}
   \widehat{\e}_\ell^*\left[\beta(w_1)^\alpha \left(\sum_{u\in\Omega(w_1)}\beta(u)\right)^{\beta}\right]\leq  \ell^{\alpha+\beta}\e_1\left[\left(\sum_{|u|=1}\beta(u)\right)^{\alpha+\beta+1}\right],
 \end{equation}
 where under $\p_1^\en$, $(\beta(u))_{|u|=1}$ is of negative multinomial distribution with parameter $1$ and $(\frac{e^{-V(u)}}{1+\sum_{|v|=1}e^{-V(v)}})_{|u|=1}$. By \eqref{NBlaw}, one sees that
 \begin{align*}
 \e_1^\en\left[\left(\sum_{|u|=1}\beta(u)\right)^{\alpha+\beta+1}\right]\leq & \e_1^\en\left[|\sum_{|u|=1}\beta(u)-\sum_{|u|=1}e^{-V(u)}|^{\alpha+\beta+1}\right]+\left(\sum_{|u|=1}e^{-V(u)}\right)^{\alpha+\beta+1}\\
 \leq & c_{50}\left[\sum_{k=1}^{\lfloor\alpha+\beta+1\rfloor} \left(\sum_{|u|=1}e^{-V(u)}\right)^k+\left(\sum_{|u|=1}e^{-V(u)}\right)^{\alpha+\beta+1}\right]
 \end{align*}
 So, by Assumption \ref{cond4}, for $\alpha+\beta+1\leq \kappa$,
 \[ 
 \e_1\left[\left(\sum_{|u|=1}\beta(u)\right)^{\alpha+\beta+1}\right]\leq c_{50} \E\left[\sum_{k=1}^{\lfloor\alpha+\beta+1\rfloor} \left(\sum_{|u|=1}e^{-V(u)}\right)^k+\left(\sum_{|u|=1}e^{-V(u)}\right)^{\alpha+\beta+1}\right]=:c_{49}.
 \]
Plugging it into \eqref{momofw1p} implies that
 \[
 \widehat{\e}_\ell^*\left[\beta(w_1)^\alpha \left(\sum_{u\in\Omega(w_1)}\beta(u)\right)^{\beta}\right]\leq c_{49} \ell^{\alpha+\beta}.
 \]
 \end{proof}

\end{document}